\documentclass{amsart}
\usepackage{amssymb}
\usepackage{mathrsfs}
\usepackage{stmaryrd}
\usepackage{imakeidx}
\usepackage{hyperref}
\usepackage{enumitem}

\usepackage{tikz}
\usetikzlibrary{cd}

\usepackage{dsfont}

\makeindex

\newcommand{\Z}{\mathbb{Z}}
\newcommand{\C}{\mathbb{C}}
\newcommand{\F}{\mathbb{F}}
\newcommand{\Q}{\mathbb{Q}}
\newcommand{\Ql}{\mathbb{Q}_\ell}

\newcommand{\bk}{\Bbbk}
\newcommand{\Ga}{\mathbb{G}_{\mathrm{a}}}

\newcommand{\Mod}{\mathrm{Mod}}
\newcommand{\biMod}{\text{-}\Mod\text{-}}
\newcommand{\lMod}{\text{-}\Mod}
\newcommand{\grbiMod}{\text{-}\Mod^\Z\text{-}}
\newcommand{\grlMod}{\text{-}\Mod^\Z}
\DeclareMathOperator{\Sym}{Sym}

\newcommand{\real}{\mathrm{real}}
\DeclareMathOperator{\gr}{gr}
\newcommand{\sA}{\mathscr{A}}
\newcommand{\sT}{\mathscr{T}}
\newcommand{\Cb}{C^{\mathrm{b}}}
\newcommand{\Kb}{K^{\mathrm{b}}}
\newcommand{\mix}{{\mathrm{mix}}}

\newcommand{\Db}{D^{\mathrm{b}}}
\newcommand{\Dmix}{D^\mix}

\newcommand{\coH}{\mathsf{H}}

\newcommand{\fh}{\mathfrak{h}}
\newcommand{\uu}[1]{\underline{#1}}
\newcommand{\uv}{{\underline{v}}}
\newcommand{\uw}{{\underline{w}}}

\newcommand{\bim}{{\mathrm{bim}}}
\newcommand{\mult}{{\mathsf{m}}}

\newcommand{\cX}{\mathscr{X}}
\newcommand{\GKM}{\mathscr{G}}
\newcommand{\BKM}{\mathscr{B}}
\newcommand{\TKM}{\mathscr{T}}
\newcommand{\UKM}{\mathscr{U}}
\newcommand{\PKM}{\mathscr{P}}
\newcommand{\LKM}{\mathscr{L}}
\newcommand{\bX}{\mathbf{X}}

\makeatletter
\newcommand*\leftdash{\!\rotatebox[origin=c]{-45}{$\dabar@\dabar@\dabar@$}\!}
\newcommand*\rightdash{\!\rotatebox[origin=c]{45}{$\dabar@\dabar@\dabar@$}\!}
\makeatother
\renewcommand{\fatslash}{\!\mathord{\mathchar"2728}\;}
\renewcommand{\fatbslash}{\mathord{\mathchar"2729}}
\newcommand{\BGB}{\BKM \backslash \GKM / \BKM}
\newcommand{\UGB}{\UKM \fatbslash \GKM/\BKM}
\newcommand{\UGBold}{\UKM \backslash \GKM/\BKM}
\newcommand{\UGU}{\UKM \fatbslash \GKM\fatslash \UKM}
\newcommand{\UGUby}{\UKM\leftdash\GKM\rightdash\UKM}
\newcommand{\UGUvee}{\UKM^\vee \fatbslash \GKM^\vee\fatslash \UKM^\vee}
\newcommand{\UGUveeby}{\UKM^\vee \leftdash \GKM^\vee\rightdash \UKM^\vee}
\newcommand{\BGBvee}{\BKM^\vee\backslash \GKM^\vee/\BKM^\vee}
\newcommand{\BGUvee}{\BKM^\vee\backslash \GKM^\vee/\UKM^\vee}
\newcommand{\PGBvee}{\PKM_J^\vee\backslash \GKM^\vee/\BKM^\vee}
\newcommand{\PGUvee}{\PKM_J^\vee\backslash \GKM^\vee/\UKM^\vee}

\newcommand{\Parity}{\mathrm{Parity}}
\newcommand{\ParityBS}{\mathrm{Parity}_{\mathrm{BS}}}
\newcommand{\ParityBSp}{\mathrm{Parity}_{\mathrm{BS}}^\oplus}
\newcommand{\cE}{\mathcal{E}}
\newcommand{\Diag}{\mathscr{D}}
\newcommand{\DiagBS}{\mathscr{D}_{\mathrm{BS}}}
\newcommand{\DiagBSp}{\mathscr{D}_{\mathrm{BS}}^\oplus}

\newcommand{\oDiagBSp}{\overline{\mathscr{D}}{}^\oplus_{\mathrm{BS}}}

\newcommand{\uDiagBSp}{\underline{\mathscr{D}}{}^\oplus_{\mathrm{BS}}}

\newcommand{\ustar}{\mathbin{\underline{\star}}}
\newcommand{\Perv}{\mathrm{Perv}}

\newcommand{\Tilt}{\mathrm{Tilt}}
\newcommand{\TiltBS}{\mathrm{\Tilt}_{\mathrm{BS}}}
\newcommand{\TiltBSp}{\mathrm{\Tilt}_{\mathrm{BS}}^\oplus}
\newcommand{\cT}{\mathcal{T}}
\newcommand{\Tmon}{\widetilde{\mathcal{T}}}
\newcommand{\hatstar}{\mathbin{\widehat{\star}}}

\newcommand{\cP}{\mathcal{P}}
\newcommand{\Pmon}{\widetilde{\mathcal{P}}}
\newcommand{\bbDelta}{\Delta\llap{$\scriptstyle\Delta$}}
\newcommand{\bbnabla}{\raisebox{2pt}{\rlap{$\scriptstyle\nabla$}}\nabla}

\newcommand{\BE}{\mathsf{BE}}
\newcommand{\RE}{\mathsf{RE}}
\newcommand{\LM}{\mathsf{LM}}
\newcommand{\FM}{\mathsf{FM}}
\newcommand{\LE}{{\mathsf{RE}^\vee}}
\newcommand{\ForBERE}{\mathsf{For}^{\BE}_{\RE}}

\newcommand{\ForLMRE}{\mathsf{For}^{\LM}_{\RE}}
\newcommand{\ForFMRE}{\mathsf{For}^{\FM}_{\RE}}
\newcommand{\ForFMLM}{\mathsf{For}^{\FM}_{\LM}}
\newcommand{\ForBELE}{\mathsf{For}^{\FM}_{\LE}}

\newcommand{\cF}{\mathcal{F}}
\newcommand{\cG}{\mathcal{G}}
\newcommand{\cH}{\mathcal{H}}
\newcommand{\cL}{\mathcal{L}}
\newcommand{\Whit}{\mathrm{Wh}}
\newcommand{\Av}{\mathsf{Av}}
\newcommand{\IC}{\mathrm{IC}}

\newsavebox\lowerdot
\savebox\lowerdot{%
\begin{tikzpicture}[scale=0.3,thick,baseline]
 \draw (0,-0.5) to (0,0.5);
 \node at (0,-0.5) {$\bullet$};
\end{tikzpicture}%
}
\newsavebox\upperdot
\savebox\upperdot{%
\begin{tikzpicture}[scale=0.3,thick,baseline]
 \draw (0,-0.5) to (0,0.5);
 \node at (0,0.5) {$\bullet$};
\end{tikzpicture}%
}
\newsavebox\upperlowerdot
\savebox\upperlowerdot{%
\begin{tikzpicture}[scale=0.3,thick,baseline]
 \draw (0,-1) to (0,-0.4);
 \draw (0,0.4) to (0,1);
 \node at (0,-0.4) {$\bullet$};
 \node at (0,0.4) {$\bullet$};
\end{tikzpicture}%
}
\newsavebox\lowerupperdot
\savebox\lowerupperdot{%
\begin{tikzpicture}[scale=0.3,thick,baseline]
 \draw (0,-0.5) to (0,0.5);
 \node at (0,-0.5) {$\bullet$};
 \node at (0,0.5) {$\bullet$};
\end{tikzpicture}%
}
\newsavebox\capmor
\savebox\capmor{%
\begin{tikzpicture}[yscale=0.1,xscale=0.1,baseline,thick] \draw[black] (-1,0) to[out=90, in=180] (0,2) to[out=0, in=90] (1,0); \end{tikzpicture}%
}
\newsavebox\cupmor
\savebox\cupmor{%
\begin{tikzpicture}[yscale=0.1,xscale=0.1,thick] \draw[black] (-1,2) to[out=90, in=180] (0,0) to[out=0, in=90] (1,2); \end{tikzpicture}%
}
\newsavebox\invymor
\savebox\invymor{%
\begin{tikzpicture}[yscale=0.2,xscale=0.1,baseline,thick] \draw (-1,-1) -- (0,0) -- (1,-1); \draw (0,0) -- (0,1); \end{tikzpicture}%
}
\newsavebox\ymor
\savebox\ymor{%
\begin{tikzpicture}[yscale=-0.2,xscale=0.1,baseline,thick] \draw (-1,-1) -- (0,0) -- (1,-1); \draw (0,0) -- (0,1); \end{tikzpicture}%
}

\newcommand{\id}{\mathrm{id}}
\newcommand{\pt}{\mathrm{pt}}
\newcommand{\simto}{\overset{\sim}{\to}}
\newcommand{\la}{\langle}
\newcommand{\ra}{\rangle}

\DeclareMathOperator{\Hom}{Hom}
\DeclareMathOperator{\End}{End}
\DeclareMathOperator{\uHom}{\underline{Hom}}

\DeclareMathOperator{\gHom}{\mathbb{H}\mathsf{om}}
\DeclareMathOperator{\gEnd}{\mathbb{E}\mathsf{nd}}

\newcommand{\V}{\mathbb{V}}
\newcommand{\W}{\mathbb{W}}

\newcommand{\JW}{{}^J \hspace{-1pt} W}

\newcommand{\Gp}{G}
\newcommand{\BGp}{B}
\newcommand{\TGp}{T}
\newcommand{\Iw}{\mathrm{Iw}}
\newcommand{\Fl}{\mathrm{Fl}}
\newcommand{\Gr}{\mathrm{Gr}}
\newcommand{\Wf}{W_{\mathrm{f}}}
\newcommand{\fW}{{}^{\mathrm{f}} W}
\newcommand{\Sf}{S_{\mathrm{f}}}
\newcommand{\IW}{\mathcal{IW}}
\newcommand{\ppn}{{}^\ell \hspace{-1pt} n}

\newcommand{\bG}{\mathbf{G}}
\newcommand{\bB}{\mathbf{B}}
\newcommand{\bT}{\mathbf{T}}
\newcommand{\bbX}{\mathbb{X}}

\renewcommand{\hat}{\widehat}
\renewcommand{\tilde}{\widetilde}

\newcommand{\hs}{{\hat{s}}}
\newcommand{\htt}{{\hat{t}}}
\newcommand{\hu}{{\hat{u}}}

\newcommand{\uH}{\underline{H}}

\makeatletter
\def\lotimes{\@ifnextchar_{\@lotimessub}{\@lotimesnosub}}
\def\@lotimessub_#1{\mathchoice{\mathbin{\mathop{\otimes}^L}_{#1}}%
  {\otimes^L_{#1}}{\otimes^L_{#1}}{\otimes^L_{#1}}}
\def\@lotimesnosub{\mathbin{\mathop{\otimes}^L}}
\makeatother

\numberwithin{equation}{section}
\numberwithin{figure}{section}
\newtheorem{thm}{Theorem}[section]
\newtheorem{lem}[thm]{Lemma}
\newtheorem{prop}[thm]{Proposition}
\newtheorem{cor}[thm]{Corollary}

\theoremstyle{definition}

\theoremstyle{remark}
\newtheorem{rmk}[thm]{Remark}

\title[Koszul duality and characters of tilting modules]{Koszul duality for Kac--Moody groups and characters of tilting modules}

\thanks{P.A. was supported by NSF Grant No.~DMS-1500890. S.R. was partially supported by ANR Grant No.~ANR-13-BS01-0001-01. This project has received funding from the European Research Council (ERC) under the European Union's Horizon 2020 research and innovation programme (grant agreement No 677147).}

\author[P. N. Achar]{Pramod N. Achar}
\address{Department of Mathematics\\
  Louisiana State University\\
  Baton Rouge, LA 70803\\
  U.S.A.}
\email{pramod@math.lsu.edu}

\author[S. Makisumi]{Shotaro Makisumi}
\address{Department of Mathematics\\
Stanford University \\
Stanford, CA \\
U.S.A.}
\email{makisumi@stanford.edu}

\author[S. Riche]{Simon Riche}
\address{Universit\'e Clermont Auvergne, CNRS, LMBP, F-63000 Clermont-Ferrand, France.
}
\email{simon.riche@uca.fr}

\author[G. Williamson]{Geordie Williamson}
\address{School of Mathematics and Statistics F07, University of
  Sydney NSW 2006, Australia. }
\email{g.williamson@sydney.edu.au}

\begin{document}

\begin{abstract}
We establish a character formula for indecomposable tilting modules for connected reductive groups in characteristic $\ell$ in
terms of $\ell$-Kazhdan--Lusztig polynomials, for $\ell > h$ the Coxeter number. Using results of Andersen,
one may deduce a character formula for simple modules if $\ell \ge 2h-2$. Our results are a consequence of an extension to modular coefficients of 
a monoidal Koszul duality equivalence established by Bezrukavnikov and Yun.
\end{abstract}

\maketitle

\section{Introduction}
\label{sec:intro}

\subsection{Overview}

Let $\mathbf{G}$ denote a connected reductive group defined over an
algebraically closed field $\bk$ of positive characteristic $\ell$ bigger than the Coxeter number $h$ of $\mathbf{G}$, and let
$\mathrm{Rep}(\mathbf{G})$ denote its category of algebraic
representations. In this paper we establish a character formula for
the indecomposable tilting modules in the principal block $\mathrm{Rep}_0(\mathbf{G})$ of $\mathrm{Rep}(\mathbf{G})$ (which, by classical work, implies in theory a character formula for any tilting module in $\mathrm{Rep}(\mathbf{G})$). The answer
is given in terms of the $\ell$-Kazhdan--Lusztig polynomials of the affine Hecke
algebra of the dual root system, and confirms a conjecture of the last
two authors ~\cite{rw}. Thanks to an observation of Andersen, our
results also imply a formula for the characters of the simple modules
of $\mathbf{G}$ if $\ell \geq 2h-2$.

The problem of determining the simple characters of $\mathbf{G}$ has a rich
history. Following important early calculations of Jantzen in ranks${}\le 3$, Lusztig proposed a conjecture under the assumption that
$\ell$ is larger than the Coxeter number \cite{lusztig-conjecture}. Lusztig's conjecture was
established for sufficiently large $\ell$ \cite{AJS, kl-quantum,LUSMon,kt} and
subsequently for $\ell$ larger than an explicit 
enormous bound \cite{F}. On the other hand, ideas of Soergel,
Elias, He, and the fourth author led to a uniform construction of many
counterexamples \cite{soergel-poschar,ew,HeW,williamson-explosion}. These
counterexamples involve primes $\ell$ which grow exponentially
in the Coxeter number. 

The question of tilting characters is even more mysterious. Despite
the central importance of tilting modules in the modular representation theory of $\mathbf{G}$ and
related groups (e.g. symmetric groups), their characters appear
extremely difficult to determine: at present there
is a complete understanding only for tori (where the
problem is trivial) and $\mathbf{G} = \mathrm{SL}_2$.
 The case of a quantum group at a
root of unity was settled in work of Soergel~\cite{soergel-char-for,
  soergel-char-tilt}, and a conjecture of Andersen would imply that
these characters determine the modular tilting characters for weights
in the lowest $\ell^2$-alcove. However, for tilting modules (in
contrast to simple modules), there is no finite set of weights which
determines the answer in general.

Until the present series of works, all known or conjectured character
formulas for algebraic groups or quantum groups involved
some sort of Kazhdan--Lusztig polynomials. These polynomials admit a
combinatorial definition (involving only the affine Weyl group, viewed
as a Coxeter group), but also have a geometric meaning as the
graded dimensions of the stalks of intersection cohomology
complexes. The character formula proved in the current work instead involves
$\ell$-Kazhdan--Lusztig polynomials. These polynomials may be computed
algorithmically via diagrammatic algebra, and also have a geometric
meaning as the graded dimensions of the stalks of the $\ell$-parity
sheaves. It is important to note, however, that the algorithm to
calculate the $\ell$-Kazhdan--Lusztig polynomials is much more involved
than the original Kazhdan--Lusztig algorithm. On the other hand, the
formulas involving $\ell$-Kazhdan--Lusztig polynomials hold as soon as
$\ell$ is larger than the Coxeter 
number.\footnote{In fact, we expect a form of these formulas to hold for all
  $\ell$. See \cite[Conjecture 1.7]{rw} and \cite{el} where
  this conjecture is proved for the general linear group.} Thus
``independence of $\ell$'' and the Lusztig conjecture hold as soon as one has
agreement between $\ell$-Kazhdan--Lusztig polynomials and their
classical counterparts. (When this agreement occurs remains, however,
an important open question.)

The proof of our main result relies on a body of recent
work~\cite{prinblock,arider,mr:etsps,rw} establishing links between
representations of reductive groups and the geometry of affine flag
varieties.  This earlier work, summarized in
Figure~\ref{fig:modular-new} and discussed in~\S\ref{ss:intro-application-RT} below, had suggested that the character formula
for tilting modules would follow from a suitable kind of ``monoidal
modular Koszul duality'' for Hecke categories of parity sheaves
associated to affine flag varieties.  (An important antecedent for
these ideas is work of Bezrukavnikov--Yun \cite{by}, which establishes such an
equivalence with coefficients of characteristic $0$.) The authors'
previous paper~\cite{amrw} made it possible to formulate the monoidal
Koszul duality conjecture precisely.  In the present paper, we prove
the monoidal Koszul duality theorem, and we deduce our tilting
character formula as a consequence.

A striking aspect of monoidal Koszul duality is that the Hecke
category attached to a Kac--Moody group and to its Langlands dual are (in
a sense made precise by Theorem \ref{thm:monoidal-intro} below) formal
consequences of one another. In other words, the Hecke category already ``knows''
the Hecke category of its Langlands dual group. One can view this
result as analogous to the geometric Satake equivalence: any
complex reductive group already ``knows'' the category of
representations of its dual group. We expect this Langlands duality
for Hecke categories to have other applications in modular representation theory.

In the remainder of the introduction, we review what Koszul duality means for flag varieties, and what role it has played in representation theory.  We will give a precise statement of monoidal Koszul duality, and we will discuss characteristic-$0$ antecedents to our results.

\begin{figure}
\[
\begin{tikzcd}[row sep=large, column sep=small]
&&& \Db\Perv^\mix_{(\Iw)}(\Gr, \bk) 
  \ar[dlll, rightarrow, bend right=20, "\substack{\text{graded Finkelberg--}\\ \text{Mirkovi\'c conjecture~\cite{prinblock}}}" description]
  \ar[dd, leftrightarrow, bend left=20, dashed, "\substack{\text{parabolic--}\\ \text{Whittaker}\\ \text{duality}}"{name=P, description}]
& \Tilt(\UGU,\bk)
  \ar[dd, leftrightarrow, bend left=20, dashed, "\substack{\text{monoidal}\\ \text{Koszul}\\ \text{duality}}"{name=M, description}]
  \ar[l] \\
\Db\mathrm{Rep}_0(\mathbf{G})
  \ar[rr, leftarrow, "\text{\cite{prinblock}}" description] & \hspace{2em} &
\Db\mathrm{Coh}\rlap{$^{G \times \mathbb{G}_{\mathrm{m}}}(\widetilde{\mathcal{N}})$}
  \ar[ur, leftrightarrow, "\text{\cite{arider,mr:etsps}}" description] \\
&&& \Db\Perv^\mix_{\mathcal{IW}}(\Fl, \bk)
  \ar[ulll, rightarrow, dashed, bend left=20, "\substack{\text{tilting character formula}\\ \text{(conjectured in~\cite{rw})}}" description]
& \Parity(\BGB,\bk)
  \ar[l]
  \ar[Rightarrow, from=M, to=P]
\end{tikzcd}
\]
\caption{Reductive groups and Koszul duality}\label{fig:modular-new}
\end{figure}
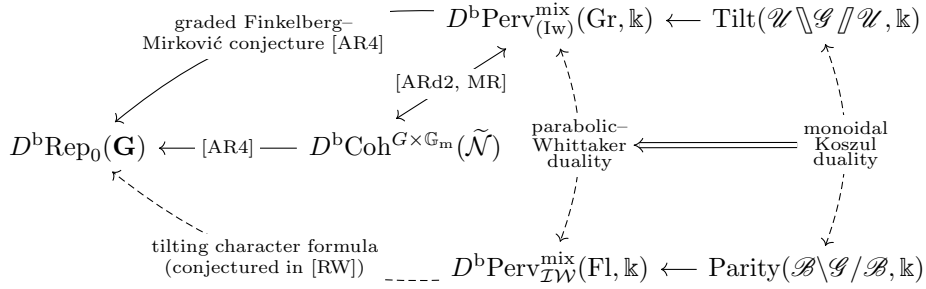

\subsection{Koszul duality for flag varieties of reductive groups}
\label{ss:bgs-duality-new}

Let $G$ be a complex semisimple algebraic group, let $B \subset G$ be a Borel subgroup, and let $T \subset B$ be a maximal torus. Let $\Db_{(B)}(G/B, \C)$ be the derived category of complexes of $\C$-sheaves on $G/B$ which are constructible with respect to the stratification by $B$-orbits (called the \emph{Bruhat stratification}), and let $\Perv_{(B)}(G/B, \C)$ be the heart of the perverse t-structure on this category.  Let $G^\vee$ be the Langlands dual group. In general, we use a superscript ``$^\vee$'' to indicate objects attached to $G^\vee$: for example, $T^\vee$, $\Perv_{(B^\vee)}(G^\vee/B^\vee,\C)$, etc.

Koszul duality for $G/B$ was first introduced by Be{\u\i}linson--Ginzburg--Soergel in~\cite{bgs}, motivated by two related ideas:
\begin{enumerate}
\item
the desire to explain the Kazhdan--Lusztig inversion formula for Kazhdan--Lusztig polynomials in categorical terms;
\item
the desire to relate two different geometric approaches to the study of the category $\mathcal{O}$ of the Lie algebra of $G$: one which originates in the Be{\u\i}linson--Bernstein localization theory~\cite{beilinson-bernstein} and leads to an equivalence of categories between a regular block of $\mathcal{O}$ and $\Perv_{(B)}(G/B, \C)$, as in~\cite[Proposition~3.5.2]{bgs}; and one due to Soergel which relates projective objects in a regular block of $\mathcal{O}$ with semisimple complexes (i.e.~direct sums of shifted simple perverse sheaves) in $\Db_{(B^\vee)}(G^\vee/B^\vee,\C)$, as in~\cite{soergel-kat}.
\end{enumerate}

The statement of Koszul duality in~\cite{bgs} involves a new category, denoted by $\Perv^\mix_{(B)}(G/B, \C)$, that serves as a ``graded version'' of $\Perv_{(B)}(G/B, \C)$.  (It is defined in terms of Deligne's mixed sheaves on an $\F_p$-version of the flag variety; see~\cite[\S 1.2]{amrw} for a more precise discussion.)  For each $w \in W$, there are four notable objects supported on the closure of $BwB/B$: denote by
\[
\IC^\mix_w, \qquad
\Delta^\mix_w, \qquad
\nabla^\mix_w, \qquad
\mathrm{T}^\mix_w
\]
the simple, standard, costandard, and indecomposable tilting objects, respectively, normalized so that their restrictions to $BwB/B$ have weight $0$.

Let us set $\Dmix_{(B)}(G/B,\C) := \Db \Perv^\mix_{(B)}(G/B, \C)$.  The construction of~\cite{bgs} provides an equivalence of categories\footnote{To be precise, the functor we call $\varkappa$ is actually the composition of the functor constructed in~\cite{bgs} with the Radon transform of~\cite{bbm,yun} (see also~\cite{bg}).  For a discussion of various versions of Koszul duality, see~\cite[Chapter~1]{amrw}.}
\begin{equation}\label{eqn:bgs-duality-intro}
\varkappa : \Dmix_{(B)}(G/B, \C) \simto \Dmix_{(B^\vee)}(G^\vee/B^\vee, \C)
\end{equation}
that satisfies $\varkappa \circ \langle 1 \rangle \cong [1] \langle -1 \rangle \circ \varkappa$, where $\langle 1 \rangle$ is the inverse of a square root of the Tate twist.  It also satisfies
\begin{align*}
\varkappa(\IC^\mix_w) &\cong \mathrm{T}^{\vee, \mix}_{w^{-1}}, & \varkappa(\Delta^\mix_w) &\cong \Delta^{\vee,\mix}_{w^{-1}} \\
\varkappa(\mathrm{T}^\mix_w) &\cong \IC^{\vee, \mix}_{w^{-1}}, & \varkappa(\nabla^\mix_w) &\cong \nabla^{\vee,\mix}_{w^{-1}}.
\end{align*}
The Kazhdan--Lusztig inversion formula can be understood as a ``combinatorial shadow'' of this equivalence.

\subsection{The Kac--Moody case and quantum groups}
\label{ss:kac-moody-intro}

These ideas were later generalized by Bezrukavnikov--Yun~\cite{by} to the case where $G$ is replaced by a general Kac--Moody group $\GKM$.  Let $\BKM \subset \GKM$ be a Borel subgroup, and let $\UKM \subset \BKM$ be its unipotent radical.  An important new idea in~\cite{by} (also suggested in~\cite{bg}) is that a richer version of Koszul duality can be obtained if one ``deforms'' the categories of semisimple complexes on $\GKM/\BKM$ and tilting perverse sheaves on $\GKM^\vee/\BKM^\vee$ along a polynomial ring. The $\BKM$-constructible semisimple complexes are thus replaced by the \emph{$\BKM$-equivariant} semisimple complexes, and the tilting perverse sheaves are replaced by the so-called ``free-monodromic'' objects constructed (via a very technical procedure) by Yun using certain pro-objects in the derived category of $\GKM^\vee/\UKM^\vee$, see~\cite[Appendix~A]{by}.  These deformed categories each have a monoidal structure, given by an appropriate kind of convolution product.  The main result of~\cite{by} is an equivalence of monoidal categories
\begin{equation}\label{eqn:by-duality-intro}
\tilde\varkappa: \mathrm{Semis}(\BGB,\C) \simto \Tilt(\UGUveeby,\C)
\end{equation}
relating $\BKM$-equivariant semisimple complexes on $\GKM/\BKM$ and free-monodromic tilting perverse sheaves attached to $\GKM^\vee$. 
From this, Bezrukavnikov--Yun then deduce a Kac--Moody analogue of~\eqref{eqn:bgs-duality-intro}.

As in~\S\ref{ss:bgs-duality-new}, this result has a combinatorial motivation in terms of Kazhdan--Lusztig polynomials~\cite{yun}, and a representation-theoretic motivation in terms of analogues of the category $\mathcal{O}$ for Kac--Moody Lie algebras.

But a third motivation for the work in~\cite{by}, specifically in the case of \emph{affine} Kac--Moody groups, came from the hope of uniting two geometric approaches to the study of representations of Lusztig's quantum groups at a root of unity (see e.g.~\cite[\S 1.2]{bez:ctm}), which we review below.  Let $\mathrm{Rep}_0(\mathsf{U}_\zeta)$ denote the principal block of the category of finite-dimensional representations of Lusztig's quantum group $\mathsf{U}_\zeta$ associated with an adjoint semisimple complex algebraic group $G$, specialized at a root of unity $\zeta$.

The first approach comes from~\cite{abg}.  The main result of~\cite[Part~I]{abg} relates\footnote{We will not try to make the meaning of ``relates'' precise; this involves technical difficulties which are irrelevant for our present purposes.} $\mathrm{Rep}_0(\mathsf{U}_\zeta)$ to the derived category of equivariant coherent sheaves on the Springer resolution $\widetilde{\mathcal{N}}$ of $G$, denoted by $\Db\mathrm{Coh}^{G \times \mathbb{G}_{\mathrm{m}}}(\widetilde{\mathcal{N}})$. Then the main result of~\cite[Part~II]{abg} states that $\Db\mathrm{Coh}^{G \times \mathbb{G}_{\mathrm{m}}}(\widetilde{\mathcal{N}})$ is equivalent to the derived category of Iwahori-constructible perverse sheaves on the affine Grassmannian $\Gr$ of the Langlands dual semisimple group $G^\vee$. Together, these results give a new proof of Lusztig's character formula for simple modules in $\mathrm{Rep}_0(\mathsf{U}_\zeta)$.  (This character formula was already known when~\cite{abg} appeared, by combining work of Kazhdan--Lusztig~\cite{kl-quantum}, Lusztig~\cite{LUSMon} and Kashiwara--Tanisaki~\cite{kt}.)

The second approach comes from~\cite{ab}, whose main result gives an equivalence between $\Db\mathrm{Coh}^{G \times \mathbb{G}_{\mathrm{m}}}(\widetilde{\mathcal{N}})$ and a certain category of Iwahori--Whittaker\footnote{See~\cite{ab} or~\S\ref{ss:duality-Gr} below for the meaning of this term.} sheaves on the affine flag variety $\Fl$ of $G^\vee$.  The composition of this equivalence with~\cite[Part~I]{abg} matches simple Iwahori--Whittaker perverse sheaves on $\Fl$ with \emph{tilting} (rather than simple) modules in $\mathrm{Rep}_0(\mathsf{U}_\zeta)$.  This leads to a new proof of a character formula for tilting modules, previously obtained by Soergel~\cite{soergel-char-tilt, soergel-char-for}. (See~\cite{jantzen} for more details on these questions.)

The two approaches to $\mathrm{Rep}_0(\mathsf{U}_\zeta)$ described above are summarized in (the left half of) Figure~\ref{fig:quantum-new}.  From this diagram, one might speculate that there is an equivalence relating $\Db\Perv^\mix_{(\Iw)}(\Gr, \C)$ to $\Db\Perv^\mix_{\mathcal{IW}}(\Fl, \C)$ that sends tilting perverse sheaves to simple ones, and vice versa.  This is achieved in~\cite{by}, where the desired equivalence, a form of ``parabolic Koszul duality,'' is deduced from~\eqref{eqn:by-duality-intro} in the case where $\GKM$ is the affine Kac--Moody group associated to $G^\vee$. (In this case, one can use the same group $\GKM$ on both sides of~\eqref{eqn:by-duality-intro} because of symmetrizability.)

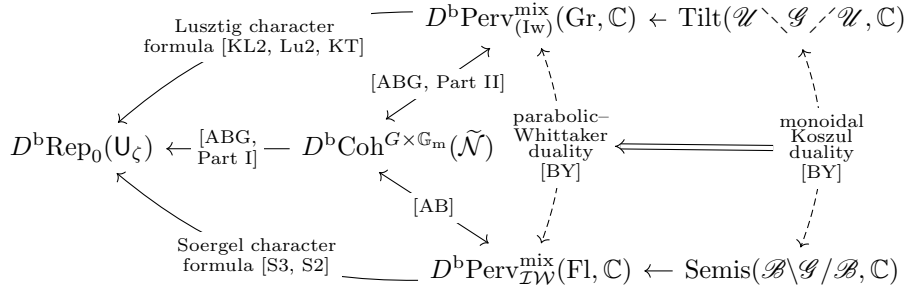
\begin{figure}
\[
\begin{tikzcd}[row sep=large, column sep=0.675em]
&&& \Db\Perv^\mix_{(\Iw)}(\Gr, \C) 
  \ar[dlll, rightarrow, bend right=20, "\substack{\text{Lusztig character}\\ \text{formula~\cite{kl-quantum,LUSMon,kt}}}" description]
  \ar[dd, leftrightarrow, dashed, bend left=20, "\substack{\text{parabolic--}\\ \text{Whittaker}\\ \text{duality}\\ \text{\cite{by}}}"{name=P, description}]
& \Tilt(\UGUby,\C)
  \ar[dd, leftrightarrow, dashed, bend left=20, "\substack{\text{monoidal}\\ \text{Koszul}\\ \text{duality}\\ \text{\cite{by}}}"{name=M, description}]
  \ar[l] \\
\Db\mathrm{Rep}_0(\mathsf{U}_\zeta)
  \ar[rr, leftarrow, "\substack{\text{[ABG,}\\ \text{Part I]}}" description] & \hspace{2.5em} &
\Db\mathrm{Coh}\rlap{$^{G \times \mathbb{G}_{\mathrm{m}}}(\widetilde{\mathcal{N}})$}
  \ar[ur, leftrightarrow, "\text{\cite[Part II]{abg}}" description] 
  \ar[dr, leftrightarrow, "\text{\cite{ab}}" description] \\
&&& \Db\Perv^\mix_{\mathcal{IW}}(\Fl, \C)
  \ar[ulll, rightarrow, bend left=20, "\substack{\text{Soergel character}\\ \text{formula~\cite{soergel-char-tilt,soergel-char-for}}}" description]
& \mathrm{Semis}(\BGB,\C)
  \ar[l]
  \ar[Rightarrow, from=M, to=P]
\end{tikzcd}
\]
\caption{Quantum groups and Koszul duality}\label{fig:quantum-new}
\end{figure}

\subsection{The modular case}

The main geometric result of the present paper is an analogue of~\eqref{eqn:by-duality-intro} in the case when the sheaves under consideration have coefficients in a field of arbitrary characteristic.  We return to the setting where $\GKM$ is an arbitrary complex Kac--Moody group, and $\BKM \subset \GKM$ is a Borel subgroup. Let $\bk$ be a field. The first difficulty when trying to generalize the constructions of~\S\S\ref{ss:bgs-duality-new}--\ref{ss:kac-moody-intro} to the setting of Bruhat-constructible $\bk$-sheaves on $\GKM/\BKM$ is to understand the appropriate definition of the category $\Perv^\mix_{(\BKM)}(\GKM/\BKM, \bk)$ of ``mixed'' perverse sheaves, as Deligne's notion of mixed perverse sheaves has no obvious analogue in this setting. This difficulty was overcome in~\cite{modrap2}, where this category was defined in terms of chain complexes over the additive category $\Parity_{(\BKM)}(\GKM/\BKM,\bk)$ of Bruhat-constructible parity complexes on $\GKM/\BKM$ (in the sense of Juteau--Mautner--Williamson~\cite{jmw}).

As explained in~\S\ref{ss:kac-moody-intro}, the starting point of the Bezrukavnikov--Yun approach is the consideration of two ``deformations'' of the category of Bruhat-constructible sheaves along a polynomial ring. The replacement of constructible sheaves by equivariant sheaves has a straightforward analogue in our setting, and leads to the monoidal category $(\Parity(\BGB,\bk), \star)$ of $\BKM$-equivariant parity complexes on $\GKM/\BKM$. The second deformation uses ``free-monodromic'' sheaves; the adaptation of this construction to our setting is much more difficult. A major hurdle is that the ``log of monodromy'' construction (central to~\cite{by}) is problematic in characteristic $p$ because of denominators. This problem was circumvented in~\cite{amrw}, where we constructed the monoidal category $(\Tilt(\UGU, \bk), \hatstar)$ of free-monodromic mixed tilting perverse sheaves on $\GKM/\BKM$.  With this notation introduced, we can state our main geometric results.

\begin{thm}
\label{thm:monoidal-intro}
There is an equivalence of monoidal categories
\[
\tilde\varkappa: \Parity(\BGB,\bk) \to \Tilt(\UGUvee,\bk).
\]
\end{thm}

By ``killing'' the deformations and passing to bounded homotopy categories, we obtain the following consequence (where we denote by $\Delta_w$, $\nabla_w$, $\cE_w$, $\cT_w$ the standard object, costandard object, indecomposable parity complex and indecomposable tilting perverse sheaves attached to $w$ respectively).

\begin{thm}
\label{thm:intro-main-new}
There is an equivalence of triangulated categories
\[
\kappa : \Dmix_{(\BKM)}(\GKM/\BKM, \bk) \simto \Dmix_{(\BKM^\vee)}(\BKM^\vee \backslash \GKM^\vee, \bk)
\]
which satisfies $\kappa \circ \langle -1 \rangle[1] \cong \la 1 \ra \circ \kappa$ and
\[
\kappa(\Delta_w) \cong \Delta^\vee_w, \quad \kappa(\nabla_w) \cong \nabla^\vee_w, \quad \kappa(\cE_w) \cong \cT^\vee_w, \quad \kappa(\cT_w) \cong \cE^\vee_w.
\]
\end{thm}

The proofs of Theorems~\ref{thm:monoidal-intro} and~\ref{thm:intro-main-new} make use of the Elias--Williamson diagrammatic category~\cite{ew} as an intermediary between the two sides. In~\cite{by}, this intermediary role was instead played by Soergel bimodules, and the proof involved the study of two functors called $\mathbb{H}$ and $\mathbb{V}$, as in the following diagram:
\[
\begin{tikzcd}[column sep=small]
\left\{ 
\begin{array}{c}
\text{equivariant parity} \\
\text{complexes on $\GKM/\BKM$}
\end{array}
\right\} \ar[r, dashed, "\mathbb{H}"]
&
\left\{
\begin{array}{c}
\text{Soergel} \\
\text{bimodules}
\end{array}
\right\}
&
\left\{ 
\begin{array}{c}
\text{free-monodromic tilting} \\
\text{sheaves on $\GKM^\vee/\BKM^\vee$}
\end{array}
\right\}.
\ar[l, dashed, "\mathbb{V}" swap]
\end{tikzcd}
\]
In our setting, since the Elias--Williamson category is defined by generators and relations, we rather reverse these arrows and consider the diagram
\[
\begin{tikzcd}[column sep=small]
\left\{ 
\begin{array}{c}
\text{equivariant parity} \\
\text{complexes on $\GKM/\BKM$}
\end{array}
\right\}
&
\left\{
\begin{array}{c}
\text{E.--W.} \\
\text{category}
\end{array}
\right\}
\ar[r, dashed]
\ar[l, dashed]
&
\left\{ 
\begin{array}{c}
\text{free-monodromic tilting} \\
\text{sheaves on $\GKM^\vee/\BKM^\vee$}
\end{array}
\right\}.
\end{tikzcd}
\]
The left arrow has already been constructed by the last two authors in~\cite{rw}; what we do here is to construct the right arrow. As in~\cite{rw}, to do this, one must say where to send generating objects and morphisms, and then one must check relations.  It is straightforward to deal with the generators.  To check relations, we reduce the question  to the case where $\GKM$ is a (finite-dimensional) reductive group and $\bk$ has characteristic $0$. This case can be studied using known properties of Soergel bimodules, along with an analogue of the functor $\mathbb{V}$.

As in~\cite{by}, there is a further generalization of Theorem~\ref{thm:intro-main-new} to the setting where on the left-hand side the flag variety $\GKM/\BKM$ is replaced by $\GKM/\PKM$ for $\PKM$ a parabolic subgroup of finite type. The right-hand side must then be replaced by an appropriate category of Whittaker-type sheaves on $\GKM^\vee/\BKM^\vee$; see Section~\ref{sec:parabolic-whittaker} for details.

\subsection{Application to representation theory}
\label{ss:intro-application-RT}

The main motivation for us to construct the modular Koszul duality equivalence in the Kac--Moody setting rather than only for reductive groups (as already obtained by the first and third authors in~\cite{modrap2}) comes from the hope of completing Figure~\ref{fig:modular-new}, with inspiration from Figure~\ref{fig:quantum-new}.

Let $\mathbf{G}$ be an adjoint semisimple group over an algebraically closed field $\bk$ of characteristic $\ell$ bigger than the Coxeter number $h$ of $\mathbf{G}$, and let $\mathrm{Rep}_0(\mathbf{G})$ be the principal block of the category of finite-dimensional algebraic representations of $\mathbf{G}$. As in~\S\ref{ss:kac-moody-intro}, there should be two geometric approaches to $\mathrm{Rep}_0(\mathbf{G})$.

The first approach was developed in~\cite{arider,mr:etsps,prinblock}.  In~\cite{prinblock}, the first and third authors constructed a functor relating $\mathrm{Rep}_0(\mathbf{G})$ to $\Db\mathrm{Coh}^{G \times \mathbb{G}_{\mathrm{m}}}(\widetilde{\mathcal{N}})$, analogous to that in~\cite[Part~I]{abg}.  When combined with earlier work with Rider~\cite{arider} and with Mautner~\cite{mr:etsps}, this leads to a functor
\[
\Perv^\mix_{(\Iw)}(\Gr, \bk) \to
\mathrm{Rep}_0(\mathbf{G}),
\]
which realizes $\Perv^\mix_{(\Iw)}(\Gr, \bk)$ as a ``graded version'' of $\mathrm{Rep}_0(\mathbf{G})$. In particular, this result reduces the problem of computing the characters of indecomposable tilting modules in $\mathrm{Rep}_0(\mathbf{G})$ to that of describing the indecomposable tilting perverse sheaves in $\Perv^\mix_{(\Iw)}(\Gr, \bk)$---but it does \emph{not} solve the problem, since no description of the latter was known at the time. (The approach developed in~\cite{yun} does not apply in the modular setting, since Yun's crucial ``condition (W)'' does not hold in this case.)

The second approach conjecturally aims to relate $\mathrm{Rep}_0(\mathbf{G})$ to Iwahori--Whittaker sheaves on $\Fl$, which provide a categorification of the \emph{antispherical module} of the affine Hecke algebra.  In~\cite{rw}, the third and fourth authors, inspired by~\cite{ab}, conjectured that characters of tilting modules in $\mathrm{Rep}_0(\mathbf{G})$ can be expressed in terms of the $\ell$-canonical basis of the antispherical module.  This conjecture was proved in~\cite{rw} in the case $\mathbf{G}=\mathrm{GL}_n(\bk)$, but by methods specific to the type-$\mathbf{A}$ situation.  The conjecture would hold in general if a modular analogue of~\cite{ab} were known, but this was not available when~\cite{rw} was written.

Recall that in Figure~\ref{fig:quantum-new}, Koszul duality provided a link between two known geometric approaches to $\mathrm{Rep}_0(\mathsf{U}_\zeta)$.  In Figure~\ref{fig:modular-new}, we turn this idea around: by combining the results of~\cite{arider,mr:etsps,prinblock} with the special case of Theorem~\ref{thm:intro-main-new} where $\GKM$ is an affine Kac--Moody group, we prove the conjecture of~\cite{rw} in general.  The precise statement appears in Theorem~\ref{thm:char-formula-tiltings}.

\subsection{Some perspectives}

The tilting character formula that we have obtained is an important result in itself, but we also believe it will lead to a better understanding of the category $\mathrm{Rep}(\mathbf{G})$, as illustrated by the following further results.

The fourth author has obtained and implemented an algorithm for explicit computations with the character formula from Theorem~\ref{thm:char-formula-tiltings}; see~\cite{jw}. This algorithm has made it possible to compute tilting characters far beyond what was previously known. It seems likely that this formula can be made more explicit, at least in certain cases; see~\cite{lw} for first results and conjectures in this direction.

In a different direction, this formula allows one to generalize Ostrik's description of tensor ideals in categories of representations of quantum groups at a root of unity~\cite{ostrik} to the setting of modular representations of reductive groups; here the proof is essentially identical, replacing the Kazhdan--Lusztig combinatorics by the $p$-Kazhdan--Lusztig combinatorics. This result provides a new tool to attack the Humphreys conjecture on support varieties of tilting $\mathbf{G}$-modules~\cite{humphreys}; see~\cite{ahr} for some progress in this direction.

\subsection{Contents}

We begin in~\S\ref{sec:prelim} with background related to the Elias--Williamson diagrammatic category, mixed perverse sheaves, and results from~\cite{amrw}.  In~\S\ref{sec:functor-V}, we define and study the functor $\V$ in the finite type case.  Next,~\S\ref{sec:Diag-Tilt} contains the construction of the functor from the Elias--Williamson category to free-monodromic tilting sheaves. In~\S\ref{sec:Koszul-duality}, we further study this functor, and we prove Theorems~\ref{thm:monoidal-intro} and~\ref{thm:intro-main-new}.  The parabolic--Whittaker variant of Koszul duality is deduced in~\S\ref{sec:parabolic-whittaker}.  Lastly, in~\S\ref{sec:application}, we complete the program described in~\S\ref{ss:intro-application-RT} to determine the tilting character formula.

\section{Preliminaries}
\label{sec:prelim}

In this section we review the main constructions of~\cite{amrw}, and quote the results we will need in the subsequent sections.

\subsection{The Elias--Williamson diagrammatic category}

Let $(W,S)$ be a Coxeter system with $S$ finite, and let $\bk$ be an
integral domain. A finite sequence of elements of $S$ will be called
an \emph{expression}. A \emph{realization} of $(W,S)$ over $\bk$ is a
triple $\fh = (V, \{\alpha_s^\vee\}_{s \in S}, \{\alpha_s\}_{s \in
  S})$ where $V$ is a finitely generated free $\bk$-module, and the
subsets $\{\alpha_s^\vee\}_{s \in S} \subset V$, $\{\alpha_s\}_{s
\in S} \subset V^* := \Hom_\bk(V,\bk)$ of ``simple coroots'' and
``simple roots'' satisfy certain conditions recalled
in~\cite[\S 2.1]{amrw}. If the realization $\fh$ satisfies further
technical conditions (it is \emph{balanced} and satisfies \emph{Demazure
surjectivity}), then, following Elias--Williamson~\cite{ew}, one can associate to $(W,S)$ and $\fh$ a $\bk$-linear strict monoidal category $\DiagBS(\fh,W)$ defined by generators and relations; see~\cite[\S 2.2--2.3]{amrw}. This category carries a ``shift-of-grading'' autoequivalence, denoted by $(1)$.  For any expression $\uw$, there is a corresponding object $B_\uw$, and every object of $\DiagBS(\fh,W)$ is of the form $B_\uw(n)$ for some expression $\uw$ and some integer $n$. For any $X,Y$ in  $\DiagBS(\fh,W)$, the graded $\bk$-module $\bigoplus_{n \in \Z} \Hom_{\DiagBS(\fh,W)}(X,Y(n))$ admits a natural structure of graded bimodule over the ring
$R := \Sym(V^*)$,
where $V^*$ is in degree $2$. (This structure is obtained by adding ``polynomial boxes'' to the left or to the right of a given diagram.)

The category $\DiagBS(\fh,W)$ is not additive, and it is sometimes convenient to take its additive envelope (i.e., to formally adjoin direct sums).  The resulting category is denoted by $\DiagBSp(\fh,W)$.
If $\bk$ is a field or a complete local ring, we may also work with the Karoubian envelope of $\DiagBSp(\fh,W)$, denoted simply by $\Diag(\fh,W)$.  Up to shift, the isomorphism classes of indecomposable objects in $\Diag(\fh,W)$ are in bijection with $W$~\cite{ew}.  In particular, for each $w \in W$, there is a corresponding indecomposable object denoted by $B_w$.

In this paper we will only consider a certain family of Coxeter groups and realizations that we call \emph{Cartan realizations of crystallographic Coxeter groups}, and which arise in the following way.
Let $A$ be a generalized Cartan matrix with rows and columns
parametrized by a finite set $I$, and let $(I,\bX, \{\alpha_i\}_{i \in
  I}, \{\alpha_i^\vee \}_{i \in I})$ be an associated Kac--Moody root
datum in the sense of~\cite[\S1.2]{tits}; in other words $\bX$ is a
finitely generated free abelian group, $\{\alpha_i \}_{i \in I}
\subset \bX$, $\{\alpha_i^\vee \}_{i \in I} \subset \Hom_\Z(\bX,\Z)$ are subsets,
and $\alpha_i^\vee(\alpha_j) = a_{ij}$ for any $i,j \in I$. To $A$
one associates in a standard way a (crystallographic) Coxeter system
$(W,S)$ with $S$ in bijection with $I$; see~\cite[\S10.1]{amrw}. Then
for any integral domain $\bk$ one can define a realization $\fh_\bk =
(V, \{\alpha_s^\vee\}_{s \in S}, \{\alpha_s\}_{s \in S})$ of $(W,S)$
over $\bk$ as follows. We set $V = \bk \otimes_\Z
\Hom_\Z(\bX,\Z)$. Then for $s \in S$ we let $\alpha_s$,
resp.~$\alpha_s^\vee$, denote the image of the corresponding simple
root, resp.~coroot, in $V^*$, resp.~$V$. The realizations obtained in
this way are always balanced, but they might not satisfy Demazure
surjectivity. We remedy this in the following way. If all the maps $\alpha_s : \Hom_\Z(\bX,\Z) \to \Z$ and $\alpha_s^\vee : \bX \to \Z$ are surjective we set $\Z'=\Z$, and otherwise we set $\Z'=\Z[\frac{1}{2}]$. Then $\fh_\bk$ satisfies Demazure surjectivity provided there exists a ring morphism $\Z' \to \bk$.

\subsection{Kac--Moody groups and their flag varieties}
\label{ss:KM-groups}

From now on we assume that $\bk$ is a Noetherian integral domain of finite global dimension, and that there exists a ring morphism $\Z' \to \bk$.

The Cartan realizations of crystallographic Coxeter groups are related to geometry in the following way.
Following~\cite{mathieu,mathieu-KM}, one can associate to $A$ and the
root datum an integral Kac--Moody group $\GKM_\Z$ (a group ind-scheme
over $\Z$), together with a Borel subgroup $\BKM_\Z$ (see~\cite[\S
10.2]{amrw} for further remarks, and~\cite[\S 9.1]{rw} for an overview of the construction). Let $\UKM_\Z$ be the pro-unipotent radical of $\BKM_\Z$.  Denote by $\GKM$, $\BKM$, and $\UKM$ the base change to $\C$ of $\GKM_\Z$, $\BKM_\Z$, and $\UKM_\Z$, respectively.  Let $\cX := \GKM/\BKM$ be the flag variety, and recall that we have a Bruhat decomposition
\[
\cX = \bigsqcup_{w \in W} \cX_w,
\]
where each $\cX_w$ is a $\BKM$-orbit isomorphic to an affine space of dimension $\ell(w)$. As in~\cite{amrw}, we denote the $\BKM$-equivariant derived category of $\bk$-sheaves on $\cX$ by $\Db(\BGB,\bk)$. (By definition, the objects in this category are supported on a \emph{finite} union of $\BKM$-orbits.) As in~\cite{modrap2, amrw}, the shift functor on this category will be denoted by $\{1\}$.

To each expression $\uw$, one can associate an object $\cE_{\uw}$ of
$\Db(\BGB,\bk)$, called the \emph{Bott--Samelson parity complex}
associated to $\uw$ .  The strictly full subcategory of $\Db(\BGB,\bk)$ consisting of objects that are isomorphic to shifts of Bott--Samelson parity complexes is denoted by $\ParityBS(\BGB,\bk)$, and its additive envelope is denoted by $\ParityBSp(\BGB,\bk)$.  These are monoidal categories with respect to the convolution product $\star$.

If $\bk$ is a field or a complete local ring, we may also work with the Karoubian envelope of $\ParityBSp(\BGB,\bk)$, denoted by $\Parity(\BGB,\bk)$.  Up to shift, the isomorphism classes of indecomposable objects in $\Parity(\BGB,\bk)$ are in bijection with $W$~\cite{jmw}.  In particular, for each $w \in W$, there is a corresponding indecomposable object denoted by $\cE_w$.

By~\cite[Theorem~10.6]{rw}, there exists a canonical equivalence of monoidal categories
\begin{equation}
\label{eqn:equivalence-rw}
\Psi: \DiagBS(\fh_\bk,W) \simto \ParityBS(\BGB,\bk)
\end{equation}
that intertwines $(1)$ with $\{1\}$ and sends $B_\uw$ to $\cE_{\uw}$. This equivalence induces an equivalence
\[
\DiagBSp(\fh_\bk,W) \cong \ParityBSp(\BGB,\bk)
\]
and, if $\bk$ is a field or a complete local ring, an equivalence
\[
\Diag(\fh_\bk,W) \cong \Parity(\BGB,\bk).
\]

\subsection{Free-monodromic tilting sheaves}
\label{ss:fm-tilting-sheaves}

In~\cite[Chap.~10]{amrw}, we have defined the category of \emph{Bott--Samelson free-monodromic tilting sheaves} on $\cX$, denoted by $\TiltBS(\UGU,\bk)$.  This category is equipped with an autoequivalence $\la 1\ra$, called the \emph{Tate twist}.  For every expression $\uw$, there is a corresponding object $\Tmon_\uw^\bk$, and every object is isomorphic to $\Tmon^\bk_\uw\la n\ra$ for some expression $\uw$ and some integer $n$. (Below, the superscript ``$\bk$'' will be omitted when no confusion is likely.) The explicit construction of this category (and of the convolution bifunctor considered below) is long and quite technical, but its details will not be needed in the present paper.

By construction, the category $\TiltBS(\UGU,\bk)$ is a full subcategory in a category $\Dmix(\UGU,\bk)$, whose objects are pairs consisting of a sequence of objects of $\ParityBSp(\BGB,\bk)$ and a certain ``differential.'' If $\cF,\cG$ are objects of $\Dmix(\UGU,\bk)$, then $\Hom_{\Dmix(\UGU,\bk)}(\cF,\cG)$ is the degree-$(0,0)$ cohomology of a complex of graded $\bk$-modules denoted by $\uHom_\FM(\cF,\cG)$, whose total cohomology (a $\Z^2$-graded $\bk$-module) is denoted by $\gHom_\FM(\cF,\cG)$. The Tate twist autoequivalence $\langle 1 \rangle$ extends to $\Dmix(\UGU,\bk)$, and for any $j \in \Z$ we have $\Hom_{\Dmix(\UGU,\bk)}(\cF,\cG \langle j \rangle) = \gHom_\FM(\cF,\cG)^0_{-j}$.

Again by construction,
for $\cF,\cG$ in $\Dmix(\UGU,\bk)$, the $\Z^2$-graded $\bk$-module $\gHom_\FM(\cF,\cG)$ admits a natural right action of $R^\vee=\Sym(V)$, where $V$ is in bidegree $(0,-2)$.  This action, called the \emph{right monodromy action}, is compatible with composition: for any $f \in \gHom_\FM(\cF,\cG)$, $g \in \gHom_\FM(\cG,\cH)$, and $x \in R^\vee$, we have
\begin{equation}\label{eqn:rightmon-morph}
(g \circ f) \cdot x = (g \cdot x) \circ f = g \circ (f \cdot x).
\end{equation}
On the other hand, by~\cite[Theorem~5.2.2]{amrw}, we also have a $\Z^2$-graded algebra morphism
\[
\mu_\cF : R^\vee \to \gHom_\FM(\cF,\cF),
\]
called the \emph{left monodromy map}.  It has the property that for any $f \in \gHom_\FM(\cF,\cG)$ and any $x \in R^\vee$, we have
\begin{equation}
\label{eqn:mu-morph}
\mu_\cG(x) \circ f = f \circ \mu_\cF(x).
\end{equation}

For any Noetherian integral domain $\bk'$ of finite global dimension and any ring morphism $\bk \to \bk'$, there exists a natural functor
\[
\bk' : \Dmix(\UGU,\bk) \to \Dmix(\UGU,\bk')
\]
that commutes with Tate twists and sends $\Tmon^\bk_\uw$ to $\Tmon^{\bk'}_\uw$ for any expression $\uw$.

The additive envelope of $\TiltBS(\UGU,\bk)$ is denoted by $\TiltBSp(\UGU,\bk)$.  If $\bk$ is a field,\footnote{The same results hold if $\bk$ is a complete local ring, but this case was not treated explicitly in~\cite{amrw}.} we may also work with the Karoubian envelope of the category $\TiltBSp(\UGU,\bk)$, denoted by $\Tilt(\UGU,\bk)$.  This category is Krull--Sch\-midt, and its indecomposable objects were classified in~\cite[Theorem~10.7.1]{amrw}: up to Tate twist, they are in bijection with $W$.  In particular, for each $w \in W$, there is a corresponding indecomposable object, denoted by $\Tmon_w^\bk$.

The main result of~\cite{amrw} asserts that $\TiltBS(\UGU,\bk)$ is a monoidal category with respect to \emph{monodromic convolution}, denoted by $\hatstar$.  Of course, the category $\TiltBSp(\UGU,\bk)$ (and, when appropriate, the category $\Tilt(\UGU,\bk)$) inherit a monoidal structure as well. 
In fact, for $\cF,\cF',\cG,\cG'$ in $\TiltBS(\UGU,\bk)$, the action of $\hatstar$ on morphisms is induced by a morphism of complexes
\[
\uHom_\FM(\cF,\cG) \otimes \uHom_\FM(\cF',\cG') \to \uHom_\FM(\cF \hatstar \cF', \cG \hatstar \cG');
\]
see~\cite[\S 6.2]{amrw}.
It therefore induces a morphism
\[
\gHom_\FM(\cF,\cG) \otimes \gHom_\FM(\cF',\cG') \to \gHom_\FM(\cF \hatstar \cF', \cG \hatstar \cG').
\]
By construction, for $f \in \gHom_\FM(\cF,\cG)$, $g \in \gHom_\FM(\cF',\cG')$, and $x \in R^\vee$, we have
\begin{equation}
\label{eqn:morph-Rvee-left-right}
(f\cdot x) \hatstar g = f \hatstar (\mu_{\cG'}(x) \circ g).
\end{equation}
Finally, by construction again, for any expressions $\uv, \uw$ we have
\begin{equation}
\label{eqn:Tmon-conv}
\Tmon_{\uv} \hatstar \Tmon_\uw \cong \Tmon_{\uv\uw},
\end{equation}
where $\uv\uw$ means the concatenation of $\uv$ and $\uw$.

\subsection{The constructible derived category}
\label{ss:constructible-Dmix}

In~\cite{amrw}, in addition to the categories defined above, we considered two other categories of sheaves on $\cX$: the \emph{left-monodromic category}, denoted by $\Dmix(\UGB,\bk)$, and the \emph{right-equivariant category}, denoted by $\Dmix(\UGBold,\bk)$.  These categories are related by various functors as shown below:
\[
\begin{tikzcd}[column sep=0pt]
\hspace{-2em} \TiltBSp(\UGU,\bk), \hatstar \hspace{-2em}\ar[dr, "\ForFMLM" description] &&&&
\hspace{-2em}\ParityBSp(\BGB,\bk), \star \hspace{-2em}\ar[dl, "\ForBERE" description] \\
& \hspace{-2em}\Dmix(\UGB,\bk) \ar[rr, "\sim", "\ForLMRE"'] & \hspace{1em} &
\Dmix(\UGBold,\bk)\hspace{-2em}
\end{tikzcd}
\]
Here $\Dmix(\UGB,\bk)$ and $\Dmix(\UGBold,\bk)$ admit natural structures of triangulated categories, and the functor $\ForLMRE$ is an equivalence of triangulated categories by~\cite[Theorem~4.6.2]{amrw}.
By construction, the category $\Dmix(\UGBold, \bk)$ is canonically equivalent (as a triangulated category) to $\Kb\ParityBSp(\UGBold, \bk)$, where $\ParityBSp(\UGBold, \bk)$ is defined as for $\ParityBSp(\BGB, \bk)$, but using the $\UKM$-equivariant derived category of $\cX$ instead of its $\BKM$-equivariant derived category. Therefore, if $\bk$ is a field or a complete local ring, this category is equivalent to $\Kb\Parity(\UGBold, \bk)$, i.e.~to the category denoted $\Dmix_{(\BKM)}(\cX, \bk)$ in~\cite{modrap2} (see~\cite[\S 4.9 and \S 10.4]{amrw}); in particular any object of $\Parity(\UGBold, \bk)$ can be naturally considered as an object of $\Dmix(\UGBold, \bk)$.

As in the category $\Dmix(\UGU,\bk)$, for $\cF,\cG$ in $\Dmix(\UGB, \bk)$, the $\bk$-module $\Hom_{\Dmix(\UGB,\bk)}(\cF,\cG)$ is defined as the degree-$(0,0)$ cohomology of a complex of graded $\bk$-modules denoted $\uHom_\LM(\cF,\cG)$. The total cohomology of this complex is denoted $\gHom_\LM(\cF,\cG)$; then for $i,j \in \Z$ we have
\begin{equation}
\label{eqn:gHom-Hom-LM}
\gHom_\LM(\cF,\cG)^i_j \cong \Hom_{\Dmix(\UGB,\bk)}(\cF, \cG[i] \langle -j \rangle),
\end{equation}
see~\cite[Remark~4.5.2]{amrw}. For $\cF,\cG$ in $\Dmix(\UGU, \bk)$, the action of the functor $\ForFMLM$ is induced by a morphism of complexes
\[
\uHom_\FM(\cF,\cG) \to \uHom_\LM(\ForFMLM(\cF), \ForFMLM(\cG)),
\]
see~\cite[\S 5.1]{amrw}.

The Tate twist and extension-of-scalars functors are also defined for the categories $\Dmix(\UGB,\bk)$ and $\Dmix(\UGBold,\bk)$, and commute with the forgetful functors. For any expression $\uw$ we set
\[
\cT^\bk_\uw := \ForFMLM(\Tmon^\bk_\uw).
\]
(In this setting also, the superscript ``$\bk$'' will be omitted when no confusion is likely.)

For the following result, see~\cite[Corollary~10.6.2]{amrw}.

\begin{prop}
\label{prop:morph-Tilt-UGU}
For any expressions $\uv,\uw$ and any $i,j \in \Z$, we have
\[
\gHom_\FM(\Tmon^\bk_\uv, \Tmon^\bk_\uw)^i_j = 0 \quad \text{unless $i=0$.}
\]
Moreover, $\gHom_\FM(\Tmon^\bk_\uv, \Tmon^\bk_\uw)^0_\bullet$ is graded free as a right $R^\vee$-module, and the morphism
\[
\gHom_\FM(\Tmon^\bk_\uv, \Tmon^\bk_\uw)^0_\bullet \otimes_{R^\vee} \bk \to \gHom_\LM(\cT^\bk_\uv, \cT^\bk_\uw)^0_\bullet
\]
induced by the functor $\ForFMLM$ is an isomorphism. Finally, for any Noetherian integral domain $\bk'$ of finite global dimension and any ring morphism $\bk \to \bk'$, the functor $\bk'$ induces an isomorphism
\[
\bk' \otimes_\bk \Hom_{\Dmix(\UGU,\bk)}(\Tmon^\bk_\uv, \Tmon^\bk_\uw \langle j \rangle) \simto \Hom_{\Dmix(\UGU,\bk')}(\Tmon^{\bk'}_\uv, \Tmon^{\bk'}_\uw \langle j \rangle)
\]
for any $j \in \Z$.
\end{prop}

In the case when $\bk$ is a field, we have also defined a subcategory
\[
\Tilt(\UGB,\bk) \subset \Dmix(\UGB,\bk)
\]
in~\cite[\S 10.5]{amrw}. By~\cite[Theorem~11.4.2]{amrw}, this category admits a natural action of the monoidal category $\Tilt(\UGU,\bk)$; the corresponding bifunctor will also be denoted $\hatstar$. By~\cite[(6.18)]{amrw}, for $\cF,\cG$ in $\Tilt(\UGU,\bk)$, we have
\begin{equation}
\label{eqn:For-hatstar}
\ForFMLM(\cF \hatstar \cG) \cong \cF \hatstar \ForFMLM(\cG).
\end{equation}
The indecomposable objects in this category were classified in~\cite[Corollary~10.5.5]{amrw}: up to Tate twist, they are in bijection with $W$.  In particular, for each $w \in W$, there is a corresponding indecomposable object, denoted by $\cT^\bk_w$. Moreover we have $\cT^\bk_w \cong \ForFMLM(\Tmon^\bk_w)$.

\subsection{Realization functors}
\label{ss:realization}

In this subsection, we review a (variant of a) construction due to Be{\u\i}linson~\cite[Appendix]{beilinson}.  A triangulated category $\sT$ is said to \emph{admit a filtered version} if there exists a filtered triangulated category $\tilde\sT$ over $\sT$, in the sense of~\cite[Definition~A.1]{beilinson}. An additive subcategory $\sA \subset \sT$ is said to \emph{have no negative self-Exts} if $\Hom_\sT(M,N[n]) = 0$ for all $M, N \in \sA$ and all $n < 0$.

The following is a variant of the main result of~\cite[Appendix]{beilinson}.

\begin{prop}
\label{prop:realization}
Let $\sT$ be a triangulated category that admits a filtered version, and let $\sA \subset \sT$ be a full additive category with no negative self-Exts.  There is a functor of triangulated categories
\[
\real : \Kb\sA \to \sT
\]
whose restriction to $\sA$ is the inclusion functor.  In addition, if $\sA$ is the heart of a t-structure (and hence an abelian category), this functor factors through a functor
\[
\real: \Db\sA \to \sT.
\]
\end{prop}
In~\cite{beilinson}, this result is only stated in the case where $\sA$ is the heart of a t-structure. For details in a more general setting, see~\cite[\S3]{rider}.
\begin{proof}[Sketch of proof]
The filtered category $\tilde\sT$ comes with functors $\gr_i: \tilde\sT \to \sT$ for each $i$.  Let $\tilde\sA \subset \tilde\sT$ be the full subcategory consisting of objects $M$ such that $\gr_i M = 0$ for all but finitely many $i$, and such that $\gr_i M \in \sA[-i]$ for all $i \in \Z$.  An argument similar to~\cite[Proposition~A.5]{beilinson} shows that $\tilde\sA \cong \Cb\sA$.  The forgetful functor $\tilde\sT \to \sT$ induces an additive functor $\Cb\sA \to \sT$, which then factors through $\Kb\sA$ or, if $\sA$ is the heart of a t-structure, through $\Db\sA$.
\end{proof}

The following statement is a variant of~\cite[Lemma~A.7.1]{beilinson}.  We omit its proof.

\begin{prop}
\label{prop:realization-functor}
Let $\sT_1$ and $\sT_2$ be two triangulated categories admitting a filtered version, and let $\sA_1 \subset \sT_1$, $\sA_2 \subset \sT_2$ be two additive categories with no negative self-Exts. Let $F: \sT_1 \to \sT_2$ be a triangulated functor that restricts to an additive functor $F_0: \sA_1 \to \sA_2$.  If $F$ lifts to a functor of filtered triangulated categories $\tilde F: \tilde\sT_1 \to \tilde\sT_2$, then the following diagram commutes up to natural isomorphism:
\[
\begin{tikzcd}
\Kb\sA_1 \ar[r, "\real"] \ar[d, "\Kb(F_0)"'] & \sT_1 \ar[d, "F"] \\
\Kb\sA_2 \ar[r, "\real"] & \sT_2.
\end{tikzcd}
\]
\end{prop}

In this paper, we will mainly use these constructions in the case where $\sT = \Dmix(\UGB,\bk)$ or $\Dmix(\UGBold,\bk)$.  These two are equivalent (via~$\ForLMRE$), and the latter, as the homotopy category of an additive category, admits a filtered version by the construction of~\cite[\S2.5]{ar:kdsf}.  Here is an application of this theory.

\begin{lem}\label{lem:tilt-realization}
There is an equivalence of triangulated categories
\[
\real: \Kb\TiltBSp(\UGB,\bk) \to \Dmix(\UGB,\bk).
\]
\end{lem}
\begin{proof}
By~\cite[Proposition~10.6.1]{amrw}, for all $\cF, \cG \in \TiltBSp(\UGB,\bk)$, we have $\Hom_{\Dmix(\UGB,\bk)}(\cF,\cG[n]) = 0$ for all $n \ne 0$.  It follows from this that the realization functor exists and is fully faithful.  A routine support argument shows that the image of this functor generates $\Dmix(\UGB,\bk)$, so it is essentially surjective as well.
\end{proof}

\subsection{The perverse t-structure}
\label{ss:perv}

In this subsection we assume that $\bk$ is a field.

As recalled in~\cite[\S 10.5]{amrw}, the category $\Dmix(\UGBold,\bk)$ admits a natural ``perverse'' t-structure, constructed in~\cite{modrap2}. We will denote by
\[
\Perv(\UGB,\bk) \subset \Dmix(\UGB,\bk)
\]
the inverse image under the equivalence $\ForLMRE$ of the heart of this t-structure. This category is stable under the Tate twist, and has a natural structure of graded highest weight category with weight poset $W$ (for the Bruhat order). We will denote by $\bbDelta_w$ and $\bbnabla_w$ the corresponding standard and costandard objects.
By~\cite[Proposition~10.5.1]{amrw}, the category of tilting objects in $\Perv(\UGB,\bk)$ identifies with the subcategory $\Tilt(\UGB, \bk)$ considered above. From this it follows that the natural functors
\begin{equation}
\label{eqn:realization-Dmix}
\Kb \Tilt(\UGB, \bk) \to \Db \Perv(\UGB,\bk) \to \Dmix(\UGB,\bk)
\end{equation}
are equivalences of triangulated categories, using, say,~\cite[Lemma~A.5]{modrap2}.

As a special case of~\cite[Proposition~7.6.3]{amrw}, for any $s \in S$ there exists a triangulated functor
\[
C_s : \Dmix(\UGB,\bk) \to \Dmix(\UGB,\bk)
\]
whose restriction to $\Tilt(\UGB, \bk)$ is isomorphic to the functor $\Tmon_s \hatstar (-)$.

\begin{lem}
\label{lem:Cs-exact}
The functor $C_s$ is exact for the perverse t-structure.
\end{lem}

\begin{proof}
By~\cite[Proposition~3.4]{modrap2} the nonnegative part, resp.~the nonpositive part, of the perverse t-structure is generated under extensions by the objects of the form $\bbnabla_w \langle n \rangle [m]$ with $w \in W$, $n \in \Z$ and $m \in \Z_{\leq 0}$, resp.~by the objects of the form $\bbDelta_w \langle n \rangle [m]$ with $w \in W$, $n \in \Z$ and $m \in \Z_{\geq 0}$. With this in mind, the claim follows from~\cite[Lemma~10.5.3]{amrw}.
\end{proof}

Following~\cite[\S3.1]{modrap2}, we denote by $\IC^\mix_w$ the image of the natural map $\bbDelta_w \to \bbnabla_w$.  Every simple object in $\Perv(\UGB,\bk)$ is isomorphic to $\IC^\mix_w\la n\ra$ for some $w \in W$ and some $n \in \Z$. In the special case $w=1$, we have $\IC^\mix_1=\cT_1=\bbDelta_1=\bbnabla_1$.

\begin{lem}
\label{lem:Cs-simple}
If $w \ne 1$, then $[ C_s(\IC^\mix_w) : \cT_1\la n\ra ] = 0$ for all $n \in \Z$.
\end{lem}
\begin{proof}
For $\cF \in \Perv(\UGB,\bk)$, let $q(\cF) := \sum_{n \in \Z} [\cF : \cT_1\la n\ra]$.  The lemma amounts to saying that $q(C_s(\IC^\mix_w)) = 0$ if $w \ne 1$.  By~\cite[Lemma~4.9]{modrap2}, there is a short exact sequence
\begin{equation}\label{eqn:delta-socle}
0 \to \cT_1\la -\ell(w)\ra \to \bbDelta_w \to \cG \to 0
\end{equation}
where $q(\cG) = 0$. We deduce that $q(\bbDelta_w) = 1$.  Then,
using~\cite[Proposition~10.5.3]{amrw}, we find that
$q(C_s(\bbDelta_w)) = 2$ for all $w \in W$.  (This holds even if $w =
1$.)  Now apply $C_s$ to~\eqref{eqn:delta-socle} to obtain
\[
0 \to C_s(\bbDelta_1\la -\ell(w)\ra) \to C_s(\bbDelta_w) \to C_s(\cG) \to 0.
\]
Since $q(C_s(\bbDelta_1\la -\ell(w)\ra)) = q(C_s(\bbDelta_w)) = 2$, and since $q$ is  additive on short exact sequences, we find that $q(C_s(\cG)) = 0$.  If $w \ne 1$, then $\IC^\mix_w$ is a quotient of $\cG$, so $C_s(\IC^\mix_w)$ is a quotient of $C_s(\cG)$.  It follows that $q(C_s(\IC^\mix_w)) = 0$, as desired.
\end{proof} 

\subsection{Tilting Hom formula}
The Hecke algebra $\cH_W$ is the algebra with free $\Z[v, v^{-1}]$-basis $\{ H_w \mid w \in W\}$, with multiplicative unit $H_1$, and multiplication determined by the rule
\[
 H_wH_s = \begin{cases}
           H_{ws} &\text{if $ws > w$;} \\
           (v^{-1} - v)H_w + H_{ws} &\text{if $ws < w$.}
          \end{cases}
\]
Similar formulas describe $H_sH_w$ depending on whether $sw < w$ or $sw > w$.  Next, for any expression $\uw = (s_1, \ldots, s_k)$, set
\[
 \uH_\uw := (H_{s_1} + v) \cdots (H_{s_k} + v) \in \cH_W.
\]
Observe that
\begin{equation}\label{eqn:hecke-us-w}
\uH_s H_w = (H_s + v)H_w =
\begin{cases}
H_{sw} + v H_w & \text{if $sw > w$;} \\
H_{sw} + v^{-1}H_w & \text{if $sw < w$.}
\end{cases}
\end{equation}
Define a symmetric $\Z[v, v^{-1}]$-bilinear pairing
\[
 \la -, - \ra : \cH_W \times \cH_W \to \Z[v, v^{-1}]
\]
by $\la H_x, H_y \ra = \delta_{xy}$ for $x,y \in W$.

\begin{lem}
\label{lem:tilting-hom-formula-pre}
Assume that $\bk$ is a field, and let $\uw$ be an expression.  We have
\[
\uH_\uw = \sum_{\substack{y \in W \\ n \in \Z}} (\cT_\uw : \bbDelta_y\la n\ra) v^n H_y
= \sum_{\substack{y \in W \\ n \in \Z}} (\cT_\uw : \bbnabla_y\la -n\ra) v^n H_y.
\]
\end{lem}
\begin{proof}[Sketch of proof]
According to~\cite[Lemma~10.5.3]{amrw}, for any $w \in W$, the perverse sheaf $C_s(\bbDelta_w)$ has a filtration by standard objects, and the multiplicities are given by $(C_s(\bbDelta_w) : \bbDelta_{sw}) = 1$,
\[
(C_s(\bbDelta_w) : \bbDelta_w\la n\ra) =
\begin{cases}
1 & \text{if $n = 1$ and $sw > w$, or if $n = -1$ and $sw < w$,} \\
0 & \text{otherwise,}
\end{cases}
\]
and $(C_s(\bbDelta_w) : \bbDelta_y\la n\ra) = 0$ in all other cases. Comparing this with~\eqref{eqn:hecke-us-w}, one can show by induction on the length of $\uw$ that $(\cT_\uw : \bbDelta_y\la n\ra)$ is equal to the coefficient of $v^n H_y$ in $\uH_\uw$.  Similar reasoning shows that this same integer is also equal to $(\cT_\uw: \bbnabla_y\la -n\ra)$.
\end{proof}

\begin{lem}
\label{lem:tilting-hom-formula}
 For any expressions $\uv, \uw$, we have
 \[
  \sum_{n \in \Z} \left( \mathrm{rk}_\bk \Hom_{\Dmix(\UGB,\bk)}(\cT_\uv, \cT_\uw\la n \ra) \right) v^n = \la\uH_\uv, \uH_\uw\ra.
 \]
\end{lem}
\begin{proof}
By the last statement of Proposition~\ref{prop:morph-Tilt-UGU}, we may check this after extension of scalars to any field. Over a field, we have
\[
\sum_{n \in \Z} \left( \dim \Hom(\cT_\uv, \cT_\uw\la n \ra) \right) v^n \\
= \sum_{\substack{y \in W\\ n,m \in \Z}}
(\cT_\uv : \bbDelta_y\la m\ra)(\cT_\uw\la n\ra : \bbnabla_y\la m\ra) v^n.
\]
On the other hand, using Lemma~\ref{lem:tilting-hom-formula-pre}, we have
\[
\la\uH_\uv, \uH_\uw\ra
= \sum_{\substack{y \in W\\ m, k \in \Z}} (\cT_\uv : \bbDelta_y\la m\ra)(\cT_\uw : \bbnabla_y\la -k\ra)v^{m+k}.
\]
The result follows by setting $k = n-m$.
\end{proof}

\section{Constructing a functor \texorpdfstring{$\V$}{V}}
\label{sec:functor-V}

In this section we assume that $A$ is of finite type, i.e.~that $\GKM$
is a connected complex reductive group. We also assume that $\bk$ is a field of characteristic $0$.\footnote{We restrict to characteristic~$0$ since this is the setting we will need. But more generally the results of this section hold if there exists a ring morphism $\Z' \to \bk$ and if the natural morphism $R^\vee \to \coH^\bullet(\BKM^\vee \backslash \GKM^\vee; \bk)$ introduced in the proof of Lemma~\ref{lem:dim-Hom-V'} is surjective. This is satisfied e.g.~if $\GKM$ is isomorphic to a product of groups $\mathrm{GL}_n(\bk)$ and of quasi-simple groups not of type $\mathbf{A}$, and if $\mathrm{char}(\bk)$ is good for $G$; see~\cite[Proposition~4.1]{modrap1}.} 

\subsection{The big tilting perverse sheaf}

Let $w_0$ be the longest element in $W$, and consider the indecomposable object $\cT_{w_0}$ in $\Tilt(\UGB,\bk)$. We set
\[
\cP:= \cT_{w_0} \langle - \ell(w_0) \rangle.
\]
The same arguments as in~\cite[\S 5.11]{modrap1}, using the results of~\cite[\S 4.4]{modrap2}, show that $\cP$ is the projective cover of the simple object $\cT_1$ in the abelian category $\Perv(\UGB,\bk)$. In particular, using~\eqref{eqn:gHom-Hom-LM} and~\eqref{eqn:realization-Dmix} we deduce that we have
\begin{equation}
\label{eqn:gHom-Tw0-T1}
\gHom_\LM(\cP, \cT_1)^n_m = \begin{cases}
\bk & \text{if $n=m=0$;} \\
0 & \text{otherwise.}
\end{cases}
\end{equation}
Using~\cite[Lemma~4.9]{modrap2} we also deduce that for $w \in W$ and $m \in \Z$ we have
\begin{equation}
\label{eqn:Hom-P-costandard}
\dim \bigl( \Hom_{\Dmix(\UGB,\bk)}(\cP, \bbnabla_w \langle m \rangle) \bigr) = \begin{cases}
1 & \text{if $m=-\ell(w)$;} \\
0 & \text{otherwise.}
\end{cases}
\end{equation}

\begin{lem}
\label{lem:multiplicities-P}
In the abelian category $\Perv(\UGB,\bk)$ we have
\[
[\cP : \cT_1 \langle m \rangle] = 0 \quad \text{unless $m \leq 0$,}
\]
and moreover $[\cP : \cT_1]=1$. In particular, we have
\[
\End_{\Dmix(\UGB,\bk)}(\cP) = \bk \cdot \id.
\]
\end{lem}

\begin{proof}
Since $\cP$ is a projective object in $\Perv(\UGB,\bk)$, it admits a filtration by standard objects. Moreover,~\eqref{eqn:Hom-P-costandard} shows that the subquotients in such a filtration are the objects $\bbDelta_w \langle -\ell(w) \rangle$ for $w \in W$, each appearing once. Combining this with~\cite[Lemma~4.9]{modrap2} we deduce that
\[
[\cP : \cT_1 \langle m \rangle] = \# \{w \in W \mid m=-2\ell(w)\},
\]
which implies the desired statement.
\end{proof}

\begin{lem}
\label{lem:Cs-P}
For any $s \in S$ we have $\Tmon_s \hatstar \cP \cong \cP \langle -1 \rangle \oplus \cP \langle 1 \rangle$.
\end{lem}

\begin{proof}
The object $\Tmon_s \hatstar \cP$ belongs to $\Tilt(\UGB, \bk)$. Since such an object is uniquely characterized by the multiplicities of standard objects in a standard filtration, to conclude it suffices to prove that for any $i \in \Z$ and $w \in W$ we have
\[
(\Tmon_s \hatstar \cP : \bbDelta_w \langle i \rangle) = (\cP \langle -1 \rangle \oplus \cP \langle 1 \rangle : \bbDelta_w \langle i \rangle),
\]
i.e.~that
\[
(\Tmon_s \hatstar \cP : \bbDelta_w \langle i \rangle) = \begin{cases}
1 & \text{if $i=-\ell(w) \pm 1$;}\\
0 & \text{otherwise.}
\end{cases}
\]
This easily follows from~\cite[Lemma~10.5.3]{amrw} (see also~\cite[Proof of Lem\-ma~10.5.4]{amrw}).
\end{proof}

\subsection{A free-monodromic analogue of \texorpdfstring{$\cP$}{P}}
\label{ss:Pmon}

From now on we fix (once and for all) an object $\Tmon_{w_0}$ as in~\S\ref{ss:fm-tilting-sheaves}; then $\Tmon_{w_0}$ belongs to $\Tilt(\UGU,\bk)$ and satisfies $\ForFMLM(\Tmon_{w_0}) \cong \cT_{w_0}$. We set
\[
\Pmon := \Tmon_{w_0} \langle -\ell(w_0) \rangle,
\]
so that $\ForFMLM(\Pmon) \cong \cP$.

Using Proposition~\ref{prop:morph-Tilt-UGU} and~\eqref{eqn:gHom-Tw0-T1} we see that there exists an isomorphism of bigraded vector spaces
\[
\gHom_\FM(\Pmon,\Tmon_1) \cong R^\vee.
\]
In particular, we deduce that
\[
\dim_\bk \bigl( \Hom_{\Tilt(\UGU,\bk)}(\Pmon, \Tmon_1) \bigr) = 1.
\]
We also fix a nonzero morphism $\xi : \Pmon \to \Tmon_1$ (which is unique up to nonzero scalar), and set $\xi' := \ForFMLM(\xi)$, a generator of $\Hom_{\Dmix(\UGB, \bk)}(\cP, \cT_1)$.

\begin{lem}
\label{lem:conv-P}
The objects $\Pmon \hatstar \cP$ and $\Pmon \hatstar \Pmon \hatstar \cP$ are direct sums of copies of $\cP \langle i \rangle$ with $i \in \Z_{\leq 0}$, with $\cP$ appearing once.
\end{lem}

\begin{proof}
It is enough to prove the claim for $\Pmon \hatstar \cP$; the case of $\Pmon \hatstar \Pmon \hatstar \cP$ follows.

Let $\uw$ be a reduced expression for $w_0$. Then, by~\cite[Theorem~10.7.1]{amrw}, $\Pmon$ is a direct summand in $\Tmon_{\uw} \langle - \ell(w_0) \rangle$, so $\Pmon \hatstar \cP$ is a direct summand in $(\Tmon_{\uw} \langle - \ell(w_0) \rangle) \hatstar \cP$. The first claim is then a direct consequence of Lemma~\ref{lem:Cs-P}. We also deduce that the multiplicity of $\cP$ in $\Pmon \hatstar \cP$ is at most $1$.

To prove the claim about the multiplicity of $\cP$, we observe that the morphism $\id \hatstar \xi' : \Pmon \hatstar \cP \to \Pmon \hatstar \cT_1 \cong \cP$ is surjective. (Since $\Pmon$ is a direct summand in $\Tmon_{\uw} \langle - \ell(w_0) \rangle$, this follows from Lemma~\ref{lem:Cs-exact}.)
Since the image of any morphism $\cP \langle i \rangle \to \cP$ with $i<0$ is contained in the radical of $\cP$, we deduce that $\cP$ does indeed occur as a direct summand of $\Pmon \hatstar \cP$.
\end{proof}

\begin{cor}
\label{cor:Hom-P-conv}
We have
\[
\dim_\bk \bigl( \Hom_{\Dmix(\UGB, \bk)}(\cP, \Pmon \hatstar \cP \langle i \rangle) \bigr) = \begin{cases}
1 & \text{if $i=0$;} \\
0 & \text{if $i<0$.}
\end{cases}
\]
We also have
\[
\dim_\bk \bigl( \Hom_{\Dmix(\UGB, \bk)}(\cP, \Pmon \hatstar \Pmon \hatstar \cP \langle i \rangle) \bigr) = \begin{cases}
1 & \text{if $i=0$;} \\
0 & \text{if $i<0$.}
\end{cases}
\]
\end{cor}

\begin{proof}
The claims follow from Lemma~\ref{lem:multiplicities-P} and Lemma~\ref{lem:conv-P}.
\end{proof}

\subsection{Morphisms from \texorpdfstring{$\Pmon$}{P} to \texorpdfstring{$\Tmon_s$}{Ts}}
\label{ss:Hom-P-Ts}

Let us fix $s \in S$. Consider the morphism
\[
\ForFMLM(\hat{\epsilon}_s) \langle -1 \rangle : \cT_s \langle -1 \rangle \to \cT_1,
\]
where $\hat{\epsilon}_s$ is defined in~\cite[\S 5.3.4]{amrw}.
Since $\cP$ is projective, and since $[\cT_s \langle -1 \rangle : \cT_1]=1$ (see~\cite[Example~4.6.4]{amrw}), there exists a unique morphism $\zeta_s' : \cP \to \cT_s \langle -1 \rangle$ such that $(\ForFMLM(\hat{\epsilon}_s) \langle -1 \rangle) \circ \zeta_s' = \xi'$.

\begin{lem}
\label{lem:V-Ts}
There exists a unique morphism
\[
\zeta_s : \Pmon \to \Tmon_s \langle - 1 \rangle
\]
in $\Dmix(\UGU, \bk)$ such that $\ForFMLM(\zeta_s)=\zeta_s'$. Moreover we have $(\hat{\epsilon}_s \langle -1 \rangle) \circ \zeta_s = \xi$, and there exists a unique $\Z^2$-graded $R^\vee$-bimodule isomorphism
\[
R^\vee \otimes_{(R^\vee)^s} R^\vee \simto \gHom_{\FM}(\Pmon, \Tmon_s  \langle -1 \rangle)
\]
sending $1 \otimes 1$ to $\zeta_s$.
\end{lem}

\begin{proof}
Since $[\cT_s \langle -1 \rangle : \cT_1] = [\cT_s \langle -1 \rangle : \cT_1 \langle -2 \rangle]=1$ and $[\cT_s \langle -1 \rangle : \cT_1 \langle m \rangle]=0$ if $m \notin \{0, -2\}$, the $\Z^2$-graded $\bk$-vector space $\gHom_\LM(\cP, \cT_s \langle -1 \rangle)$ has dimension $2$, and is nonzero in degrees $0$ and $-2$. By Proposition~\ref{prop:morph-Tilt-UGU} this implies that the graded (right) $R^\vee$-module $\gHom_{\FM}(\Pmon, \Tmon_s \langle -1 \rangle)$ is free of rank $2$, and generated in degrees $0$ and $-2$, and that the functor $\ForFMLM$ induces an isomorphism
\[
\Hom_{\Dmix(\UGU, \bk)}(\Pmon, \Tmon_s \langle -1 \rangle) \simto \Hom_{\Dmix(\UGB, \bk)}(\cP, \cT_s \langle -1 \rangle).
\]
This proves the existence and uniqueness of $\zeta_s$. The fact that $(\hat{\epsilon}_s \langle -1 \rangle) \circ \zeta_s = \xi$ follows from similar arguments.

Now, we consider the morphism
\[
R^\vee \otimes_\bk R^\vee \to \gHom_{\FM}(\Pmon, \Tmon_s  \langle -1 \rangle)
\]
defined by $x \otimes y \mapsto (\mu_{\Tmon_s}(x) \circ \zeta_s) \cdot y$.  By~\eqref{eqn:rightmon-morph}, for $x,y \in R^\vee$, we have
\[
(\mu_{\Tmon_s}(x) \circ \zeta_s) \cdot y = (\mu_{\Tmon_s}(x) \cdot y) \circ \zeta_s.
\]
In view of~\cite[Proposition~5.3.1 and its proof]{amrw}, this implies that our morphism factors through a ($\Z^2$-graded) bimodule morphism
\[
R^\vee \otimes_{(R^\vee)^s} R^\vee \to \gHom_{\LM}(\Pmon, \Tmon_s  \langle -1 \rangle).
\]
The right $R^\vee$-modules under consideration are both free of rank $2$, and generated in degrees $0$ and $-2$ (see again~\cite[Proposition~5.3.1 and its proof]{amrw} for the left-hand side). Hence to prove that our morphism is an isomorphism, it suffices to show that the induced morphism
\[
R^\vee \otimes_{(R^\vee)^s} \bk \to \gHom_{\FM}(\Pmon, \Tmon_s  \langle -1 \rangle) \otimes_{R^\vee} \bk \cong \gHom_\LM(\cP, \cT_s \langle -1 \rangle)
\]
is an isomorphism. The latter fact is clear from the discussion in \cite[Example~4.7.4]{amrw}.
\end{proof}

\subsection{Coalgebra structure}

\begin{prop}
\label{prop:Pmon-coalgebra}
The object $\Pmon$ admits a canonical coalgebra structure in the monoidal category $\Tilt(\UGU,\bk)$ with counit $\xi$.
\end{prop}

\begin{proof}
Our proof is very close to that in~\cite[Proposition~4.6.4]{by}. We need to define a counit morphism $\Pmon \to \Tmon_1$ (which should be $\xi$) and a comultiplication morphism $\delta : \Pmon \to \Pmon \hatstar \Pmon$, and check that these data satisfy the counit and coassociativity axioms.

To define the comultiplication, we first observe that there exists a unique morphism $\delta' : \cP \to \Pmon \hatstar \cP$ such that $(\xi \hatstar \xi') \circ \delta' = \xi'$. In fact, the morphism $\xi \hatstar \xi' = (\xi \hatstar \id_{\cT_1}) \circ (\id_{\Pmon} \hatstar \xi')$ is surjective by the proof of Lemma~\ref{lem:conv-P}. Since its restriction to any summand of the form $\cP \langle i \rangle$ with $i<0$ must vanish, this proves the existence of $\delta'$ in view of Lemma~\ref{lem:conv-P} and~\eqref{eqn:gHom-Tw0-T1}. Uniqueness is also clear from this lemma since $\End_{\Dmix(\UGB,\bk)}(\cP) = \bk \cdot \id$ and $\Hom_{\Dmix(\UGB,\bk)}(\cP, \cP\langle i \rangle)=0$ if $i<0$.

Now, combining Corollary~\ref{cor:Hom-P-conv} and Proposition~\ref{prop:morph-Tilt-UGU} (see also~\eqref{eqn:For-hatstar}) we see that $\gHom_\FM(\Pmon, \Pmon \hatstar \Pmon)$ is a direct sum of copies of $R^\vee\la i\ra$ with $i \le 0$, in which $R^\vee$ itself occurs with multiplicity~$1$.  Moreover, the functor $\ForFMLM$ induces an isomorphism
\[
\Hom_{\Dmix(\UGU,\bk)}(\Pmon, \Pmon \hatstar \Pmon) \simto \Hom_{\Dmix(\UGB,\bk)}(\cP, \Pmon \hatstar \cP).
\]
From the previous paragraph we then deduce that there exists a unique morphism $\delta : \Pmon \to \Pmon \hatstar \Pmon$ such that $(\xi \hatstar \xi) \circ \delta = \xi$. This defines our comultiplication.

It remains to show that $\xi$ and $\delta$ satisfy the required axioms. We observe that as above the vector space $\Hom_{\Dmix(\UGU,\bk)}(\Pmon, \Pmon \hatstar \Pmon \hatstar \Pmon)$ is $1$-dimensional. Hence $(\delta \hatstar \id) \circ \delta$ and $(\id \hatstar \delta) \circ \delta$ are proportional. Moreover, we have
\[
(\xi \hatstar \xi \hatstar \xi) \circ ((\delta \hatstar \id) \circ \delta) = (\xi \hatstar \xi \hatstar \xi) \circ ((\id \hatstar \delta) \circ \delta) = \xi.
\]
Hence $(\delta \hatstar \id) \circ \delta = (\id \hatstar \delta) \circ \delta$, proving coassociativity. The counit axiom can be checked similarly, and the proof is complete.
\end{proof}

\subsection{The functor \texorpdfstring{$\V$}{V}}

The $\Z^2$-graded algebra $R^\vee$ is concentrated in degrees in $\{0\} \times \Z$, so it makes sense to regard it as just a $\Z$-graded algebra. Similarly, if $\cF, \cG$ belong to $\Tilt(\UGU, \bk)$, then by Proposition~\ref{prop:morph-Tilt-UGU}, $\gHom_\FM(\cF,\cG)$ can (and will) be regarded as a $\Z$-graded $\bk$-module.

Consider the $\Z$-graded algebra morphism
\[
R^\vee \otimes R^\vee \to \gEnd_\FM(\Pmon)
\]
sending $x \otimes y$ to $\mu_{\Pmon}(x) \cdot y$. This morphism allows us to define a functor
\[
\V := \gHom_\FM(\Pmon, -) : \Tilt(\UGU, \bk) \to R^\vee \grbiMod R^\vee,
\]
where $R^\vee \grbiMod R^\vee$ is the category of graded $R^\vee$-bimodules. This functor intertwines Tate twist with the shift-of-grading functor $\langle 1 \rangle$ on $R^\vee \grbiMod R^\vee$, where the latter is normalized as in~\cite[\S 3.1]{amrw}.

The arguments below will sometimes make use of the functor
\[
\V' := \gHom_\LM(\cP, -) : \Tilt(\UGB, \bk) \to R^\vee \grlMod,
\]
where $R^\vee \grlMod$ is the category of graded left $R^\vee$-modules, and the morphism $R^\vee \to \gEnd_\LM(\cP)$ is $\mu_{\cP}$.  Proposition~\ref{prop:morph-Tilt-UGU} implies that the following diagram commutes up to natural isomorphism:
\begin{equation}\label{eqn:VVprime-commute}
\begin{tikzcd}
\Tilt(\UGU,\bk) \ar[r, "\V"] \ar[d, "\ForFMLM"'] & R^\vee \grbiMod R^\vee \ar[d, "({-}) \otimes_{R^\vee} \bk"] \\
\Tilt(\UGB,\bk) \ar[r, "\V'"] & R^\vee \grlMod.
\end{tikzcd}
\end{equation}

\begin{prop}
\label{prop:V-monoidal}
The functor $\V$ admits a canonical monoidal structure which intertwines the convolution $\hatstar$ on $\Tilt(\UGU,\bk)$ and the natural tensor product of graded $R^\vee$-bimodules.
\end{prop}

\begin{proof}
Let $\gamma : R^\vee \simto \V(\Tmon_1)$ be the isomorphism determined by $\gamma(1) = \xi$, where $\xi : \Pmon \to \Tmon_1$ is the morphism fixed in~\S\ref{ss:Pmon}. We need to define a natural isomorphism of bifunctors
\[
 \beta : \V(-) \otimes_{R^\vee} \V(-) \to \V(- \hatstar -)
\]
so that the data $(\V, \beta, \gamma)$ satisfies the associativity and unitality axioms of a monoi\-dal functor.

We begin by defining a morphism of bifunctors
\[
\V(-) \otimes_\bk \V(-) \to \V(- \hatstar -)
\]
as follows. If $\cF, \cG$ belong to $\Tilt(\UGU,\bk)$ and $f \in \V(\cF)_m$, $g \in \V(\cG)_{m'}$, then we can consider
\[
f \hatstar g : \Pmon \hatstar \Pmon \to \cF \hatstar \cG \langle -m-m' \rangle.
\]
Composing this morphism with the comultiplication from Proposition~\ref{prop:Pmon-coalgebra}, we obtain an element of $\V(\cF \hatstar \cG)_{m+m'}$. This defines the desired morphism, and by~\eqref{eqn:morph-Rvee-left-right} this morphism factors through a morphism
\[
\beta : \V(-) \otimes_{R^\vee} \V(-) \to \V(- \hatstar -).
\]
For later use, note that a very similar construction, using the map $\delta': \cP \to \Pmon \hatstar \cP$ from the proof of Proposition~\ref{prop:Pmon-coalgebra} in place of the comultiplication, yields a natural transformation
\[
\beta' : \V(-) \otimes_{R^\vee} \V'(-) \to \V'(- \hatstar -).
\]

The associativity axiom for $(\V, \beta, \gamma)$ follows from the bifunctoriality of $\hatstar$, the compatibility of the associator isomorphism in $\Tilt(\UGU,\bk)$ with morphisms (see \cite[Proposition~7.2.2]{amrw}), and the coassociativity axiom for the coalgebra structure of $\Pmon$ (see Proposition~\ref{prop:Pmon-coalgebra}). The unitality axioms for $(\V, \beta, \gamma)$ follow from the naturality of the unitor isomorphisms in $\Tilt(\UGU,\bk)$ (see \cite[Lemma~7.1.1]{amrw}) and the counit axioms for the coalgebra structure of $\Pmon$ (see Proposition~\ref{prop:Pmon-coalgebra}).

To conclude, it remains only to prove that $\beta$ is an isomorphism. By Proposition~\ref{prop:morph-Tilt-UGU},
$\V$ takes values in the subcategory consisting of bimodules which are free as graded right $R^\vee$-modules. It is therefore enough to prove that $\beta$ remains an isomorphism after applying $({-}) \otimes_{R^\vee} \bk: R^\vee \grbiMod R^\vee \to R^\vee \grlMod$.  In other words, it is enough to prove that $\beta'$ is an isomorphism.  Using~\eqref{eqn:Tmon-conv}, we can further reduce the problem to showing that for any $s \in S$, the morphism of functors
\[
\beta'(\Tmon_s, -) : \V(\Tmon_s) \otimes_{R^\vee} \V'(-) \to \V'(\Tmon_s \hatstar -)
\]
is an isomorphism.

For this we will ``extend'' the functors $\V(\Tmon_s) \otimes_{R^\vee} \V'(-)$ and $\V'(\Tmon_s \hatstar -)$ to exact functors $\Perv(\UGB, \bk) \to R^\vee \grlMod$ as follows. First, as explained at the beginning of the section, the category $\Tilt(\UGB, \bk)$ identifies naturally with an additive subcategory of $\Perv(\UGB, \bk)$. We can extend $\V(\Tmon_s) \otimes_{R^\vee} \V'(-)$ to a functor
\[
\W_1^s : \Perv(\UGB, \bk) \to R^\vee \grlMod
\]
by setting $\W_1^s(\cF) := \V(\Tmon_s) \otimes_{R^\vee} \gHom_{\LM}(\cP, \cF)$. Since $\V(\Tmon_s)$ is free as a right $R^\vee$-module and since $\cP$ is a projective object in $\Perv(\UGB, \bk)$, this functor is exact.
For the functor $\V'(\Tmon_s \hatstar -)$, we define an exact functor
\[
\W_2^s : \Perv(\UGB, \bk) \to R^\vee \grlMod
\]
by setting $\W_2^s(\cF) := \gHom_\LM(\cP, C_s(\cF))$. In this case, exactness follows from Lemma~\ref{lem:Cs-exact}.

We claim that the morphism $\beta'(\Tmon_s, -)$ is induced by a morphism of functors $\gamma^s : \W^s_1 \to \W^s_2$.  To see this we need a different construction of the functor $\W_2^s$. Consider the functor
\[
\Kb(\V'(\Tmon_s \hatstar -)) : \Kb(\Tilt(\UGB, \bk)) \to \Kb(R^\vee \grlMod).
\]
As seen in~\S\ref{ss:perv}, the natural functor
\[
\Kb(\Tilt(\UGB, \bk)) \to \Db(\Perv(\UGB, \bk))
\]
is an equivalence. Moreover, it is clear by construction that the following diagram commutes:
\[
\begin{tikzcd}[column sep=large]
\Kb \Tilt(\UGB, \bk) \ar[d, "\wr" swap] \ar[rr, "\Kb(\V'(\Tmon_s \hatstar -))"] && \Kb(R^\vee \grlMod) \ar[d] \\
\Db \Perv(\UGB, \bk) \ar[rr, "\Db(\W_2^s)"] && \Db(R^\vee \grlMod).
\end{tikzcd}
\]
Similarly we have a commutative diagram
\[
\begin{tikzcd}[column sep=huge]
\Kb \Tilt(\UGB, \bk) \ar[d, "\wr" swap] \ar[rr, "\Kb(\V(\Tmon_s) \otimes_{R^\vee} \V'(-))"] && \Kb(R^\vee \grlMod) \ar[d] \\
\Db \Perv(\UGB, \bk) \ar[rr, "\Db(\W_1^s)"] && \Db(R^\vee \grlMod).
\end{tikzcd}
\]
Hence the morphism of functors $\beta'(\Tmon_s, -)$ induces a morphism $\Db(\W_1^s) \to \Db(\W_2^s)$, which restricts to the desired morphism $\gamma^s$.

We will now prove that $\gamma^s$ is an isomorphism, thereby finishing the proof. By the 5-lemma, it is enough to prove that $\gamma^s(\cF)$ is an isomorphism for any simple object $\cF$ in $\Perv(\UGB, \bk)$. After a Tate twist, we may assume that $\cF = \IC^\mix_w$ for some $w \in W$.  If $w \ne 1$, then it is clear that $\W_1^s(\IC^\mix_w) = 0$, and it follows from Lemma~\ref{lem:Cs-simple} that $\W_2^s(\IC^\mix_w) = 0$, so there is nothing to prove in this case.  It remains to consider the case $w = 1$. In other words, we must prove that the morphism
\[
\gamma^s(\cT_1) : \W_1^s(\cT_1) \to \W_2^s(\cT_1)
\]
is an isomorphism. By construction this morphism identifies with $\beta'(\Tmon_s, \cT_1)$. Recall now that 
\[
\V(\Tmon_s) \cong R^\vee \otimes_{(R^\vee)^s} R^\vee \langle 1 \rangle, \quad \V'(\cT_1) \cong \bk, \quad \V'(\Tmon_s \hatstar \cT_1) \cong \V'(\cT_s) \cong R^\vee \otimes_{(R^\vee)_+^s} \bk \langle 1 \rangle.
\]
In particular, both $\W_1^s(\cT_1)$ and $\W_2^s(\cT_1)$ are cyclic as left $R^\vee$-modules, and generated in degree $1$. Hence to conclude, it remains only to prove that
\[
\beta'(\Tmon_s, \cT_1)(\zeta_s \otimes \xi') = \zeta_s',
\]
i.e.~that
\begin{equation}\label{eqn:V-monoidal-last}
(\zeta_s \hatstar \xi') \circ \delta' = \zeta'_s.
\end{equation}
(Here we identify $\Tmon_s \hatstar \cT_1$ and $\cT_s$ in the canonical way; see~\eqref{eqn:For-hatstar}.) However we have
\begin{multline*}
(\ForFMLM(\hat{\epsilon}_s) \langle -1 \rangle) \circ (\zeta_s \hatstar \xi') \circ \delta' = ((\hat{\epsilon}_s \langle -1 \rangle)  \hatstar \id_{\cT_1}) \circ (\zeta_s \hatstar \xi') \circ \delta' \\
= ((\hat{\epsilon}_s \langle -1 \rangle \circ \zeta_s)  \hatstar \xi') \circ \delta' = (\xi \hatstar \xi') \circ \delta' = \xi'.
\end{multline*}
By construction of $\zeta_s'$ (see~\S\ref{ss:Hom-P-Ts}), this proves~\eqref{eqn:V-monoidal-last}, as desired.
\end{proof}

\subsection{Full faithfulness}
\label{ss:V-ff}

The goal of this subsection is to prove the following claim.

\begin{thm}
\label{thm:V-fully-faithful}
The functor
\[
\V : \Tilt(\UGU, \bk) \to R^\vee \grbiMod R^\vee
\]
is fully faithful.
\end{thm}

Before proving this result we need some preliminary lemmas.

\begin{lem}
\label{lem:V'-faithful}
The functor
\[
\V' : \Tilt(\UGB, \bk) \to R^\vee \grlMod
\]
introduced in the proof of Proposition~{\rm \ref{prop:V-monoidal}}
is faithful.
\end{lem}

\begin{proof}
The argument for this proof is taken from~\cite{bbm}. By construction of the functor $\V'$, to prove the lemma it suffices to prove that the image of any nonzero morphism between objects of $\Tilt(\UGB, \bk)$ admits a Tate twist of $\cT_1$ as a composition factor. In fact this follows from the observation that the only possible simple quotients of objects of $\Tilt(\UGB, \bk)$ are Tate twists of $\cT_1$, since such objects admit costandard filtrations, and since the head of any costandard object in $\Perv(\UGB, \bk)$ is a Tate twist of $\cT_1$, by~\cite[Lemma~4.9]{modrap2}.
\end{proof}

\begin{lem}
\label{lem:dim-Hom-V'}
For any $\cF, \cG$ in $\TiltBSp(\UGB, \bk)$, the $\bk$-vector spaces
\[
\bigoplus_{m \in \Z} \Hom_{\Dmix(\UGB, \bk)}(\cF, \cG \langle m \rangle) \quad \text{and} \quad \Hom_{R^\vee}(\V'(\cF), \V'(\cG))
\]
have the same dimension.
\end{lem}

\begin{proof}[Sketch of proof]
By construction of the category $\TiltBSp(\UGB, \bk)$, we can assume that $\cF=\cT_\uv$ and $\cG = \cT_\uw$ for some expressions $\uv, \uw$. In this case, the dimension of $\bigoplus_{n \in \Z} \Hom_{\Dmix(\UGB, \bk)}(\cT_{\uv}, \cT_{\uw} \langle n \rangle)$ is determined in Lemma~\ref{lem:tilting-hom-formula}.

On the other hand, let $\TKM^\vee$ be the torus which is Langlands dual to $\TKM$, and let $\GKM^\vee$ be a complex connected reductive group containing $\TKM^\vee$ as a maximal torus and whose root system (with respect to $\TKM^\vee$) is dual to that of $(\GKM,\TKM)$. Let also $\BKM^\vee$ be the Borel subgroup of $\GKM^\vee$ containing $\TKM^\vee$ whose roots are the coroots of $\BKM$. Then the Borel construction shows that there exists a natural surjective algebra morphism $R^\vee \to \coH^{-\bullet}(\BKM^\vee \backslash \GKM^\vee; \bk)$. For any $s \in S$ we let $\PKM_s^\vee$ be the minimal parabolic subgroup of $\GKM^\vee$ containing $\BKM^\vee$, and set $\cE_s^\vee := \uu{\bk}_{\BKM^\vee \backslash \PKM_s^\vee} \{1\} \in \Db(\BKM^\vee \backslash \GKM^\vee / \BKM^\vee, \bk)$. Then for any expression $\uu{u}=(s_1, \ldots, s_r)$, we set
\[
\cE_{\uu{u}}^\vee := \cE_{s_1}^\vee \star^{\BKM^\vee} \cdots \star^{\BKM^\vee} \cE_{s_r}^\vee,
\]
where $\star^{\BKM^\vee}$ is the natural convolution product on $\Db(\BKM^\vee \backslash \GKM^\vee / \BKM^\vee, \bk)$. (In the present proof these objects will be considered as objects in the ordinary derived category $\Db(\BKM^\vee \backslash \GKM^\vee, \bk)$.)

If $\uv=(s_1, \ldots, s_i)$ and $\uw = (t_1, \ldots, t_j)$, then it is well known from the theory of Soergel bimodules that we have canonical isomorphisms of $R^\vee$-modules
\begin{align*}
\coH^{-\bullet}(\BKM^\vee \backslash \GKM^\vee, \cE_\uv^\vee) &\cong R^\vee \otimes_{(R^\vee)^{s_1}} \cdots \otimes_{(R^\vee)^{s_{i-1}}} R^\vee\otimes_{(R^\vee)^{s_{i}}} \bk \langle i \rangle, \\ 
\coH^{-\bullet}(\BKM^\vee \backslash \GKM^\vee, \cE_\uw^\vee) &\cong R^\vee \otimes_{(R^\vee)^{t_1}} \cdots \otimes_{(R^\vee)^{t_{j-1}}} R^\vee \otimes_{(R^\vee)^{t_{j}}} \bk \langle j \rangle;
\end{align*}
see e.g.~\cite[Korollar~2]{soergel-kat}.
Comparing with Lemma~\ref{lem:V-Ts} and using Proposition~\ref{prop:V-monoidal} and its proof, we deduce isomorphisms of $R^\vee$-modules
\[
\coH^{-\bullet}(\BKM^\vee \backslash \GKM^\vee, \cE_\uv^\vee) \cong \V'(\cT_\uv), \quad \coH^{-\bullet}(\BKM^\vee \backslash \GKM^\vee, \cE_\uw^\vee) \cong \V'(\cT_\uw).
\]

It is also well known that the functor $\coH^{-\bullet}(\BKM^\vee \backslash \GKM^\vee,-)$ induces an isomorphism
\begin{multline*}
\bigoplus_{m \in \Z} \Hom_{\Db(\BKM^\vee \backslash \GKM^\vee, \bk)}(\cE_\uv, \cE_\uw \{m\}) \simto \\
\Hom_{\coH^{-\bullet}(\BKM^\vee \backslash \GKM^\vee; \bk)} \bigl(\coH^{-\bullet}(\BKM^\vee \backslash \GKM^\vee, \cE_\uv^\vee), \coH^{-\bullet}(\BKM^\vee \backslash \GKM^\vee, \cE_\uw^\vee) \bigr),
\end{multline*}
by~\cite[Erweiterungssatz~17]{soergel-kat}. (See also~\cite{ginzburg} and~\cite[Theorem~4.1]{arider1} for alternative proofs, in more general contexts.) Using~\cite[Proposition~2.6]{jmw} to compute the dimension of the left-hand side, we finally obtain a formula for the dimension of $\Hom_{R^\vee}(\V'(\cT_{\uv}), \V'(\cT_{\uw}))$ which coincides with the one for the vector space $\bigoplus_{n \in \Z} \Hom_{\Dmix(\UGB, \bk)}(\cT_{\uv}, \cT_{\uw} \langle n \rangle)$ considered above. 
\end{proof}

\begin{proof}[Proof of Theorem~{\rm \ref{thm:V-fully-faithful}}]
We have to prove that for any expressions $\uv, \uw$ and any $m \in \Z$, the functor $\V$ induces an isomorphism
\[
\Hom_{\Dmix(\UGU, \bk)}(\Tmon_\uv, \Tmon_\uw \langle m \rangle) \simto \Hom_{R^\vee \grbiMod R^\vee}(\V(\Tmon_\uv), \V(\Tmon_\uw) \langle m \rangle).
\]
In fact we will prove that this functor induces an isomorphism
\begin{equation}
\label{eqn:V-ff}
\gHom_{\FM}(\Tmon_\uv, \Tmon_\uw) \simto \Hom_{R^\vee \biMod R^\vee}(\V(\Tmon_\uv), \V(\Tmon_\uw)),
\end{equation}
where the right-hand side means morphisms of (ungraded) bimodules.

For this we note that by Proposition~\ref{prop:morph-Tilt-UGU} the left-hand side is graded free as a right $R^\vee$-module, of finite rank, and that there exists a canonical isomorphism
\[
\gHom_{\FM}(\Tmon_\uv, \Tmon_\uw) \otimes_{R^\vee} \bk \simto \gHom_\LM(\cT_\uv, \cT_\uw).
\]
Now, recall the notation introduced in the proof of Lemma~\ref{lem:dim-Hom-V'}. Then we have a natural surjective algebra morphism $R^\vee \otimes_\bk R^\vee \to \coH^{-\bullet}(\BKM^\vee \backslash \GKM^\vee / \BKM^\vee ; \bk)$, canonical isomorphisms
\[
\coH^{-\bullet}(\BKM^\vee \backslash \GKM^\vee / \BKM^\vee, \cE_\uv^\vee) \cong \V(\cT_\uv), \quad \coH^{-\bullet}(\BKM^\vee \backslash \GKM^\vee / \BKM^\vee, \cE_\uw^\vee) \cong \V(\cT_\uw),
\]
and the functor $\coH^\bullet(\BKM^\vee \backslash \GKM^\vee / \BKM^\vee, -)$ induces an isomorphism
\begin{multline*}
\bigoplus_{m \in \Z} \Hom_{\Db(\BKM^\vee \backslash \GKM^\vee / \BKM^\vee, \bk)}(\cE_\uv, \cE_\uw [m]) \simto \\
\Hom_{\coH^{-\bullet}(\BKM^\vee \backslash \GKM^\vee / \BKM^\vee ; \bk)} \bigl(\coH^{-\bullet}(\BKM^\vee \backslash \GKM^\vee / \BKM^\vee, \cE_\uv^\vee), \coH^{-\bullet}(\BKM^\vee \backslash \GKM^\vee / \BKM^\vee, \cE_\uw^\vee) \bigr),
\end{multline*}
by~\cite[Proposition~2]{soergel-Langlands}. (See also~\cite[Proposition~3.1.5]{by} and~\cite[Remark~3.19]{mr:etsps} for alternative proofs, in more general contexts.) It is well known that the left-hand side is graded free as a right $R^\vee$-module, of finite rank, and that the natural morphism
\[
\left( \bigoplus_{m \in \Z} \Hom_{\Db(\BKM^\vee \backslash \GKM^\vee / \BKM^\vee, \bk)}(\cE_\uv, \cE_\uw [m]) \right) \otimes_{R^\vee} \bk \to \bigoplus_{m \in \Z} \Hom_{\Db(\BKM^\vee \backslash \GKM^\vee, \bk)}(\cE_\uv, \cE_\uw [m])
\]
is an isomorphism: see e.g.~\cite[Lemma~2.2]{mr:etsps}. Combining this with the results used in the proof of Lemma~\ref{lem:dim-Hom-V'}, we deduce that $\Hom_{R^\vee \biMod R^\vee}(\V(\Tmon_\uv), \V(\Tmon_\uw))$ is free over $R^\vee$, of finite rank, and that the natural morphism
\[
\left( \Hom_{R^\vee \biMod R^\vee}(\V(\Tmon_\uv), \V(\Tmon_\uw)) \right) \otimes_{R^\vee} \bk \to \Hom_{R^\vee \lMod}(\V'(\cT_\uv), \V'(\cT_\uw))
\]
is an isomorphism.

Finally, the isomorphism~\eqref{eqn:V-ff} follows from the fact that $\V'$ is fully faithful, as follows from Lemma~\ref{lem:V'-faithful} and Lemma~\ref{lem:dim-Hom-V'}.
\end{proof}

\section{From diagrams to tilting perverse sheaves}
\label{sec:Diag-Tilt}

In this section we come back to the general assumption that $\bk$ is a Noetherian integral domain of finite global dimension such that there exists a ring morphism $\Z' \to \bk$. 

\subsection{Statement}

We now consider the realization $\fh^*_\bk$ of $W$ over $\bk$ which is dual to $\fh_\bk$, i.e.~given by the triple $(V^*, \{\alpha_s\}_{s \in S}, \{\alpha_s^\vee \}_{s \in S})$ where $V^* := \Hom_\bk(V,\bk) = \bk \otimes_\Z \bX$. This realization satisfies Demazure surjectivity, so that we can consider the corresponding Elias--Williamson diagrammatic category $\DiagBS(\fh_\bk^*,W)$.

The goal of this section is to prove the following result.

\begin{thm} \label{thm:Diag-Tilt}
There exists a canonical $\bk$-linear monoidal functor
\[
\Phi : \DiagBS(\fh_\bk^*,W) \to \TiltBS(\UGU, \bk)
\]
sending $B_s$ to $\Tmon_s$, and such that $\Phi \circ (1) = \langle 1 \rangle \circ \Phi$.
\end{thm}

The construction of $\Phi$ is similar to the construction of the functor $\Psi$ appearing in~\eqref{eqn:equivalence-rw} (see~\cite[\S 10.4--10.5]{rw}). Namely, we define $\Phi$ on objects by $\Phi(B_\uw) = \Tmon_\uw$. To define $\Phi$ on morphisms, we need to specify the images of the generating morphisms, and check that these images satisfy the appropriate relations. These images will be described in a rather explicit way; then to check the relations we will reduce to the case $\bk$ is a field of characteristic $0$ and $A$ is of finite type, in which case we can use the functor $\V$ of Section~\ref{sec:functor-V} to deduce this claim from the corresponding (known) fact for Soergel bimodules.

We need only consider the case when $\bk = \Z'$: by the last statement of Proposition~\ref{prop:morph-Tilt-UGU}, we deduce from this case the definition of $\Phi$ for any $\bk$, and the fact that the relations hold over $\Z'$ implies that they also hold over $\bk$.

\subsection{Construction of the functor \texorpdfstring{$\Phi$}{Phi}}
\label{ss:rels-defn}

In this subsection, we define the image of $\Phi$ on each generating morphism.

\subsubsection{Polynomials}

Consider the morphism $B_\varnothing \to B_\varnothing(2m)$ given by a region labelled by $x \in (R^\vee)^0_{-2m}$. We define
\[
 \Phi\left(
  \begin{tikzpicture}[thick,scale=0.07,baseline=-2pt]
   \node at (0,0) {$x$};
   \draw[dotted] (-5,-5) rectangle (5,5);
  \end{tikzpicture}
 \right)
 := \mu_{\Tmon_\varnothing}(x) : \Tmon_\varnothing \to \Tmon_\varnothing\la2m\ra.
\]

\subsubsection{Dot morphisms}
Fix a simple reflection $s \in S$. We define
\[
 \Phi\left(
  \begin{tikzpicture}[thick,scale=0.07,baseline]
   \draw (0,5) to (0,0);
   \node at (0,0) {$\bullet$};
   \node at (0,6.7) {\tiny $s$};
  \end{tikzpicture}
 \right) := \hat\eta_s: \Tmon_\varnothing\la -1\ra \to \Tmon_s
 \quad \text{and} \quad
 \Phi\left(
  \begin{tikzpicture}[thick,scale=0.07,baseline=-5pt]
   \draw (0,-5) to (0,0);
   \node at (0,0) {$\bullet$};
   \node at (0,-6.7) {\tiny $s$};
  \end{tikzpicture}
 \right) := \hat\epsilon_s: \Tmon_s \to \Tmon_\varnothing\la 1\ra,
\]
where $\hat\eta_s$ and $\hat\epsilon_s$ are the morphisms defined in~\cite[\S5.3.4]{amrw}.

\subsubsection{Trivalent vertices}
Fix a simple reflection $s \in S$. The definition of the image of the trivalent vertices will rely on the following lemma.

\begin{lem} \label{lem:ss-comp}
 The following maps are isomorphisms:
 \begin{align*}
  \Hom_{\Dmix(\UGU,\Z')}(\Tmon_s \la1\ra, \Tmon_s \hatstar \Tmon_s) &\simto \Hom_{\Dmix(\UGU,\Z')}(\Tmon_s \la1\ra, \Tmon_s \la1\ra) \\
  f &\mapsto ( \id_{\Tmon_s} \hatstar \hat{\epsilon}_s) \circ f, \\
  \Hom_{\Dmix(\UGU,\Z')}(\Tmon_s \hatstar \Tmon_s, \Tmon_s \la-1\ra) &\simto \Hom_{\Dmix(\UGU,\Z')}(\Tmon_s \la-1\ra, \Tmon_s \la-1\ra) \\
  f &\mapsto f \circ ( \id_{\Tmon_s} \hatstar \hat{\eta}_s).
 \end{align*}
\end{lem}

\newsavebox\slowerdot
\savebox\slowerdot{%
\begin{tikzpicture}[scale=0.2,thick,baseline=-2pt]
 \draw (0,-1) to (0,0); \node at (0,0) {\tiny $\bullet$};
\end{tikzpicture}%
}
\newsavebox\supperdot
\savebox\supperdot{%
\begin{tikzpicture}[scale=0.2,thick,baseline=-2pt]
 \draw (0,0) to (0,1); \node at (0,0) {\tiny $\bullet$};
\end{tikzpicture}%
}
\newsavebox\scapmor
\savebox\scapmor{%
\begin{tikzpicture}[scale=0.1,thick,baseline=0pt]
 \draw (-1,0) to[out=90, in=180] (0,2) to[out=0, in=90] (1,0);
\end{tikzpicture}%
}
\newsavebox\scupmor
\savebox\scupmor{%
\begin{tikzpicture}[scale=0.1,thick,baseline=0pt]
 \draw (-1,2) to[out=-90, in=180] (0,0) to[out=0, in=-90] (1,2);
\end{tikzpicture}%
}
\newsavebox\sinvymor
\savebox\sinvymor{%
\begin{tikzpicture}[scale=0.1,thick,baseline=-2pt]
 \draw (-1,-1) -- (0,0) -- (1,-1); \draw (0,0) -- (0,1);
\end{tikzpicture}%
}
\newsavebox\symor
\savebox\symor{%
\begin{tikzpicture}[scale=0.1,thick,baseline=-2pt]
 \draw (-1,1) -- (0,0) -- (1,1); \draw (0,0) -- (0,-1);
\end{tikzpicture}%
}
\newsavebox\slinemor
\savebox\slinemor{%
\begin{tikzpicture}[scale=0.1,thick,baseline=-2pt]
 \draw (0,-1) -- (0,1);
\end{tikzpicture}%
}
\newsavebox\slineline
\savebox\slineline{%
\begin{tikzpicture}[scale=0.1,thick,baseline=-2pt]
 \draw (-0.5,-1) -- (-0.5,1); \draw (0.5,-1) -- (0.5,1);
\end{tikzpicture}%
}
\newsavebox\slinelowerdot
\savebox\slinelowerdot{%
\begin{tikzpicture}[scale=0.2,thick,baseline=-2pt]
 \draw (-0.5,-1) to (-0.5,1);
 \draw (0,-1) to (0,0); \node at (0,0) {\tiny $\bullet$};
\end{tikzpicture}%
}
\newsavebox\slineupperdot
\savebox\slineupperdot{%
\begin{tikzpicture}[scale=0.2,thick,baseline=-2pt]
 \draw (-0.5,-1) to (-0.5,1);
 \draw (0,0) to (0,1); \node at (0,0) {\tiny $\bullet$};
\end{tikzpicture}%
}
\newsavebox\slowerdotline
\savebox\slowerdotline{%
\begin{tikzpicture}[scale=0.2,thick,baseline=-2pt]
 \draw (0,-1) to (0,0); \node at (0,0) {\tiny $\bullet$};
 \draw (0.5,-1) to (0.5,1);
\end{tikzpicture}%
}
\newsavebox\supperdotline
\savebox\supperdotline{%
\begin{tikzpicture}[scale=0.2,thick,baseline=-2pt]
 \draw (0,0) to (0,1); \node at (0,0) {\tiny $\bullet$};
 \draw (0.5,-1) to (0.5,1);
\end{tikzpicture}%
}

Before proving this lemma, we require some preparatory work in the right equivariant category
\[
\Dmix(\UGBold, \bk) = \Kb \ParityBSp(\UGBold, \bk).
\]
For simplicity, write $\ForFMRE = \ForLMRE \circ \ForFMLM$, and set
\[
 \cF_s := \ForFMRE(\Tmon_s \hatstar \Tmon_s).
\]
Recall the right equivariant complex $\cT'_s = \ForLMRE(\cT_s)$ from~\cite[Example~4.3.4]{amrw}. We define morphisms
\[
 \phi_s : \cF_s \to \cT'_s\la1\ra \oplus \cT'_s\la-1\ra
 \quad \text{and} \quad
 \psi_s : \cT'_s\la1\ra \oplus \cT'_s\la-1\ra \to \cF_s
\]
in $\Dmix(\UGBold, \bk)$ as follows.
From the definitions we see that $\cF_s$ is given by the following complex in degrees $-2$ to $2$, where we omit direct sum signs, and we silently pass through the equivalence~\eqref{eqn:equivalence-rw}:
\[
 \begin{tikzcd}[ampersand replacement=\&, column sep=1.6cm]
  \begin{smallmatrix} \cE_\varnothing\{-2\} \end{smallmatrix}
  \ar[r, "{\left[\begin{smallmatrix} \usebox\supperdot \\ \usebox\supperdot \end{smallmatrix}\right]}"]
  \&
  \begin{smallmatrix}\cE_s\{-1\} \\ \cE_s\{-1\}\end{smallmatrix}
  \ar[r, "{\left[\begin{smallmatrix} \usebox\slowerdot & 0 \\ - \hspace{-1pt}\usebox\slineupperdot & \usebox\supperdotline \\ 0 & \usebox\slowerdot \end{smallmatrix}\right]}"]
  \&
  \begin{smallmatrix}\cE_\varnothing \\ \cE_{ss} \\ \cE_\varnothing\end{smallmatrix}
  \ar[r, "{\left[\begin{smallmatrix} \usebox\supperdot & \usebox\slowerdotline & 0 \\ -2\hspace{-3pt}\usebox\supperdot & -\hspace{-1pt}\usebox\slinelowerdot & \usebox\supperdot \end{smallmatrix}\right]}"]
  \&
  \begin{smallmatrix}\cE_s\{1\} \\ \cE_s\{1\}\end{smallmatrix}
  \ar[r, "{\left[\begin{smallmatrix} \usebox\slowerdot & \usebox\slowerdot \end{smallmatrix}\right]}"]
  \&
  \begin{smallmatrix} \cE_\varnothing\{2\}. \end{smallmatrix}
 \end{tikzcd}
\]
We also depict $\cT'_s\la1\ra \oplus \cT'_s\la-1\ra$ as the following complex in degrees $-2$ to $2$:
\[
 \begin{tikzcd}[ampersand replacement=\&, column sep=1.5cm]
  \cE_\varnothing\{-2\}
  \ar[r, "{\left[\begin{smallmatrix} \usebox\supperdot \end{smallmatrix}\right]}"]
  \&
  \cE_s\{-1\}
  \ar[r, "{\left[\begin{smallmatrix} 0 \\ \usebox\slowerdot \end{smallmatrix}\right]}"]
  \&
  \begin{smallmatrix}\cE_\varnothing \\ \cE_\varnothing\end{smallmatrix}
  \ar[r, "{\left[\begin{smallmatrix} \usebox\supperdot & 0 \end{smallmatrix}\right]}"]
  \&
  \cE_s\{1\}
  \ar[r, "{\left[\begin{smallmatrix} \usebox\slowerdot \end{smallmatrix}\right]}"]
  \&
  \cE_\varnothing\{2\}.
 \end{tikzcd}
\]
Now, let $\phi_s$ and $\psi_s$ be the morphisms represented by the following chain maps:
\[
 \begin{tikzcd}[ampersand replacement=\&, column sep=1.5cm]
  \cE_\varnothing\{-2\} \ar[r]
  \ar[d, "{\left[\begin{smallmatrix} 1 \end{smallmatrix}\right]}"]
  \&
  \begin{smallmatrix}\cE_s\{-1\} \\ \cE_s\{-1\}\end{smallmatrix} \ar[r]
  \ar[d, "{\left[\begin{smallmatrix} \usebox\slinemor & 0 \end{smallmatrix}\right]}"]
  \&
  \begin{smallmatrix}\cE_\varnothing \\ \cE_{ss} \\ \cE_\varnothing\end{smallmatrix} \ar[r]
  \ar[d, "{\left[\begin{smallmatrix} -1 & -\usebox\scapmor & 1 \\ 1 & 0 & 0 \end{smallmatrix}\right]}"]
  \&
  \begin{smallmatrix}\cE_s\{1\} \\ \cE_s\{1\}\end{smallmatrix} \ar[r]
  \ar[d, "{\left[\begin{smallmatrix} \usebox\slinemor & \usebox\slinemor \end{smallmatrix}\right]}"]
  \&
  \cE_\varnothing\{2\}
  \ar[d, "{\left[\begin{smallmatrix} 1 \end{smallmatrix}\right]}"] \\
  \cE_\varnothing\{-2\} \ar[r]
  \&
  \cE_s\{-1\} \ar[r]
  \&
  \begin{smallmatrix}\cE_\varnothing \\ \cE_\varnothing\end{smallmatrix} \ar[r]
  \&
  \cE_s\{1\} \ar[r]
  \&
  \cE_\varnothing\{2\},
 \end{tikzcd}
\]
\[
 \begin{tikzcd}[ampersand replacement=\&, column sep=1.5cm]
  \cE_\varnothing\{-2\} \ar[r]
  \&
  \begin{smallmatrix}\cE_s\{-1\} \\ \cE_s\{-1\}\end{smallmatrix} \ar[r]
  \&
  \begin{smallmatrix}\cE_\varnothing \\ \cE_{ss} \\ \cE_\varnothing\end{smallmatrix} \ar[r]
  \&
  \begin{smallmatrix}\cE_s\{1\} \\ \cE_s\{1\}\end{smallmatrix} \ar[r]
  \&
  \cE_\varnothing\{2\} \\
  \cE_\varnothing\{-2\} \ar[r]
  \ar[u, "{\left[\begin{smallmatrix} 1 \end{smallmatrix}\right]}"]
  \&
  \cE_s\{-1\} \ar[r]
  \ar[u, "{\left[\begin{smallmatrix} \usebox\slinemor \\ \usebox\slinemor \end{smallmatrix}\right]}"]
  \&
  \begin{smallmatrix}\cE_\varnothing \\ \cE_\varnothing\end{smallmatrix} \ar[r]
  \ar[u, "{\left[\begin{smallmatrix} 0 & 1 \\ 0 & -\usebox\scupmor \\ 1 & 1 \end{smallmatrix}\right]}"]
  \&
  \cE_s\{1\} \ar[r]
  \ar[u, "{\left[\begin{smallmatrix} 0 \\ \usebox\slinemor \end{smallmatrix}\right]}"]
  \&
  \cE_\varnothing\{2\}.
  \ar[u, "{\left[\begin{smallmatrix} 1 \end{smallmatrix}\right]}"]
 \end{tikzcd}
\]
Of course, one needs to check that these are indeed chain maps. In this calculation, for the component $\cE_{ss} \leadsto \cE_s\{1\}$, resp.~$\cE_s\{-1\} \leadsto \cE_{ss}$, one uses the equality
\newsavebox\linelowerdot
\savebox\linelowerdot{%
\begin{tikzpicture}[scale=0.5,thick,baseline=-2pt]
 \draw (-0.5,-1) to (-0.5,1);
 \draw (0,-1) to (0,0); \node at (0,0) {\tiny $\bullet$};
\end{tikzpicture}%
}
\newsavebox\lineupperdot
\savebox\lineupperdot{%
\begin{tikzpicture}[scale=0.5,thick,baseline=-2pt]
 \draw (-0.5,-1) to (-0.5,1);
 \draw (0,0) to (0,1); \node at (0,0) {\tiny $\bullet$};
\end{tikzpicture}%
}
\newsavebox\lowerdotline
\savebox\lowerdotline{%
\begin{tikzpicture}[scale=0.5,thick,baseline=-2pt]
 \draw (0,-1) to (0,0); \node at (0,0) {\tiny $\bullet$};
 \draw (0.5,-1) to (0.5,1);
\end{tikzpicture}%
}
\newsavebox\upperdotline
\savebox\upperdotline{%
\begin{tikzpicture}[scale=0.5,thick,baseline=-2pt]
 \draw (0,0) to (0,1); \node at (0,0) {\tiny $\bullet$};
 \draw (0.5,-1) to (0.5,1);
\end{tikzpicture}%
}
\newsavebox\capupperdot
\savebox\capupperdot{%
\begin{tikzpicture}[scale=0.5,thick,baseline=-2pt]
 \draw (0,0.25) -- (0,1); \draw node at (0,0.25) {\tiny $\bullet$};
 \draw (-0.5,-1) to[out=90, in=180] (0,-0.25) to[out=0, in=90] (0.5,-1);
\end{tikzpicture}%
}
\newsavebox\lowerdotcup
\savebox\lowerdotcup{%
\begin{tikzpicture}[scale=0.5,thick,baseline=-2pt]
 \draw (0,-0.25) -- (0,-1); \draw node at (0,-0.25) {\tiny $\bullet$};
 \draw (-0.5,1) to[out=-90, in=180] (0,0.25) to[out=0, in=-90] (0.5,1);
\end{tikzpicture}%
}
\[
 -\usebox\capupperdot - \usebox\lowerdotline + \usebox\linelowerdot = 0, \quad \text{resp.~} -\usebox\lineupperdot + \usebox\upperdotline + \usebox\lowerdotcup = 0,
\]
in $\ParityBSp(\UGBold, \Z')$.

Then we choose some lifts
 \[
  \hat{\phi}_s : \Tmon_s \hatstar \Tmon_s \simto \Tmon_s\la1\ra \oplus \Tmon_s\la-1\ra
  \quad \text{and} \quad
  \hat{\psi}_s : \Tmon_s\la1\ra \oplus \Tmon_s\la-1\ra \simto \Tmon_s \hatstar \Tmon_s
 \]
of $\phi_s$ and $\psi_s$ to $\Dmix(\UGU,\bk)$. (The existence of such lifts is guaranteed by Proposition~\ref{prop:morph-Tilt-UGU}. One can check that $\phi_s$ and $\psi_s$ are mutually inverse isomorphisms, and deduce that $\hat{\phi}_s$ and $\hat{\psi}_s$ can be chosen to be mutually inverse isomorphisms; but we will not need these facts.)

We are now ready to prove Lemma~\ref{lem:ss-comp}.

\begin{proof}[Proof of Lemma~{\rm \ref{lem:ss-comp}}]
It follows from Proposition~\ref{prop:morph-Tilt-UGU}, Lemma~\ref{lem:tilting-hom-formula}, and an easy Hecke algebra calculation that all four $\Z'$-modules in the statement of the lemma are free of rank 1. 
Hence to prove that our maps are isomorphisms it suffices to prove that they are surjective, and for this it suffices to prove that the compositions
\begin{gather*}
  \Tmon_s\la-1\ra \xrightarrow{\id_{\Tmon_s} \hatstar \hat{\eta}_s} \Tmon_s \hatstar \Tmon_s \xrightarrow{\hat{\phi}_s} \Tmon_s\la1\ra \oplus \Tmon_s\la-1\ra \xrightarrow{\hat{p}} \Tmon_s\la-1\ra, \\
  \Tmon_s\la1\ra \xrightarrow{\hat{\imath}} \Tmon_s\la1\ra \oplus \Tmon_s\la-1\ra \xrightarrow{\hat{\psi}_s} \Tmon_s \hatstar \Tmon_s \xrightarrow{\id_{\Tmon_s} \hatstar \hat{\epsilon}_s} \Tmon_s\la1\ra,
\end{gather*}
where $\hat{\imath}$ (resp.~$\hat{p}$) is the inclusion (resp.~projection), are the identity maps.

For this, note that by Proposition~\ref{prop:morph-Tilt-UGU} and Lemma~\ref{lem:tilting-hom-formula}, $\ForFMLM$ induces an isomorphism
\[
 \End_{\Dmix(\UGU, \Z')}(\Tmon_s) \simto \End_{\Dmix(\UGB, \Z')}(\cT_s).
\]
Hence one may check the claim after applying $\ForFMRE$, i.e.~show that the compositions
\begin{gather*}
 \cT'_s\la-1\ra \xrightarrow{\ForFMRE(\id_{\Tmon_s} \hatstar \hat{\eta}_s)} \cF_s \xrightarrow[\sim]{\phi_s} \cT'_s \la1\ra \oplus \cT'_s \la-1\ra \xrightarrow{p} \cT'_s\la-1\ra, \\
 \cT'_s\la1\ra \xrightarrow{i} \cT'_s\la1\ra \oplus \cT'_s\la-1\ra \xrightarrow[\sim]{\psi_s} \cF_s \xrightarrow{\ForFMRE(\id_{\Tmon_s} \hatstar \hat{\epsilon}_s)} \cT'_s \la1\ra,
\end{gather*}
where $i$ (resp.~$p$) is the inclusion (resp.~projection), are the identity maps. Depicting $\cF_s$ and $\cT'_s\la1\ra \oplus \cT'_s\la-1\ra$ as for the definition of $\phi_s$ and $\psi_s$, the morphisms
\[
 \ForFMRE(\id_{\Tmon_s} \hatstar \hat{\eta}_s) : \cT'_s\la-1\ra \to \cF_s \quad \text{and} \quad \ForFMRE(\id_{\Tmon_s} \hatstar \hat{\epsilon}_s) : \cF_s \to \cT'_s\la1\ra
\]
are represented by the following chain maps:
\[
 \begin{tikzcd}[ampersand replacement=\&, column sep=1.5cm]
  \& \&
  \cE_\varnothing \ar[r]
  \ar[d, "{\left[\begin{smallmatrix} 0 \\ 0 \\ 1 \end{smallmatrix}\right]}"]
  \&
  \cE_s\{1\} \ar[r]
  \ar[d, "{\left[\begin{smallmatrix} 0 \\ \usebox\slinemor \end{smallmatrix}\right]}"]
  \&
  \cE_\varnothing\{2\}
  \ar[d, "{\left[\begin{smallmatrix} 1 \end{smallmatrix}\right]}"] \\
  \cE_\varnothing\{-2\} \ar[r]
  \&
  \begin{smallmatrix}\cE_s\{-1\} \\ \cE_s\{-1\}\end{smallmatrix} \ar[r]
  \&
  \begin{smallmatrix}\cE_\varnothing \\ \cE_{ss} \\ \cE_\varnothing\end{smallmatrix} \ar[r]
  \&
  \begin{smallmatrix}\cE_s\{1\} \\ \cE_s\{1\}\end{smallmatrix} \ar[r]
  \&
  \cE_\varnothing\{2\},
 \end{tikzcd}
\]
\[
 \begin{tikzcd}[ampersand replacement=\&, column sep=1.5cm]
  \cE_\varnothing\{-2\} \ar[r]
  \ar[d, "{\left[\begin{smallmatrix} 1 \end{smallmatrix}\right]}"]
  \&
  \begin{smallmatrix}\cE_s\{-1\} \\ \cE_s\{-1\}\end{smallmatrix} \ar[r]
  \ar[d, "{\left[\begin{smallmatrix} \usebox\slinemor & 0 \end{smallmatrix}\right]}"]
  \&
  \begin{smallmatrix}\cE_\varnothing \\ \cE_{ss} \\ \cE_\varnothing\end{smallmatrix} \ar[r]
  \ar[d, "{\left[\begin{smallmatrix} 1 & 0 & 0 \end{smallmatrix}\right]}"]
  \&
  \begin{smallmatrix}\cE_s\{1\} \\ \cE_s\{1\}\end{smallmatrix} \ar[r]
  \&
  \cE_\varnothing\{2\} \\
  \cE_\varnothing\{-2\} \ar[r]
  \&
  \cE_s\{-1\} \ar[r]
  \&
  \cE_\varnothing.
  \&
  \&
 \end{tikzcd}
\]
Then the desired claim follows from the explicit description of the chain maps representing $\phi_s$ and $\psi_s$.
\end{proof}

By Lemma~\ref{lem:ss-comp}, we may define
\begin{gather*}
\hat{b}_1 : \Tmon_s \to \Tmon_s \hatstar \Tmon_s\la-1\ra
\quad \text{and} \quad 
\hat{b}_2 : \Tmon_s \hatstar \Tmon_s \to \Tmon_s\la-1\ra
\end{gather*}
to be the unique morphisms satisfying
\begin{equation} \label{eqn:trivalent-defn}
  (\id_{\Tmon_s} \hatstar \hat{\epsilon}_s) \circ  \hat{b}_1  = \id_{\Tmon_s}
  \quad \text{and} \quad
  \hat{b}_2 \circ ( \id_{\Tmon_s} \hatstar \hat{\eta}_s) = \id_{\Tmon_s}.
\end{equation}
We now define
\[
 \Phi\left(
  \begin{tikzpicture}[thick,baseline=-2pt,scale=0.07]
   \draw (-4,5) to (0,0) to (4,5);
   \draw (0,-5) to (0,0);
   \node at (0,-6.7) {\tiny $s$};
   \node at (-4,6.7) {\tiny $s$};
   \node at (4,6.7) {\tiny $s$};      
  \end{tikzpicture}
  \right) := \hat{b}_1
  \quad \text{and} \quad
 \Phi\left(
  \begin{tikzpicture}[thick,baseline=-2pt,scale=-0.07]
   \draw (-4,5) to (0,0) to (4,5);
   \draw (0,-5) to (0,0);
   \node at (0,-6.7) {\tiny $s$};
   \node at (-4,6.7) {\tiny $s$};
   \node at (4,6.7) {\tiny $s$};
  \end{tikzpicture}
  \right) := \hat{b}_2.
\]

\subsubsection{$2m_{st}$-valent vertices}

Fix $s, t \in S$ such that $st \in W$ has finite order. Let $m_{st}$ be this order, and let $\hs = (s,t, \ldots)$ and $\htt = (t, s, \ldots)$, where both sequences have $m_{st}$ elements. Finally, let $w:=st \ldots = ts \ldots$ (with $m_{st}$ elements in both products).

Instead of using as a starting point the category $\Db(\BGB, \Z')$, one can consider the same categories as those constructed in~\cite{amrw} but starting with the $\BKM$-equivariant derived category of $\Z'$-sheaves on $\cX \smallsetminus \partial \cX_w$, where $\partial \cX_w := \overline{\cX_w} \smallsetminus \cX_w$. In this way one obtains a ``free-monodromic'' category $\Dmix_w(\UGU, \Z')$ and a functor
\[
q^w : \Dmix(\UGU, \Z') \to \Dmix_w(\UGU, \Z')
\]
induced by pullback along the open embedding $\cX \smallsetminus \partial \cX_w \hookrightarrow \cX$.
Note that every term except $\cE_\hs$ (resp.~$\cE_\htt$) in the underlying sequence of parity complexes of $\Tmon_\hs$ (resp.~$\Tmon_\htt$) restricts to $0$ on $\cX \smallsetminus \partial \cX_w$,
and that the differential of $\Tmon_\hs$ (resp.~$\Tmon_\htt$) has component
$\cE_\hs \leadsto \cE_\hs$, resp.~$\cE_\htt \leadsto \cE_\htt$, given by $\sum w(e_i) \otimes \id \otimes \check e_i$
in the notation of~\cite{amrw}. (Here, $(e_1, \ldots, e_r)$ is a basis of $V^*$, and $(\check e_1, \ldots, \check e_r)$ is the dual basis of $V$.)
Since the restrictions of both $\cE_\hs$ and $\cE_\htt$ to $\cX_w$ are canonically isomorphic to the constant sheaf, we deduce that the functor $q^w$
sends $\Tmon_\hs$ and $\Tmon_\htt$ to canonically isomorphic objects.

\begin{lem} \label{lem:2mst-isom-mod-lower-terms}
 The functor $q^w$ induces isomorphisms
 \begin{align*}
  \Hom_{\Dmix(\UGU, \Z')}(\Tmon_\hs, \Tmon_\htt) &\simto \Hom_{\Dmix_w(\UGU, \Z')}(q^w(\Tmon_\hs), q^w(\Tmon_\htt)), \\
  \Hom_{\Dmix(\UGU, \Z')}(\Tmon_\htt, \Tmon_\hs) &\simto \Hom_{\Dmix_w(\UGU, \Z')}(q^w(\Tmon_\htt), q^w(\Tmon_\hs)).
 \end{align*}
Moreover, all of these spaces are free $\Z'$-modules of rank $1$.
\end{lem}

\begin{proof}
 By symmetry, we only need to consider the first map. Let us first show that both sides are free $\Z'$-modules of rank 1. For the left-hand side, this follows from Lemma~\ref{lem:tilting-hom-formula} and a standard computation in the Hecke algebra. For the right-hand side, this follows from the description of $q^w(\Tmon_\hs) \cong q^w(\Tmon_\htt)$ above and the definition of morphisms in $\Dmix_w(\UGU, \Z')$ (see~\cite[\S 5.3.1]{amrw} for a similar computation). To show that the first map is an isomorphism, it therefore suffices to show that it is nonzero after extension of scalars from $\Z'$ to any field $\bk$. The map so obtained may be identified with the map
 \[
  \Hom_{\Dmix(\UGU, \bk)}(\Tmon_\hs, \Tmon_\htt) \to \Hom_{\Dmix_w(\UGU, \bk)}(q^w(\Tmon_\hs), q^w(\Tmon_\htt))
 \]
 defined in the same way, using coefficients $\bk$ instead of $\Z'$.
 
 For field coefficients, by~ \cite[Theorem~10.7.1]{amrw} there are direct sum decompositions
 \begin{align*}
  \Tmon_\hs \cong \Tmon_w \oplus \text{(lower terms)} \quad \text{and} \quad  \Tmon_\htt \cong \Tmon_w \oplus \text{(lower terms)},
 \end{align*}
 where the lower terms restrict to $0$ on $\cX \smallsetminus \partial \cX_w$. Fixing these decompositions, the composition $\Tmon_\hs \twoheadrightarrow \Tmon_w \hookrightarrow \Tmon_\htt$ (projection to $\Tmon_w$ followed by inclusion) defines a morphism $\Tmon_\hs \to \Tmon_\htt$ that remains nonzero (in fact, an isomorphism) on $\cX \smallsetminus \partial \cX_w$.
\end{proof}

We now define
\[
 \hat{g}_{s,t} : \Tmon_\hs \to \Tmon_\htt
 \quad \text{and} \quad
 \hat{g}_{t,s} : \Tmon_\htt \to \Tmon_\hs
\]
to be the unique morphisms that are sent to the canonical isomorphism $q^w(\Tmon_\hs) \cong q^w(\Tmon_\htt)$ considered above under the isomorphisms of Lemma~\ref{lem:2mst-isom-mod-lower-terms}.

To define the morphisms $\hat{f}_{s,t}$, $\hat{f}_{t,s}$ that will be the image of the $2m_{st}$-valent vertices, we need another lemma. For any expression $\uw = (s_1, \ldots, s_{\ell(\uw)})$, define
\[
 \hat{\epsilon}_\uw := \hat{\epsilon}_{s_1} \hatstar \cdots \hatstar {}\hat{\epsilon}_{s_{\ell(\uw)}} : \Tmon_\uw \to \Tmon_\varnothing\la \ell(\uw) \ra,
\]
where the maps $\hat{\epsilon}_{s_i}$ are defined in~\cite[\S 5.3.4]{amrw}.

\begin{lem} \label{lem:epsilon-gen}
 For $u \in \{s, t\}$, we have
 \[
  \Hom_{\Dmix(\UGB,\Z')}(\cT_\hu, \cT_\varnothing\la m_{st}\ra) = \Z' \cdot \ForFMLM(\hat{\epsilon}_\hu).
 \]
\end{lem}
\begin{proof}
 The space in question is free of rank 1 over $\Z'$ by Lemma~\ref{lem:tilting-hom-formula} and a straightforward calculation in the Hecke algebra, so it is enough to show that $\ForFMLM(\hat{\epsilon}_\hu)$ remains nonzero after extension of scalars to any field. This may be checked after applying $\ForLMRE$, i.e.~in the right equivariant category 
$\Dmix(\UGBold, \bk)$.
 We see from the definitions that the morphism $\ForFMRE(\hat{\epsilon}_\hu)$ may be represented by the following chain map, where $v$ is the simple reflection different from $u$:
 \[
  \begin{tikzcd}[column sep=huge]
   \vdots & \\
   \cE_u\{1 - m_{st}\} \oplus \cE_v\{1 - m_{st}\} \oplus \cdots \ar[u] \ar[rd, dashed] \\
   \cE_\varnothing\{-m_{st}\} \ar[u] \ar[r, "(-1)^{m_{st}} \cdot\id"] & \cE_\varnothing\{-m_{st}\}
  \end{tikzcd}
 \]
 Here, the left-hand column depicts the complex $\ForFMRE(\Tmon_\hu)$ in chain degrees $-m_{st}, 1-m_{st}, \ldots$ (the lowest chain degrees where it is nonzero); the right-hand column depicts $\ForFMRE(\Tmon_\varnothing\la m_{st}\ra)$, which is $\cE_\varnothing\{-m_{st}\}$ concentrated in chain degree $-m_{st}$; and the chain map has a single nonzero component
 \[
  (\Tmon_\hu)^{-m_{st}} = \cE_\varnothing\{-m_{st}\} \xrightarrow{(-1)^{m_{st}} \cdot\id} \cE_\varnothing\{-m_{st}\} = (\Tmon_\varnothing\la m_{st}\ra)^{-m_{st}}.
 \]
 There is no nonzero homotopy (dashed arrow) for degree reasons, so $\ForFMRE(\hat{\epsilon}_\hu)$ is nonzero.
\end{proof}

By Lemma~\ref{lem:epsilon-gen}, we have
\begin{equation} \label{eqn:epsilon-g-const}
 \ForFMLM(\hat{\epsilon}_\htt \circ \hat{g}_{s,t}) = c_{s,t}\ForFMLM(\hat{\epsilon}_\hs)
 \quad \text{and} \quad
 \ForFMLM(\hat{\epsilon}_\hs \circ \hat{g}_{t,s}) = c_{t,s}\ForFMLM(\hat{\epsilon}_\htt)
\end{equation}
for some $c_{s,t}, c_{t,s} \in \Z'$. We now set
\begin{equation} \label{eqn:2mst-defn}
 \begin{gathered}
 \hat{f}_{s,t} := c_{t,s}\hat{g}_{s,t} : \Tmon_\hs \to \Tmon_\htt
 \quad \text{and} \quad
 \hat{f}_{t,s} := c_{s,t}\hat{g}_{t,s} : \Tmon_\htt \to \Tmon_\hs,
 \end{gathered}
\end{equation}
and define these to be the image of the $2m_{st}$-valent vertices under $\Phi$:
\[
\Phi\left(
\begin{tikzpicture}[yscale=0.5,xscale=0.3,baseline,thick]
\draw (-2.5,-1) to (0,0) to (-1.5,1);
\draw (-0.5,-1) to (0,0) to (0.5,1);
\draw (1.5,-1) to (0,0) to (2.5,1);
\draw[red] (-1.5,-1) to (0,0) to (-2.5,1);
\draw[red] (0.5,-1) to (0,0) to (-0.5,1);
\draw[red] (2.5,-1) to (0,0) to (1.5,1);
\node at (-2.5,-1.3) {\tiny $s$\vphantom{$t$}};
\node at (-1.5,1.3) {\tiny $s$\vphantom{$t$}};
\node at (-0.5,-1.3) {\tiny $\cdots$};
\node at (-1.5,-1.3) {\tiny $t$};
\node at (-2.5,1.3) {\tiny $t$};
\node at (-0.5,1.3) {\tiny $\cdots$\vphantom{$t$}};
\end{tikzpicture}
\right) := \hat{f}_{s,t}
\quad \text{and} \quad
\Phi\left(
\begin{tikzpicture}[yscale=0.5,xscale=0.3,baseline,thick]
\draw[red] (-2.5,-1) to (0,0) to (-1.5,1);
\draw[red] (-0.5,-1) to (0,0) to (0.5,1);
\draw[red] (1.5,-1) to (0,0) to (2.5,1);
\draw (-1.5,-1) to (0,0) to (-2.5,1);
\draw (0.5,-1) to (0,0) to (-0.5,1);
\draw (2.5,-1) to (0,0) to (1.5,1);
\node at (-2.5,-1.3) {\tiny $t$};
\node at (-1.5,1.3) {\tiny $t$};
\node at (-0.5,-1.3) {\tiny $\cdots$\vphantom{$t$}};
\node at (-1.5,-1.3) {\tiny $s$\vphantom{$t$}};
\node at (-2.5,1.3) {\tiny $s$\vphantom{$t$}};
\node at (-0.5,1.3) {\tiny $\cdots$\vphantom{$t$}};
\end{tikzpicture}
\right) := \hat{f}_{t,s}.
\]

\subsection{Verification of the relations}
\label{ss:rels-verif}

In this subsection, we verify that the morphisms defined in~\S\ref{ss:rels-defn} satisfy the relations from~\cite[\S\S1.4.1--1.4.3]{ew}.

Each relation only involves a subset $S'$ of $S$ (of cardinality at most 3) that generates a finite subgroup $W'$ of $W$. Fix a relation and the corresponding subset $S'$. Consider the realization
\[
 \fh^*_{S', \Z'} = (V^*, \{ \alpha_s \}_{s \in S'}, \{ \alpha_s^\vee \}_{s \in S'})
\]
of $(W', S')$ over $\Z'$, and let $\LKM_{S'}$ be the Levi subgroup of $\GKM$ associated with $S'$ (a connected reductive group with Weyl group $W'$). Set also $\UKM_{S'}:=\UKM \cap \LKM_{S'}$; then we can consider the category $\TiltBS(\UKM_{S'} \fatbslash \LKM_{S'} \fatslash \UKM_{S'},\Z')$. There are obvious fully faithful monoidal functors
\begin{multline*}
 \DiagBS(\fh^*_{S', \Z'}, W') \to \DiagBS(\fh^*_{\Z'}, W) \quad \text{and} \\
 \TiltBS(\UKM_{S'} \fatbslash \LKM_{S'} \fatslash \UKM_{S'},\Z') \to \TiltBS(\UGU, \Z'),
\end{multline*}
and the definitions of all our morphisms are identical whether considered in the category $ \TiltBS(\UKM_{S'} \fatbslash \LKM_{S'} \fatslash \UKM_{S'},\Z')$ or $\TiltBS(\UGU, \Z')$ (in particular, the constants $c_{s,t}$ are unchanged by this replacement),
so that it suffices to verify the relation for the group $\LKM_{S'}$.
We may therefore assume from the start that $A$ is a finite type Cartan matrix. Moreover, by the last statement of Proposition~\ref{prop:morph-Tilt-UGU}, we may check the relation after extension of scalars along the map $\Z' \to \Q$. 

As a further reduction, we may check the relation after passing to the Karoubian envelope of the additive closure. From now on, we work in $\Tilt(\UGU, \Q)$, where the results of~\S\ref{ss:Pmon} are available: fix an object $\Pmon$ and a nonzero morphism $\xi : \Pmon \to \Tmon_1$, and use these to define a functor $\V$ and the various other structures from Section~\ref{sec:functor-V}. We may then check the relation in the category of graded $R^\vee$-bimodules, after applying the fully faithful functor $\V$.

To do this, we compute the image of the generating morphisms under $\V$. For $s \in S$, define $B^\bim_s := R^\vee \otimes_{(R^\vee)^s} R^\vee \la1\ra$. For any expression $\uw = (s_1, s_2, \ldots, s_{\ell(\uw)})$ in $S$, define
\[
 B^\bim_\uw := B^\bim_{s_1} \otimes_{R^\vee} \cdots \otimes_{R^\vee} B^\bim_{s_{\ell(\uw)}} = R^\vee \otimes_{(R^\vee)^{s_1}} \cdots \otimes_{(R^\vee)^{s_{\ell(\uw)}}} R^\vee \la\ell(\uw)\ra.
\]
We identify $B^\bim_\uw$ with $\V(\Tmon_\uw)$ via an isomorphism
\[
 \gamma_\uw : B^\bim_\uw \simto \V(\Tmon_\uw)
\]
defined as follows. For $\uw = \varnothing$, we set $\gamma_\varnothing = \gamma$, the isomorphism from the proof of Proposition~\ref{prop:V-monoidal}. Otherwise, define $\gamma_\uw$ to be the composition
\begin{multline*}
 B^\bim_{s_1} \otimes_{R^\vee} \cdots \otimes_{R^\vee} B^\bim_{s_{\ell(\uw)}}
 \xrightarrow[\sim]{\gamma_{s_1} \otimes \cdots \otimes \gamma_{s_{\ell(\uw)}}}
 \V(\Tmon_{s_1}) \otimes_{R^\vee} \cdots \otimes_{R^\vee} \V(\Tmon_{s_{\ell(\uw)}}) \\
 \xrightarrow[\sim]{\beta} \V(\Tmon_{s_1} \hatstar \cdots \hatstar \Tmon_{s_{\ell(\uw)}}),
\end{multline*}
where for $s \in S$, $\gamma_s$ is the isomorphism of Lemma~\ref{lem:V-Ts}, and $\beta$ is defined as in the proof of Proposition~\ref{prop:V-monoidal}. Let
\[
 \zeta_\uw = \gamma_\uw(1 \otimes \cdots \otimes 1) : \Pmon \to \Tmon_\uw\la \ell(\uw) \ra.
\]
Note that $\zeta_\varnothing = \xi$ and $\zeta_{(s)} = \zeta_s$ (with the notation of Lemma~\ref{lem:V-Ts}). It also follows from Lemma~\ref{lem:V-Ts} and the coalgebra axioms (see Proposition~\ref{prop:Pmon-coalgebra}) that
\begin{equation} \label{eqn:epsilon-zeta-xi}
 \hat{\epsilon}_\uw \la -{\ell(\uw)} \ra \circ \zeta_\uw = \xi.
\end{equation}

\begin{rmk}
Note that the grading on our bimodules is opposite to the ``traditional'' one from~\cite{soergel-bim}; for instance, our $B_s^\bim$ is concentrated in degrees in $1 + \Z_{\leq 0}$.
\end{rmk}

We now compute the image of our morphisms under $\V$, under the identifications $\gamma_\uw$.
\begin{enumerate}
 \item \emph{Polynomials:}
 For $x \in R^\vee$, we have
 \[
  \V(\mu_{\Tmon_\varnothing}(x)) = \mu_{\Tmon_\varnothing}(x) \circ ({-}) \overset{\eqref{eqn:mu-morph}}{=} ({-}) \circ \mu_{\Pmon}(x).
 \]
 This by definition is the left action of $x$ on $\V(\Tmon_\varnothing)$. Under the identification $\gamma : R^\vee \simto \V(\Tmon_\varnothing)$, this becomes multiplication by $x$ on $R^\vee$.

 \item \emph{The upper dot:} The space of graded $R^\vee$-bimodule homomorphisms
 $B^\bim_s \to R^\vee\la1\ra$
 is of dimension 1, with generator
\[
\mult_s: B^\bim_s \to R^\vee\la 1\ra
\qquad\text{given by}\qquad
\mult_s(f \otimes g) = fg. 
\]
Under the identifications above, the equation $(\hat{\epsilon}_s \langle -1 \rangle) \circ \zeta_s = \xi$ becomes $\V(\hat{\epsilon}_s)(1 \otimes 1) = 1$. Hence $\V(\hat{\epsilon}_s) = \mult_s$.

 \item \emph{The lower dot:} The space of graded $R^\vee$-bimodule homomorphisms
 $R^\vee \to B^\bim_s\la1\ra$ is of dimension 1, with generator
 \[
  \delta_s: R^\vee \to B^\bim_s\la 1\ra
  \qquad\text{given by}\qquad
  \delta_s(1) = \textstyle\frac12(\alpha_s^\vee \otimes 1 + 1 \otimes \alpha_s^\vee).
 \]
The map $\delta_s$ is characterized uniquely by the fact that $\mult_s \circ \delta_s = \alpha_s^\vee \cdot \id_{R^\vee}$. We computed in~\cite[Proposition~5.3.2(1)]{amrw} that $\hat{\epsilon}_s \circ \hat{\eta}_s = \mu_{\Tmon_\varnothing}(\alpha_s^\vee)$. Applying $\V$ and using the computations above, we get $\mult_s \circ \V(\hat{\eta}_s) = \alpha_s^\vee \cdot \id_{R^\vee}$. Hence $\V(\hat{\eta}_s) = \delta_s$.

 \item \emph{The trivalent vertices:} The space of graded $R^\vee$-bimodule homomorphisms
 \[
  B^\bim_s \to B^\bim_s \otimes_{R^\vee} B^\bim_s\la-1\ra, \quad \text{resp.}\quad B^\bim_s \otimes B^\bim_s \to B^\bim_s\la-1\ra,
 \]
 is of dimension 1, with generator
 \[
  t_1 : f \otimes g \mapsto f \otimes 1 \otimes g, \quad \text{resp.} \quad t_2 : f \otimes g \otimes h \mapsto f(\partial_s g) \otimes h,
 \]
 where we have identified $B^\bim_s \otimes_{R^\vee} B^\bim_s = R^\vee \otimes_{(R^\vee)^s} R^\vee \otimes_{(R^\vee)^s} R^\vee\la2\ra$. This generator is characterized uniquely by the identity
 \[
  (\mult_s\la-1\ra \otimes_{R^\vee} \id_{B^\bim_s}) \circ t_1 = \id_{B^\bim_s} \quad \text{resp.}\quad t_2\la1\ra \circ (\delta_s \otimes_{R^\vee} \id_{B^\bim_s}) = \id_{B^\bim_s}.
 \]
 Hence
 \[
  \V(\hat{b}_1) = t_1 \quad \text{and} \quad \V(\hat{b}_2) = t_2,
 \]
 as follows by applying $\V$ to the defining identities \eqref{eqn:trivalent-defn} of $\hat{b}_1, \hat{b}_2$ and using the fact that $\V(\hat{\epsilon}_s) = \mult_s$, $\V(\hat{\eta}_s) = \delta_s$.
\end{enumerate}

Before computing the image of the $2m_{st}$-valent vertices, some preparatory work is required.  For any expression $\uw = (s_1, \ldots, s_{\ell(\uw)})$, define
\[
 \mult_\uw := \mult_{s_1} \otimes_{R^\vee} \cdots \otimes_{R^\vee} {}\mult_{s_{\ell(\uw)}} : B^\bim_\uw \to R^\vee\la \ell(\uw) \ra.
\]
Next, recall from~\cite[Proposition~4.3]{libedinsky} that $\Hom_{R^\vee \grbiMod R^\vee}(B^\bim_\hs, B^\bim_\htt)$ has dimension~$1$.  An analogue of~\cite[Lemma~4.7]{libedinsky} shows that there is a unique morphism
\begin{equation}\label{eqn:bimod-lib}
j_{s,t}: B^\bim_\hs \to B^\bim_\htt
\end{equation}
that acts as the identity map in degree $m_{st}$.  Since $m_{st}$ is the largest degree in which the bimodules $B^\bim_\hs$ and $B^\bim_\htt$ have nonzero components, this condition can be rephrased as follows: $j_{s,t}$ is the unique morphism such that there is an equality of maps
\begin{equation}\label{eqn:bimod-lib-cond}
(\mult_\htt \circ j_{s,t}) \otimes_{R^\vee} \id_\bk = \mult_\hs \otimes_{R^\vee} \id_\bk:
B^\bim_\hs \otimes_{R^\vee} \bk \to R^\vee\la m_{st}\ra \otimes_{R^\vee} \bk = \bk\la m_{st}\ra.
\end{equation}

\begin{lem} \label{lem:cst-cts}
 The constants $c_{s,t}$, $c_{t,s} \in \Z'$ defined by \eqref{eqn:epsilon-g-const} satisfy
 \begin{equation} \label{eqn:cst-cts}
  c_{s,t}c_{t,s} = 1.
 \end{equation}
\end{lem}
\begin{proof}
 It follows from the definition of $\hat{g}_{s,t}$, $\hat{g}_{t,s}$ that
 \[
  \hat{g}_{s,t} \circ \hat{g}_{t,s} \circ \hat{g}_{s,t} = \hat{g}_{s,t}.
 \]
 Applying $\ForFMLM(\hat{\epsilon}_\htt \circ -)$ to both sides and using \eqref{eqn:epsilon-g-const} repeatedly, we deduce that $c_{s,t} c_{t,s} c_{s,t} = c_{s,t}$, or in other words that $c_{s,t}(c_{t,s}c_{s,t} - 1) = 0$.
 
 To conclude, it is therefore enough to show that $c_{s,t} \neq 0$. Since $\hat{g}_{s,t}$ is a generator for the space $\Hom_{\Dmix(\UGU, \Z')}(\Tmon_\hs, \Tmon_\htt)$, this would follow if the map
 \begin{align*}
  \Hom_{\Dmix(\UGU, \Z')}(\Tmon_\hs, \Tmon_\htt) &\to \Hom_{\Dmix(\UGB, \Z')}(\cT_\hs, \cT_\varnothing\la m_{st}\ra) \\
  f &\mapsto \ForFMLM(\hat{\epsilon}_\htt \circ f)
 \end{align*}
were known to be nonzero.
 
For this, we use the functors $\V, \V'$ constructed in Section~\ref{sec:functor-V}.  Under these identifications, our map becomes
 \begin{align*}
  \Hom_{R^\vee \grbiMod R^\vee}(B^\bim_\hs, B^\bim_\htt) &\to \Hom_{R^\vee \grlMod}(B^\bim_\hs \otimes_{R^\vee} \bk, (R^\vee\la m_{st}\ra) \otimes_{R^\vee} \bk) \\
  f &\mapsto (\mult_\htt \circ f) \otimes_{R^\vee} \id_\bk,
 \end{align*}
This map is clearly nonzero, since it sends $j_{s,t}$ (from~\eqref{eqn:bimod-lib}) to a nonzero element.
\end{proof}

Now we compute the image of the $2m_{st}$-valent vertices.
\begin{enumerate}
 \setcounter{enumi}{4}
 \item \emph{$2m_{st}$-valent vertices:} By~\eqref{eqn:epsilon-g-const}, \eqref{eqn:2mst-defn}, and Lemma~\ref{lem:cst-cts}, we have $\ForFMLM(\hat{\epsilon}_\htt \circ \hat{f}_{s,t}) = \ForFMLM(\hat{\epsilon_\hs})$.  Now apply the functor $\V'$, and use the commutative square~\eqref{eqn:VVprime-commute} to deduce that
 \[
 (\V(\hat{\epsilon}_\htt) \circ \V(\hat{f}_{s,t})) \otimes_{R^\vee} \bk =
 \V(\hat{\epsilon}_\hs) \otimes_{R^\vee} \bk.
\]
Since $\V(\hat{\epsilon}_\uw) = \mult_\uw$, we conclude from~\eqref{eqn:bimod-lib-cond} that $\V(\hat{f}_{s,t}) = j_{s,t}$.
\end{enumerate}

We have thus reduced the verification of the (fixed) relation to the same verification for the appropriate morphisms of graded $R^\vee$-bimodules found above. The argument in the final paragraph of \cite[\S10.5]{rw} reduces this to the same verification for a standard Cartan realization of $(W', S')$. In this case, all the relations are known to hold, as explained in~\cite[Claim~5.14]{ew}. This concludes the proof of Theorem~\ref{thm:Diag-Tilt}.

\section{Koszul duality}
\label{sec:Koszul-duality}

As in Section~\ref{sec:Diag-Tilt} we assume that $\bk$ is a Noetherian integral domain of finite global dimension such that there exists a ring morphism $\Z' \to \bk$. 

\subsection{Statement and construction of the functors}
\label{ss:Phi-equiv}

\newcommand{\monKos}{\tilde{\varkappa}}
\newcommand{\preselfKos}{\monKos'}
\newcommand{\selfKos}{\varkappa}

We begin by fixing notation related to the Langlands dual Kac--Moody group to $\GKM$.  Namely, consider the generalized Cartan matrix ${}^t \!A$, and let $\bX^* = \Hom_\Z(\bX,\Z)$.
The Kac--Moody root datum $(I, \bX^*, \{\alpha_i^\vee : i \in I\}, \{\alpha_i : i \in I\})$ determines a Kac--Moody group
$\GKM^\vee$ as in~\S\ref{ss:KM-groups}, with maximal torus $\TKM^\vee$, Borel subgroup $\BKM^\vee$ and pro-unipotent radical $\UKM^\vee$. 

Compared to the set-up of~\S\S\ref{ss:KM-groups}--\ref{ss:constructible-Dmix}, it will be convenient for us to swap the roles of constructions on the left and right when working with $\GKM^\vee$.  For instance, we define its flag variety by $\cX^\vee := \BKM^\vee\backslash\GKM^\vee$.  We will work with the monoidal category $\ParityBS(\BGBvee, \bk)$ of equivariant Bott--Samelson parity complexes on $\cX^\vee$ (and its variants).  But we also work with the \emph{left-equivariant derived category}, denoted by $\Dmix(\BGUvee,\bk)$.  To emphasize the parallel with~\S\ref{ss:constructible-Dmix}, we denote the forgetful functor by
\[
\ForBELE := \ParityBSp(\BGBvee,\bk) \to \Dmix(\BGUvee,\bk)
\]
rather than by $\mathsf{For}^{\BE}_{\mathsf{LE}}$.  This functor is compatible with the monoidal action of the former on the latter:
\[
\ForBELE(\cF \star \cG) \cong \cF \star \ForBELE(\cG).
\]
Objects in these categories will typically be denoted with a superscript ``$^\vee$'': for instance, $\cE^\vee_\uw$ or $\Delta^\vee_w$.

Our goal in this section is to prove the following theorem.

\begin{thm}[Monoidal Koszul duality]
\label{thm:monoidal-Koszul}
There is an equivalence of monoidal categories
\[
\monKos: (\ParityBS(\BGBvee,\bk), \star) \simto (\TiltBS(\UGU,\bk), \hatstar)
\]
satisfying $\monKos \circ \{1\} \cong \la 1\ra \circ \monKos$, and such that $\monKos(\cE^\vee_\uw) \cong \Tmon_\uw$.
\end{thm}

In the course of the proof, we will simultaneously establish the following result.

\begin{thm}
\label{thm:self-Koszul}
There is an equivalence of triangulated categories
\[
\selfKos: \Dmix(\BGUvee,\bk) \simto \Dmix(\UGB,\bk)
\]
satisfying $\selfKos \circ \{1\} \cong \la 1 \ra \circ \selfKos$, and such that $\selfKos(\cE^\vee_\uw) \cong \cT_\uw$. This functor is monoidal, in the sense that for $\cF \in \ParityBS(\BGBvee,\bk)$ and $\cG \in \Dmix(\BGUvee,\bk)$, there is a natural isomorphism $\selfKos(\cF \star \cG) \cong \monKos(\cF) \hatstar \selfKos(\cG)$.
\end{thm}

Note that when $\cG$ is the skyscraper sheaf $\cE^\vee_1$, the monoidal property of $\selfKos$ implies that $\selfKos(\ForBELE(\cF)) \cong \ForBELE(\monKos(\cF))$. 

The proofs of Theorems~\ref{thm:monoidal-Koszul} and~\ref{thm:self-Koszul} will be completed in~\S\ref{ss:proof-Phi-equiv}.  For now, let us explain how to define the functors $\monKos$ and $\selfKos$. As in~\eqref{eqn:equivalence-rw}, by~\cite[Theorem~10.6]{rw}, there exists a natural equivalence of monoidal categories
\begin{equation}
\label{eqn:equiv-Diag-Parity}
\Psi^\vee: \DiagBS(\fh_\bk^*,W) \simto \ParityBS(\BGBvee, \bk)
\end{equation}
intertwining the shifts $(1)$ and $\{1\}$, and sending $B^\vee_\uw$ to $\cE^\vee_\uw$.  We define
\[
\monKos := \Phi \circ (\Psi^\vee)^{-1}: \ParityBS(\BGBvee,\bk) \to \TiltBS(\UGU,\bk).
\]

Next, let $\uDiagBSp(\fh_\bk^*,W)$ be the additive category with shift $(1)$ whose objects are the same as those of $\DiagBSp(\fh_\bk^*,W)$, and whose morphism spaces are defined by
\[
\bigoplus_{n \in \Z} \Hom_{\uDiagBSp(\fh_\bk^*,W)}(M,N(n)) := \left( \bigoplus_{n \in \Z} \Hom_{\DiagBSp(\fh_\bk^*,W)}(M,N(n)) \right) \otimes_{R^\vee} \bk.
\]
(This notation should not be confused with the notation $\oDiagBSp$ used in~\cite{amrw}, where the \emph{left} action of polynomials is killed.)
Then~\eqref{eqn:equiv-Diag-Parity} induces an equivalence of additive categories
\begin{equation}
\label{eqn:equiv-Diag-Parity-const}
\underline{\Psi}^\vee: \uDiagBSp(\fh_\bk^*,W) \simto \ParityBSp(\BGUvee, \bk),
\end{equation}

Similarly, by Proposition~\ref{prop:morph-Tilt-UGU} the composition $\ForFMLM \circ \Phi$ factors through a functor
\[
\underline{\Phi} : \uDiagBSp(\fh_\bk^*,W) \to \TiltBSp(\UGB, \bk),
\]
so that we can consider the functor
\[
\preselfKos := \underline{\Phi} \circ (\underline{\Psi}^\vee)^{-1}: \ParityBSp(\BGUvee, \bk) \to \TiltBSp(\UGB, \bk).
\]
Finally, we define $\selfKos$ to be the composition
\begin{multline}
\label{eqn:def-uPsi}
\selfKos : \Dmix(\BGUvee, \bk) = \Kb \ParityBSp(\BGUvee, \bk) \\
\xrightarrow{\Kb(\preselfKos)} \Kb \TiltBSp(\UGB, \bk) \xrightarrow[\sim]{\text{Lemma~\ref{lem:tilt-realization}}} \Dmix(\UGB, \bk).
\end{multline}

\subsection{Images of standard and costandard objects}
\label{ss:Psi-Delta-nabla}

In this subsection we assume that $\bk$ is a field.
Let $s \in S$, and consider the functor
\[
\Tmon_s \hatstar (-) : \TiltBSp(\UGB, \bk) \to \TiltBSp(\UGB, \bk).
\]
Conjugating the functor $\Kb(\Tmon_s \hatstar (-))$ by the equivalence of Lemma~\ref{lem:tilt-realization} we obtain a
triangulated functor
\[
C'_s : \Dmix(\UGB, \bk) \to \Dmix(\UGB, \bk).
\]
Of course the same construction can be done for the functor $\Tmon_1 \hatstar (-)$ (which is isomorphic to the identity functor). These constructions are functorial in the sense that any morphism from $\Tmon_s$ to any shift of $\Tmon_1$ induces a morphism of functor from $C'_s$ to the corresponding shift of the identity functor. In particular, using the morphism $\hat{\epsilon}_s$ defined in~\cite[\S 5.3.4]{amrw} we obtain a morphism of functors $\tilde{\epsilon}_s : C_s' \to \id \langle 1 \rangle$.

As in Section~\ref{sec:functor-V}, for $v \in W$ we denote by $\bbDelta_v \in \Dmix(\UGB, \bk)$, resp.~$\bbnabla_v \in \Dmix(\UGB, \bk)$ the standard, resp.~costandard, perverse sheaf associated to $v$.

\begin{lem}
\label{lem:epsilon-delta}
For any $v \in W$, the morphism $\tilde{\epsilon}_s(\bbDelta_v) : C_s'(\bbDelta_v) \to \bbDelta_v \langle 1 \rangle$ is nonzero.
\end{lem}

\begin{proof}
By~\cite[Lemma~4.9]{modrap2}, there exists an embedding $f_v : \cT_1 \langle -\ell(v) \rangle \hookrightarrow \bbDelta_v$. We deduce a commutative diagram
\[
\begin{tikzcd}[column sep=large]
C_s'(\cT_1 \langle -\ell(v) \rangle) \ar[d, "C'_s(f_v)" swap] \ar[rr, "\tilde{\epsilon}_s(\cT_1 \langle -\ell(v) \rangle)"] && \cT_1 \langle -\ell(v) +1 \rangle \ar[d, "f_v"] \\
C_s'(\bbDelta_v) \ar[rr] && \bbDelta_v \langle 1 \rangle,
\end{tikzcd}
\]
where the lower arrow is $\tilde{\epsilon}_s(\bbDelta_v)$.
By construction we have $C_s'(\cT_1) = \Tmon_s \hatstar \cT_1 \cong \ForFMLM(\Tmon_s) = \cT_s$, and it is easy to see that $\tilde{\epsilon}_s(\cT_1)$ identifies with the surjective morphism $\cT_s \to \cT_1 \langle 1 \rangle$. From this we deduce that the composition of the upper horizontal arrow with the right vertical arrow is nonzero, proving that $\tilde{\epsilon}_s(\bbDelta_v)$ is also nonzero.
\end{proof}

Now, recall the functor $C_s$ of~\S\ref{ss:perv}.

\begin{lem}
\label{lem:Cs'}
There exists an isomorphism of functors $C_s \simto C_s'$.
\end{lem}
\begin{proof}
Since $C_s$ and $C_s'$ have isomorphic restrictions to $\TiltBSp(\UGB,\bk)$, this follows from Proposition~\ref{prop:realization-functor}.
\end{proof}

For any $w \in W$ we can consider the standard and costandard (mixed) perverse sheaves $\Delta_w^\vee$, $\nabla_w^\vee$ in $\Dmix(\BGUvee, \bk)$ (constructed in~\cite{modrap2}), and also the objects $\bbDelta_w$, $\bbnabla_w$ in $\Dmix(\UGB, \bk)$ considered above. 
\begin{prop}
\label{prop:Psi-standards}
For any $w \in W$ we have
\[
\selfKos(\Delta_w^\vee) \cong \bbDelta_w, \quad \selfKos( \nabla_w^\vee) \cong \bbnabla_w.
\]
\end{prop}

\begin{proof}
We only prove the first isomorphism; the proof of the second one is similar. We proceed by induction on $w$, the case $w=1$ being clear by construction.

Let $w \in W$, and choose $s \in S$ such that $sw < w$. By the explicit description of $\Delta^\vee_s$ (see in particular~\cite[\S 10.4]{amrw}),
there exists a distinguished triangle
\[
\Delta^\vee_s \to \cE^\vee_s \to \Delta^\vee_1 \{1\} \xrightarrow{[1]},
\]
where the second morphism is the image of the ``upper dot'' morphism under~\eqref{eqn:equiv-Diag-Parity-const}. Convolving with $\Delta^\vee_{sw}$ on the right and using~\cite[Proposition~4.4]{modrap2} we deduce a distinguished triangle
\begin{equation}
\label{eqn:triangle-delta-E}
\Delta_w^\vee \to \cE^\vee_s \ustar \Delta^\vee_{sw} \to \Delta^\vee_{sw} \{1\} \xrightarrow{[1]},
\end{equation}
where the second morphism is the convolution of $\id_{\Delta_{sw}^\vee}$ with the image of the upper dot morphism. Since $\monKos$ is a monoidal functor, and by construction of the functor $\cE_s^\vee \ustar (-) : \Dmix(\BGUvee, \bk) \to \Dmix(\BGUvee, \bk)$, there exists a canonical isomorphism
\[
\selfKos \circ (\cE_s^\vee \ustar (-)) \cong C_s' \circ \selfKos.
\]
Using Lemma~\ref{lem:Cs'}, taking the image of~\eqref{eqn:triangle-delta-E} we obtain a distinguished triangle
\[
\selfKos(\Delta_w^\vee) \to C_s \circ \selfKos(\Delta_{sw}^\vee) \to \selfKos(\Delta_{sw}^\vee) \langle 1 \rangle \xrightarrow{[1]},
\]
where the second morphism is induced by the composition $C_s \simto C_s' \xrightarrow{\tilde{\epsilon}_s} \id \langle 1 \rangle$. Using induction, we can rewrite this triangle in the following form:
\begin{equation}
\label{eqn:triangle-Cs-delta-proof}
\selfKos(\Delta_w^\vee) \to C_s (\bbDelta_{sw}) \to \bbDelta_{sw} \langle 1 \rangle \xrightarrow{[1]}.
\end{equation}

It follows from Lemma~\ref{lem:epsilon-delta} that the morphism $C_s (\bbDelta_{sw}) \to \bbDelta_{sw} \langle 1 \rangle$ appearing in~\eqref{eqn:triangle-Cs-delta-proof} is nonzero. Since by adjunction we have
\[
\Hom_{\Dmix(\UGB, \bk)}(\bbDelta_w, \bbDelta_{sw} \langle 1 \rangle) = \Hom_{\Dmix(\UGB, \bk)}(\bbDelta_w [1], \bbDelta_{sw} \langle 1 \rangle) = 0,
\]
the first distinguished triangle in~\cite[Lemma~10.5.3(1)]{amrw} shows that the $\bk$-vector space $\Hom_{\Dmix(\UGB)}(C_s (\bbDelta_{sw}), \bbDelta_{sw} \langle 1 \rangle)$ is $1$-dimensional. Hence the second morphism in~\eqref{eqn:triangle-Cs-delta-proof} coincides (up to scalar) with the similar morphism in the first distinguished triangle in~\cite[Lemma~10.5.3(1)]{amrw}. Comparing these triangles we deduce an isomorphism $\selfKos(\Delta_w^\vee) \cong \bbDelta_w$, as desired.
\end{proof}

\subsection{Proof of Theorems~\ref{thm:monoidal-Koszul} and~\ref{thm:self-Koszul}}
\label{ss:proof-Phi-equiv}

We need only show that $\monKos$ and $\selfKos$ are equivalences of categories, as all the other assertions in these theorems are immediate from the definitions of these functors.  

\begin{proof}[Proof of Theorem~{\rm\ref{thm:self-Koszul}}]
It is enough to show that $\selfKos$ is fully faithful, as it is easy to see that full faithfulness implies that it is also essentially surjective.

Let us first treat the case where $\bk$ is a field. Observe that
\[
\Hom_{\Dmix(\BGUvee, \bk)}(\Delta^\vee_w, \nabla_v^\vee \langle m \rangle [n]) \cong \begin{cases}
\bk & \text{if $v=w$ and $n=m=0$} \\
0 & \text{otherwise}
\end{cases}
\]
and
\[
\Hom_{\Dmix(\UGB, \bk)}(\bbDelta_w, \bbnabla_v \langle m \rangle [n]) \cong \begin{cases}
\bk & \text{if $v=w$ and $n=m=0$} \\
0 & \text{otherwise.}
\end{cases}
\]
In view of this,
combining Proposition~\ref{prop:Psi-standards} with a classical
result sometimes called ``Be{\u\i}linson's lemma'' (see
e.g.~\cite[Lem\-ma~3.9.3]{abg}), to conclude it suffices to prove that the image under $\selfKos$ of any nonzero morphism $f : \Delta^\vee_w \to \nabla_w^\vee$ is nonzero. However the cone of $f$ is supported on $\overline{\cX^\vee_w} \smallsetminus \cX^\vee_w$, and then Proposition~\ref{prop:Psi-standards} implies that the cone of $\selfKos(f)$ is supported on $\overline{\cX_w} \smallsetminus \cX_w$. Therefore, $\selfKos(f) \neq 0$.

We next consider the case $\bk=\Z'$. Let $\uv$ and $\uw$ be expressions, let $m \in \Z$, and consider the morphism
\begin{equation}\label{eqn:phi-ff-check3}
\Hom_{\ParityBSp(\BGUvee,\Z')}(\cE_\uv, \cE_\uw \{m\}) \to \Hom_{\TiltBSp(\UGB, \Z')}(\cT_\uv, \cT_\uw \langle m \rangle)
\end{equation}
induced by $\selfKos$. Both sides are free $\Z'$-modules of finite rank, by~\cite[Lemma~2.2]{mr:etsps} and Proposition~\ref{prop:morph-Tilt-UGU},
respectively. To prove that~\eqref{eqn:phi-ff-check3} is an isomorphism, it is enough to show that it becomes an isomorphism after extension of scalars to any field $\bk$ admitting a ring homomorphism $\Z' \to \bk$.  That is, we must show that the left-hand vertical map in the commutative diagram below is an isomorphism.
\begin{equation}\label{eqn:phi-ff-diag}
\begin{tikzcd}
\bk \otimes_{\Z'} \Hom(\cE^{\Z'}_\uv, \cE^{\Z'}_\uw (m))
  \ar[r] \ar[d, "\bk \otimes \selfKos^{\Z'}"'] &
  \Hom(\cE^\bk_\uv, \cE^\bk_\uw (m))
  \ar[d, "\selfKos^\bk"] \\
\bk \otimes_{\Z'} \Hom(\cT^{\Z'}_\uv, \cT^{\Z'}_\uw \langle m \rangle)
  \ar[r] &
  \Hom(\cT^\bk_\uv, \cT^\bk_\uw \langle m \rangle)
\end{tikzcd}
\end{equation}
Here, the horizontal maps are isomorphisms, by~\cite[Lemma~2.2]{mr:etsps} and by
Proposition~\ref{prop:morph-Tilt-UGU},
respectively. The right-hand vertical map is an isomorphism by the case of field coefficients considered above.  This completes the proof for $\Z'$.

Finally, the case of general $\bk$ can be deduced from the case of $\Z'$ using another diagram like~\eqref{eqn:phi-ff-diag}.
\end{proof}

\begin{proof}[Proof of Theorem~{\rm \ref{thm:monoidal-Koszul}}]
From the definition of $\monKos$, we see that to show that it is an equivalence, we must show that $\Phi$ is an equivalence.  This latter functor is essentially surjective by construction, so it remains to show that it is fully faithful.

Let $M,N \in \DiagBS(\fh_\bk^*,W)$, 
and consider the map
\begin{multline}\label{eqn:phi-ff-check}
\bigoplus_{m \in \Z} \Hom_{\DiagBS(\fh_\bk^*,W)}(M,N(m)) \\ \to \bigoplus_{m \in \Z}\Hom_{\TiltBS(\UGU, \bk)}(\Phi(M), \Phi(N) \langle m \rangle).
\end{multline}
As right $R^\vee$-modules, both sides are free of finite rank, by~\cite[Corollary~6.13]{ew} and Proposition~\ref{prop:morph-Tilt-UGU}, respectively. By the graded Nakayama lemma, to prove that~\eqref{eqn:phi-ff-check} is an isomorphism, it is enough show that the induced map
\begin{multline}\label{eqn:phi-ff-check2}
\left( \bigoplus_{m \in \Z} \Hom_{\DiagBS(\fh_\bk^*,W)}(M,N(m)) \right) \otimes_{R^\vee} \bk \\
\to \left( \bigoplus_{m \in \Z} \Hom_{\TiltBS(\UGU)}(\Phi(M), \Phi(N) \langle m \rangle) \right) \otimes_{R^\vee} \bk
\end{multline}
is an isomorphism.  This new map is the one that arises when we apply $\underline{\Phi}$ to $M$ and $N$, regarded as objects of $\uDiagBSp(\fh_\bk^*,W)$.  Now,  Theorem~\ref{thm:self-Koszul} tells us that $\selfKos$ is an equivalence.  It follows that $\preselfKos$ is also an equivalence, as is $\underline{\Phi} \cong \preselfKos \circ \underline{\Psi}^\vee$.  We conclude that~\eqref{eqn:phi-ff-check2} and~\eqref{eqn:phi-ff-check} are isomorphisms.
\end{proof}

\subsection{Another formulation of Koszul duality}
\label{ss:Koszul-duality}

In this subsection we assume that $\bk$ is a field or a complete local ring.  We will study a variant of Theorem~\ref{thm:self-Koszul} involving $\Dmix(\UGBold,\bk)$ and $\Dmix(\BGUvee,\bk)$.  Both of these categories admit perverse t-structures as in~\cite{modrap2}.  As usual, we denote the standard and costandard objects by $\Delta_w$, $\nabla_w$, $\Delta_w^\vee$, $\nabla_w^\vee$, and the indecomposable parity complexes by $\cE_w$, $\cE^\vee_w$.  To distinguish the indecomposable tilting perverse sheaves from the corresponding objects in $\TiltBSp(\UGB,\bk)$, we denote them instead by the new symbols $\mathcal{S}_w$ and $\mathcal{S}_w^\vee$. 

The following theorem generalizes~\cite[Theorem~5.4]{modrap2} to the Kac--Moody case.

\begin{thm}[Self-duality]
\label{thm:Koszul-duality-objects}
Assume that $\bk$ is a field or a complete local ring.  There is an equivalence of triangulated categories
\[
\kappa : \Dmix(\UGBold, \bk) \simto \Dmix(\BGUvee, \bk)
\]
satisfying $\kappa \circ \langle 1 \rangle \cong \langle -1 \rangle [1] \circ \kappa$, and such that
\[
\kappa(\Delta_w) \cong \Delta^\vee_w, \quad \kappa(\nabla_w) \cong \nabla^\vee_w, \quad \kappa(\mathcal{S}_w) \cong \cE^\vee_w, \quad \kappa(\cE_w) \cong \mathcal{S}^\vee_w
\]
for any $w \in W$.
\end{thm}

\begin{proof}
We define $\kappa$ to be the inverse of the composition of equivalences
\[
\Dmix(\BGUvee,\bk) \xrightarrow[\sim]{\selfKos} \Dmix(\UGB,\bk) \xrightarrow[\sim]{\ForLMRE} \Dmix(\UGBold,\bk).
\]
It is immediate from the definition and Theorem~\ref{thm:self-Koszul} that $\kappa \circ \langle 1 \rangle \cong \langle -1 \rangle [1] \circ \kappa$, and that $\kappa(\mathcal{S}_w) \cong \cE^\vee_w$.  The calculation of $\kappa(\Delta_w)$ and $\kappa(\nabla_w)$ is identical to that in~\cite[Lemma~5.2]{modrap2}. (In the case of fields, this also follows directly from Proposition~\ref{prop:Psi-standards}.)

It remains to show that $\kappa(\cE_w) \cong \mathcal{S}^\vee_w$.  For $m,n \in \Z$ and $y,w \in W$ we have
\begin{multline*}
\Hom_{\Dmix(\BGUvee, \bk)}(\Delta^\vee_y, \kappa(\cE_w) \langle n \rangle [m]) \\
\cong \Hom_{\Dmix(\BGUvee, \bk)}(\kappa(\Delta_y), \kappa(\cE_w) \langle n \rangle [m]) \\
\cong \Hom_{\Dmix(\UGBold, \bk)}(\Delta_y, \cE_w \{n\}[m]),
\end{multline*}
so this space vanishes unless $m=0$ (by adjunction and~\cite[Remark~2.7]{modrap2}). Similar arguments show that
\[
\Hom_{\Dmix(\BGUvee, \bk)}(\kappa(\cE_w), \nabla^\vee_y \langle n \rangle [m])=0
\]
unless $m=0$. Together, these results imply that $\kappa(\cE_w)$ belongs to the heart of the perverse t-structure on $\Dmix(\BGUvee, \bk)$, and is a tilting object therein. Since $\kappa$ is an equivalence, this object is indecomposable, and then it is easy to see that it is isomorphic to $\mathcal{S}^\vee_w$.
\end{proof}

\begin{rmk}
Using Theorem~\ref{thm:Koszul-duality-objects} and the results of~\cite[Part~III]{rw}, one can express the ranks of the free $\bk$-modules $\Hom_{\Dmix(\UGBold, \bk)}(\Delta_y, \mathcal{S}_w \langle n \rangle)$ in terms of the $\ell$-canonical basis of a certain Hecke algebra in the sense of~\cite{jw}. (See Corollary~\ref{cor:characters-tilting-Gr} below for a more precise formulation of this property in a particular case.) This can be considered as a ``modular analogue'' of the results of~\cite{yun}.
\end{rmk}

\section{Parabolic--Whittaker Koszul duality}
\label{sec:parabolic-whittaker}

In this section we fix a subset $J \subset S$ of finite type. We denote by $W_J$ the subgroup of $W$ generated by $J$ (which is finite by assumption), by $w_0^J$ the longest element in $W_J$, and by $\JW \subset W$ the subset consisting of elements $w$ which are minimal in $W_J \cdot w$. Our goal is to prove a ``parabolic--Whittaker'' version of the equivalence of~\S\ref{ss:Koszul-duality} in the sense considered in~\cite{by}, with respect to the parabolic subgroups associated with $J$.

\subsection{Whittaker-type derived category}
\label{ss:Whit-derived-cat}

In this section we change our setting slightly, and consider the ``\'etale context'' of~\cite[\S 9.3]{rw}, as opposed to the ``classical context'' considered until now (and in~\cite{amrw}).

More precisely, we let $\F$ be an algebraically closed field of characteristic $p>0$. We redefine $\GKM$ to be the base change to $\F$ of the ind-group scheme $\GKM_\Z$ associated with our Kac--Moody root datum $(I, \bX, \{\alpha_i\}_{i \in I}, \{\alpha_i^\vee\}_{i \in I})$. We similarly now assume that $\BKM$ and $\TKM$ are defined over $\F$. We denote by $\UKM$ the pro-unipotent radical of $\BKM$, so that $\BKM = \TKM \ltimes \UKM$. Then $\GKM$ is an ind-group scheme over $\F$, and $\UKM$ and $\BKM$ are $\F$-group schemes (of infinite type).

We fix a prime number $\ell \neq p$, and assume that $\bk$ is either an algebraic closure of $\Ql$, or a finite extension of $\Ql$, or the ring of integers of such an extension, or a finite field of characteristic $\ell$. We also assume that there exists a ring morphism $\Z' \to \bk$. Then we can consider the \'etale $\BKM$-equivariant derived category $\Db(\BGB, \bk)$. (For a detailed treatment of the Bernstein--Lunts construction in the \'etale setting, see~\cite{weidner}.) All the categories constructed out of this in~\cite{amrw} make sense in this new setting, and for simplicity we will use the same notation.

To $J$ we can associate a subgroup scheme $\PKM_J$ of $\GKM$, as in~\cite[\S 9.1]{rw}. Following~\cite[\S 11.1]{rw} we denote by $\UKM^J$ the pro-unipotent radical of $\PKM_J$, 
and by $\LKM_J$ its Levi factor, 
which is a connected reductive $\F$-group. Finally, let $\UKM_J^-$ be the unipotent radical of the Borel subgroup of $\LKM_J$ which is opposite to $\BKM \cap \LKM_J$ (with respect to $\TKM$). Then the $\UKM^J \UKM_J^-$-orbits on $\GKM/\BKM$ are parametrized by $W$ (in the obvious way). We will denote the orbit parametrized by $w$ by $\cX_{w}^{\Whit,J}$, and its dimension by $d_w^J$.

For any $s \in J$ we have a root subgroup $\UKM_s^- \subset \UKM_J^-$. Moreover, the natural embedding induces an isomorphism of algebraic groups
\[
 \prod_{s \in J} \UKM_s^- \simto \UKM_J^- / [\UKM_J^-, \UKM_J^-].
\]
For each $s \in J$, choose, once and for all, an isomorphism $\Ga \simto \UKM_s^-$. We then obtain a morphism of algebraic groups
\[
 \UKM^J \UKM_J^- \to \UKM_J^- \to \UKM_J^- / [\UKM_J^-, \UKM_J^-] \simto \prod_{s \in J} \UKM_s^- \simto (\Ga)^J \xrightarrow{+} \Ga,
\]
which we will denote $\chi_J$.

Let us also fix a nontrivial additive character $\psi : \Z/p\Z \to \bk^\times$. (We assume that such a character exists.)
This determines a rank-one local system $\cL_\psi$ on $\Ga$, defined as the $\psi$-isotypic component in the direct image of the constant sheaf under the Artin--Schreier map $\Ga \to \Ga$ defined by $x \mapsto x^p-x$. Then $\cL_\psi$ is a multiplicative local system in the sense of~\cite[Appendix~A]{modrap1}, and hence so is $(\chi_J)^* \cL_\psi$. We will denote by
\[
 \Db_{\Whit,J}(\GKM/\BKM, \bk)
\]
the triangulated category of $(\UKM^J \UKM_J^-, \chi_J^* \cL_\psi)$-equivariant complexes on the ind-variety $\GKM/\BKM$ (see~\cite[Definition~A.1]{modrap1}). Note that if $w \in W$, then $\cX_w^{\Whit,J}$
supports a $(\UKM^J \UKM_J^-, \chi_J^* \cL_\psi)$-equivariant local system if and only if $w \in \JW$. In this case there exists a unique such local system of rank one (up to isomorphism), which we will denote by $\cL_w^J$.  We also have $d_w^J=\ell(w) + \ell(w_0^J)$.

\subsection{Whittaker-type parity complexes and mixed derived categories}

As observed in particular in~\cite[\S 11.1]{rw}, the notion of parity complexes from~\cite{jmw} makes sense in $\Db_{\Whit,J}(\GKM/\BKM, \bk)$; we will denote by $\Parity_{\Whit, J}(\GKM/\BKM, \bk)$ the corresponding full subcategory of $\Db_{\Whit,J}(\GKM/\BKM, \bk)$. The indecomposable objects in this category are parametrized by $\JW \times \Z$ in the standard way (see~\cite[Remark~11.6]{rw}); the object corresponding to $(w,0)$ will be denoted $\cE_w^J$.

Since we have the category $\Parity_{\Whit, J}(\GKM/\BKM, \bk)$, we can define the mixed derived category
\[
 \Dmix_{\Whit, J}(\GKM/\BKM, \bk) := \Kb \Parity_{\Whit, J}(\GKM/\BKM, \bk).
\]
The recollement formalism developed in~\cite[\S 2.4]{modrap2} also works in this setting (for the closed subvarieties consisting of a union of a finite number of $\UKM^J \UKM_J^-$-orbits and their open complements). Hence, for $w \in \JW$, if we denote by $i_w^J$ the embedding of the $\UKM^J \UKM_J^-$-orbit parametrized by $w$ in $\GKM/\BKM$,
we can define the standard and costandard objects
\[
 \Delta_{w,J}^{\Whit} := (i_w^J)_! \cL_w^J \{d_w^J\}, \qquad \nabla_{w,J}^{\Whit} := (i_w^J)_* \cL_w^J \{d_w^J\},
\]
as in~\cite[\S 2.5]{modrap2}. By~\cite[Lemma~3.2]{modrap2}, these objects satisfy
\[
 \Hom_{\Dmix_{\Whit, J}(\GKM/\BKM, \bk)} \bigl( \Delta_{v,J}^{\Whit}, \nabla_{w,J}^{\Whit} \langle m \rangle [n] \bigr) \cong
 \begin{cases}
   \bk & \text{if $v=w$ and $n=m=0$;} \\
   0 & \text{otherwise.}
 \end{cases}
\]

There exists a natural ``averaging'' triangulated functor from $\Db(\BGB, \bk)$ to $\Db_{\Whit, J}(\GKM/\BKM, \bk)$, 
defined as convolution on the left with the object $\Delta_{1,J}^\Whit$.
By~\cite[Corollary~11.5]{rw}, this functor sends parity complexes to parity complexes, and hence defines a functor
\[
 \Av^{\mathrm{eq}}_J : \Parity(\BGB, \bk) \to \Parity_{\Whit, J}(\GKM/\BKM, \bk).
\]
It is not difficult to check that for any $f \in \mathsf{H}^i_{\BKM}(\pt, \bk) = \Hom_{\Db(\BGB,\bk)}(\cE_1, \cE_1\{i\})$, if $i>0$ we have $\Av_J^{\mathrm{eq}}(f)=0$. From this it follows that $\Av^{\mathrm{eq}}_J$ factors through a functor
\[
 \Av_J : \Parity(\UGBold, \bk) \to \Parity_{\Whit, J}(\GKM/\BKM, \bk),
\]
where $\Parity(\UGBold, \bk)$ is the category of $\UKM$-equivariant parity complexes on $\GKM/\BKM$.

\begin{lem}
\label{lem:Av-Delta-nabla}
For any $w \in \JW$, we have
\[
\Av_J(\Delta_w) \cong \Delta_{w,J}^{\Whit}, \qquad \Av_J(\nabla_w) \cong \nabla_{w,J}^{\Whit}.
\]
\end{lem}

\begin{proof}
We only prove the first isomorphism; the proof of the second one is similar.

The category $\Parity_{\Whit, J}(\GKM/\BKM, \bk)$ admits a natural convolution action on the right by $\Parity(\BGB, \bk)$ (see in particular~\cite[Lemma~11.4]{rw}). We deduce an action of $\Dmix(\BGB, \bk)$ on $\Dmix_{\Whit,J}(\GKM/\BKM)$, which will be denoted $\ustar$.
By construction, the functor $\Av_J$ commutes with convolution on the right (see e.g.~\cite[(11.1)]{rw}); therefore we have
\[
\Av_J(\Delta_w) \cong \Av_J(\Delta_1 \ustar \Delta_w) \cong \Av_J(\Delta_1) \ustar \Delta_w.
\]
Now, by definition, we have $\Av_J(\Delta_1) \cong \Delta_{1,J}^{\Whit}$. Hence to conclude it suffices to prove that if $v \in \JW$ and $s \in S$ are such that $vs \in \JW$ and $\ell(vs) = \ell(v)+1$, we have
\[
\Delta_{v,J}^{\Whit} \ustar \Delta_s \cong \Delta_{vs,J}^\Whit.
\]

As in the proof of Proposition~\ref{prop:Psi-standards}
there exists a distinguished triangle
\[
\Delta_s \to \cE_s \to \cE_1\{1\} \xrightarrow{[1]},
\]
where the second map is induced by restriction along the closed embedding $\BKM/\BKM \hookrightarrow \PKM_s/\BKM$, where $\PKM_s$ is the minimal standard parabolic subgroup of $\GKM$ associated with $s$. Convolving on the left with $\Delta_{v,J}^\Whit$ we obtain a distinguished triangle
\begin{equation}
\label{eqn:triangle-Delta-Whit}
\Delta_{v,J}^{\Whit} \ustar \Delta_s \to \Delta_{v,J}^{\Whit} \ustar \cE_s
\to \Delta_{v,J}^{\Whit} \{1\} \xrightarrow{[1]}.
\end{equation}
Now it is not difficult to check that $\Delta_{v,J}^{\Whit} \ustar \cE_s$ is isomorphic to the $!$-pushforward of the shift by $\ell(v)+1$ of the unique rank-$1$ $(\UKM^J \UKM_J^-, \chi_J^* \cL_\psi)$-equivariant local system on $\cX_v^{\Whit,J} \sqcup \cX_{vs}^{\Whit,J}$, in such a way that the second map in~\eqref{eqn:triangle-Delta-Whit} is induced by the $*$-adjunction morphism associated with the closed embedding $\cX_v^{\Whit,J} \hookrightarrow \cX_v^{\Whit,J} \sqcup \cX_{vs}^{\Whit,J}$. The recollement formalism implies that the cocone of this morphism is $\Delta_{vs,J}^\Whit$, and the desired isomorphism follows.
\end{proof}

Let now $\langle \cE_w : w \notin \JW \rangle_{\oplus,\{1\}}$ be the full additive subcategory of the category $\Parity(\UGBold, \bk)$ whose objects are the direct sums of objects of the form $\cE_w \{n\}$ with $w \in W \smallsetminus \JW$ and $n \in \Z$. By~\cite[Lemma~11.7]{rw}, the functor $\Av_J$ vanishes on $\langle \cE_w : w \notin \JW \rangle_{\oplus,\{1\}}$, so it induces a functor
\[
 \Parity(\UGBold, \bk) / \langle \cE_w : w \notin \JW \rangle_{\oplus,\{1\}} \to \Parity_{\Whit, J}(\GKM/\BKM, \bk).
\]
(Here the quotient we consider is the ``naive'' quotient of additive categories, i.e.~the category whose $\Hom$-groups are the quotients of those in $\Parity(\UGBold, \bk)$ by the subgroup of morphisms which factor through an object of $\langle \cE_w : w \notin \JW \rangle_{\oplus,\{1\}}$.)

The following is a restatement of~\cite[Theorem~11.11]{rw}.

\begin{prop}
\label{prop:Av-Parity}
 The functor
 \[
 \Parity(\UGBold, \bk) / \langle \cE_w : w \notin \JW \rangle_{\oplus,\{1\}} \to \Parity_{\Whit, J}(\GKM/\BKM, \bk)
\]
induced by $\Av_J$ is an equivalence of categories.
\end{prop}

\subsection{Mixed tilting perverse sheaves on parabolic flag varieties}

In this subsection, $\bk$ is an arbitrary complete local ring.

We consider the Langlands dual Kac--Moody group $\GKM^\vee$ (still defined over $\mathbb{C}$), and its subgroups $\TKM^\vee$, $\BKM^\vee$ and $\UKM^\vee$ as in~\S\ref{ss:Phi-equiv}.
The choice of $J$ determines a parabolic subgroup $\PKM_J^\vee$ in $\GKM^\vee$, so that we can consider the $\BKM^\vee$-equivariant derived category of sheaves on the parabolic flag variety $\PKM_J^\vee \backslash \GKM^\vee$, which we will denote $\Db(\PGBvee, \bk)$, and then the mixed derived category $\Dmix(\PGUvee, \bk)$ constructed as the bounded homotopy category of the category of parity complexes. In this category we have standard and costandard objects parame\-trized by $\JW$ (see~\cite{modrap2}), which will be denoted by $\Delta_{w,J}^{\vee}$ and $\nabla_{w,J}^\vee$ respectively.

We denote by
\[
 \pi_J : \BKM^\vee \backslash \GKM^\vee \to \PKM_J^\vee \backslash \GKM^\vee
\]
the quotient map. The functor
\[
 (\pi_J)_* : \Db(\BGUvee, \bk) \to \Db(\PGUvee, \bk)
\]
sends parity complexes to parity complexes (see~\cite[\S 9.4]{rw}), so it induces a triangulated functor from $\Dmix(\BGUvee, \bk)$ to $\Dmix(\PGUvee, \bk)$, which will also be denoted $(\pi_J)_*$.
If $w \in W$ and if $v$ is the minimal element in $W_J \cdot w$, then by~\cite[Lemma~3.8]{modrap2} we have
\begin{equation}
\label{eqn:piJ-Delta-nabla}
  (\pi_J)_* \Delta_w^\vee \cong \Delta_{v,J}^\vee \{\ell(v)-\ell(w)\}, \qquad (\pi_J)_* \nabla_w^\vee \cong \nabla_{v,J}^\vee \{-\ell(v)+\ell(w)\}.
\end{equation}

Recall from~\S\ref{ss:Koszul-duality} that we have the subcategory
\[
\Tilt^\mix(\BGUvee, \bk) \subset \Dmix(\BGUvee, \bk)
\]
of tilting objects in the heart of the perverse t-structure, whose indecomposable objects are parametrized by $W \times \Z$, and that we denote by $\mathcal{S}^\vee_w$ the object corresponding to $(w,0)$. Similarly we have the category $\Tilt^\mix(\PGUvee, \bk)$ of tilting objects in the heart of the perverse t-structure on $\Dmix(\PGUvee, \bk)$, whose indecomposable objects are parametrized by $\JW \times \Z$; we will denote by $\mathcal{S}^\vee_{w,J}$ the object corresponding to $(w,0)$.

\begin{lem}
\label{lem:piJ-tilting}
 \begin{enumerate}
  \item 
  \label{it:piJ-tilting-1}
  The functor $(\pi_J)_*$ restricts to a functor
  \[
   \Tilt^\mix(\BGUvee, \bk) \to \Tilt^\mix(\PGUvee, \bk).
  \]
  \item
  \label{it:piJ-tilting-2}
  If $w \in W \smallsetminus \JW$, we have $(\pi_J)_* \mathcal{S}^\vee_w = 0$.
 \end{enumerate}
\end{lem}

\begin{proof}
 The proof is copied from~\cite{yun}. For any $v \in \JW$, we denote by $i_{v,J}^\vee$ the embedding of the $\BKM^\vee$-orbit on $\PKM_J^\vee \backslash \GKM^\vee$ parametrized by $v$.
 
 \eqref{it:piJ-tilting-1}
 Let $\cF$ be in $\Tilt^\mix(\BGUvee, \bk)$. Then $\cF$ belongs to the subcategory of $\Dmix(\BGUvee, \bk)$ generated under extensions by objects of the form $\Delta_w^\vee \langle n \rangle$ with $w \in W$ and $n \in \Z$. In view of~\eqref{eqn:piJ-Delta-nabla}, we deduce that $(\pi_J)_* \cF$ belongs to the subcategory of $\Dmix(\PGUvee, \bk)$ generated under extensions by objects of the form $\Delta_{v,J}^\vee \langle n \rangle [m]$ with $v \in \JW$, $n \in \Z$ and $m \in \Z_{\leq 0}$. This implies that $(i_{v,J}^\vee)^* (\pi_J)_* \cF$ belongs to the subcategory generated under extensions by objects of the form $\bigl( \underline{\bk}\{\ell(v)\} \bigr) \langle n \rangle [m]$ with $m \leq 0$.
 
On the other hand, $\cF$ belongs to the subcategory of $\Dmix(\BGUvee, \bk)$ generated under extensions by objects of the form $\nabla_w^\vee \langle n \rangle$ with $w \in W$ and $n \in \Z$. Using~\eqref{eqn:piJ-Delta-nabla} and the fact that the objects $\nabla^\vee_{v,J}$ are perverse (see~\cite[Theorem~4.7]{modrap2}), this implies that $(\pi_J)_* \cF$ lives in nonpositive perverse degrees, i.e.~that $(i_{v,J}^\vee)^* (\pi_J)_* \cF$ is in nonpositive perverse degrees for any $v$.

Combining these two properties, we obtain that for any $v \in W$ the object $(i_{v,J}^\vee)^* (\pi_J)_* \cF$ is a direct sum of objects of the form $\bigl( \underline{\bk}\{\ell(v)\} \bigr)\langle n \rangle$. Using Verdier duality we obtain the same property for $(i_{v,J}^\vee)^! (\pi_J)_* \cF$, which finally implies that $(\pi_J)_* \cF$ is a tilting perverse sheaf.
  
 \eqref{it:piJ-tilting-2}
 First we assume that $\bk=\Q$. In this setting the indecomposable parity complex $\cE_w$ on $\GKM/\BKM$ coincides with the intersection cohomology complex $\IC_w$ (see~\cite{kl, springer}). Using Theorem~\ref{thm:Koszul-duality-objects} and~\cite[Remark~2.7]{modrap2}, we deduce that if $u,v \in W$ and $v<u$ we have
 \[
  (\mathcal{S}_u^\vee : \Delta_v^\vee \langle n \rangle) = 0 \quad \text{unless $n>0$.}
 \]
 Then using~\eqref{eqn:piJ-Delta-nabla} we obtain that if $w \in W \smallsetminus \JW$ then
 \[
 \bigl( (\pi_J)_* \mathcal{S}_w^\vee :  \Delta_{v,J}^\vee \langle n \rangle \bigr) = 0 \quad \text{unless $n>0$.}
 \]
 If $(\pi_J)_* \mathcal{S}_w^\vee \neq 0$ and $v$ is maximal such that this multiplicity is nonzero, then this property contradicts the Verdier self-duality of $(\pi_J)_* \mathcal{S}_w^\vee$.

Now, for any expression $\uw$, and for any choice of coefficients $\bk'$, we denote by $\mathcal{S}_{\uw}^{\vee,\bk'}$ the image of the Bott--Samelson type tilting mixed perverse sheaf on $\BKM^\vee \backslash \GKM^\vee$ (constructed as in~\S\ref{ss:constructible-Dmix}, but for the Langlands dual group, and with the roles of left and right multiplication swapped) under the forgetful functor $\ForLMRE$. We claim that if $\uw$
starts with a simple reflection in $J$, we have $(\pi_J)_* \mathcal{S}_{\uw}^{\vee, \Z'}=0$. In fact we have
 \[
 \Q \bigl( (\pi_J)_* \mathcal{S}_{\uw}^{\vee, \Z'} \bigr) \cong (\pi_J)_* \mathcal{S}_{\uw}^{\vee, \Q},
 \]
 where $\Q$ is the natural extension-of-scalars functor.
 Now it is not difficult to see that $\mathcal{S}_{\uw}^{\vee, \Q}$ is a direct sum of objects of the form $\mathcal{S}_v^{\vee, \Q} \langle m \rangle$ with $m \in \Z$ and $v \in W \smallsetminus \JW$. (One can e.g.~use Koszul duality to translate the question to the setting of parity complexes, where it follows from equivariance considerations.) Hence, by the case of $\Q$ treated above, we have $\Q \bigl( (\pi_J)_* \mathcal{S}_{\uw}^{\vee, \Z'} \bigr)=0$. On the other hand it is not difficult to see that $(\pi_J)_* \mathcal{S}_{\uw}^{\vee, \Z'}$ has stalks that are free over $\Z'$. Hence these stalks are $0$, which implies that $(\pi_J)_* \mathcal{S}_{\uw}^{\vee, \Z'} =0$.
 
 Finally we prove the claim in general. For this we choose a reduced expression $\uw$ for $w$ starting with a simple reflection in $J$. Then $\mathcal{S}_w^{\vee, \bk}$ is a direct summand of $\mathcal{S}_{\uw}^{\vee, \bk}$. But
 \[
 (\pi_J)_* \mathcal{S}_{\uw}^{\vee, \bk} \cong \bk \bigl( (\pi_J)_* \mathcal{S}_{\uw}^{\vee, \Z'} \bigr) = 0,
 \]
 which proves the desired vanishing.
\end{proof}

Now let $\langle \mathcal{S}^\vee_w : w \notin \JW \rangle_{\oplus, \langle 1 \rangle}$ be the full additive subcategory of the category $\Tilt^\mix(\BGUvee, \bk)$ consisting of direct sums of objects of the form $\mathcal{S}^\vee_w \langle m \rangle$ with $w \in W \smallsetminus \JW$ and $m \in \Z$. Lemma~\ref{lem:piJ-tilting} tells us that the functor $(\pi_J)_*$ restricts to a functor
\[
(\pi_J)_*^\Tilt: \Tilt^\mix(\BGUvee, \bk) \to \Tilt^\mix(\PGUvee, \bk)
\]
that then factors through a functor
\begin{multline}
 \label{eqn:def-PiJ}
 \Pi_J : \Tilt^\mix(\BGUvee, \bk) / \langle \mathcal{S}^\vee_w : w \notin \JW \rangle_{\oplus,\langle 1 \rangle} \to \\
 \Tilt^\mix(\PGUvee, \bk).
\end{multline}

We denote by $\Perv^\mix(\BGUvee,\bk)$ and $\Perv^\mix(\PGUvee, \bk)$ the hearts of the perverse t-structures on $\Dmix(\BGUvee, \bk)$ and $\Dmix(\PGUvee, \bk)$ respectively.  The following statement uses the theory of realization functors from~\S\ref{ss:realization}.

\begin{lem}
 The following diagram commutes up to isomorphism:
 \[
  \begin{tikzcd}
 \Kb \Tilt^\mix(\BGUvee,\bk) \ar[d, "\Kb((\pi_J)_*^{\Tilt})" swap] \ar[rr, "\real"] && \Dmix(\BGUvee, \bk) \ar[d, "(\pi_J)_*"] \\
 \Kb \Tilt^\mix(\PGUvee,\bk) \ar[rr, "\real"] && \Dmix(\PGUvee, \bk).
\end{tikzcd}
 \]
\end{lem}
\begin{proof}
Since the functor $(\pi_J)_*$ is induced by a functor from $\Parity(\BGUvee, \bk)$ to $\Parity(\PGUvee,\bk)$, it lifts to a functor between the filtered versions of the categories on the right-hand side. The lemma then follows from Proposition~\ref{prop:realization-functor}.
\end{proof}

\subsection{Parabolic--Whittaker Koszul duality}
\label{ss:parabolic-Whittaker}

We come back to the assumptions of~\S\ref{ss:Whit-derived-cat}.  Recall the equivalence of categories $\kappa$ constructed in~\S\ref{ss:Koszul-duality}. 

\begin{thm}
\label{thm:duality-par-Whit}
 There exists an equivalence of triangulated categories $\kappa_J$ which fits into the following commutative diagram:
 \[
   \begin{tikzcd}[column sep=large]
 \Dmix(\UGBold, \bk) \ar[d, "\Av_J" swap] \ar[rr, "\kappa"] && \Dmix(\BGUvee, \bk) \ar[d, "(\pi_J)_*"] \\
 \Dmix_{\Whit, J}(\GKM/\BKM, \bk) \ar[rr, "\kappa_J"] && \Dmix(\PGUvee, \bk).
\end{tikzcd}
 \]
 Moreover, $\kappa_J$ satisfies
 \begin{equation}
 \label{eqn:kappaJ-objects}
 \kappa_J(\Delta_{w,J}^\Whit) \cong \Delta^\vee_{w,J}, \quad \kappa_J(\Delta_{w,J}^\Whit) \cong \Delta^\vee_{w,J}, \quad \kappa_J(\cE_w^J) \cong \mathcal{S}^\vee_{w,J}
 \end{equation}
 for any $w \in \JW$.
\end{thm}

\begin{proof}
 The equivalence of categories
 \[
  \Parity(\UGBold, \bk) \simto \Tilt^\mix(\BGUvee, \bk)
 \]
obtained by restricting $\kappa$ induces an equivalence of categories
\begin{multline*}
 \Parity(\UGBold, \bk) / \langle \cE_w : w \notin \JW \rangle_{\oplus, \{1\}} \\
 \simto \Tilt^\mix(\BGUvee, \bk) / \langle \mathcal{S}_w^\vee : w \notin \JW \rangle_{\oplus, \langle 1 \rangle}.
\end{multline*}
Using Proposition~\ref{prop:Av-Parity} we deduce an equivalence of categories
\[
  \Parity_{\Whit, J}(\GKM/\BKM, \bk) \simto \Tilt^\mix(\BGUvee, \bk) / \langle \mathcal{S}_w^\vee : w \notin \JW \rangle_{\oplus, \langle 1 \rangle}.
\]
We denote by
\[
 \kappa_J : \Dmix_{\Whit, J}(G/B, \bk) \to \Dmix(\PGUvee, \bk)
\]
the functor obtained by composing this equivalence with the functor $\Pi_J$ from~\eqref{eqn:def-PiJ}, and then passing to bounded homotopy categories. With this definition the diagram of the statement clearly commutes.

Now we prove that $\kappa_J$ is an equivalence of categories.
Using the commutativity of our diagram and comparing Lemma~\ref{lem:Av-Delta-nabla} and~\eqref{eqn:piJ-Delta-nabla} we see that for any $w \in \JW$ we have
\[
 \kappa_J(\Delta^{\Whit}_{w,J}) \cong \Delta^\vee_{w,J}, \quad \kappa_J(\nabla^{\Whit}_{w,J}) \cong \nabla^\vee_{w,J}.
\]
Moreover, since the functor $(\pi_J)_*$ induces an isomorphism
\[
\Hom_{\Dmix(\BGUvee, \bk)}(\Delta^\vee_{w}, \nabla^\vee_{w}) \simto \Hom_{\Dmix(\PGUvee, \bk)}(\Delta^\vee_{w,J}, \nabla^\vee_{w,J}),
\]
we see that $\kappa_J$ induces an isomorphism
\[
 \Hom_{\Dmix_{\Whit, J}(\GKM/\BKM, \bk)}(\Delta^{\Whit}_{w,J}, \nabla^{\Whit}_{w,J}) \simto \Hom_{\Dmix(\PGUvee, \bk)}(\Delta^\vee_{w,J}, \nabla^\vee_{w,J}).
\]
Then standard arguments (see e.g.~the proof of Theorem~\ref{thm:self-Koszul}) imply that $\kappa_J$ is an equivalence of categories.

Finally we prove the isomorphisms~\eqref{eqn:kappaJ-objects}. The first two isomorphisms have already been observed above. For the third isomorphism, recall that $\cE_w^J \cong \Av_J(\cE_w)$, see~\cite[Corollary~11.10]{rw}. It follows that $\kappa_J(\cE_w^J) \cong (\pi_J)_* \kappa(\cE_w) \cong (\pi_J)_* \mathcal{S}^\vee_w$, hence that this object is a tilting perverse sheaf by Lemma~\ref{lem:piJ-tilting}. Since $\kappa_J$ is an equivalence this object is indecomposable, and then it is easy to see that it is isomorphic to $\mathcal{S}_{w,J}^\vee$.
\end{proof}

\begin{rmk}
\begin{enumerate}
\item
The proof of Theorem~\ref{thm:duality-par-Whit} shows that the equivalence $\Phi$ induces an equivalence
\[
\Diag^{\mathrm{asph}, \bk}_J(\GKM) \simto \Tilt^\mix(\PGUvee, \bk)
\]
(where the left-hand side is defined in~\cite[\S 11.5]{rw}), and that for any $w \in \JW$ we have $(\pi_J)_*(\mathcal{S}^\vee_w) \cong \mathcal{S}^\vee_{w,J}$. (In the case of characteristic-$0$ coefficients, such an isomorphism follows from~\cite[Proposition~3.4.1]{yun}.)
\item
Using the same constructions as in~\cite{modrap2} one can endow the category $\Dmix_{\Whit,J}(\GKM/\BKM, \bk)$ with a perverse t-structure, whose heart is a graded highest weight category with standard, resp.~costandard, objects $\{\Delta^\Whit_{w,J} : w \in \JW\}$, resp.~$\{\nabla^\Whit_{w,J} : w \in \JW\}$. Then one can easily check that the images under $\kappa_J$ of the indecomposable tilting objects in this heart are the indecomposable parity complexes on $\PKM^\vee_J \backslash \GKM^\vee$, seen as objects in $\Dmix(\PGUvee, \bk)$.
\end{enumerate}
\end{rmk}

\section{Application to the tilting character formula}
\label{sec:application}

In this section we apply our preceding results together with those of~\cite{prinblock} to prove the character formula for tilting representations of reductive algebraic groups over fields of positive characteristic conjectured in~\cite{rw}.

\subsection{Koszul duality for affine flag varieties}
\label{ss:duality-affine}

Let $\Gp$ be a semisimple, simply connected complex algebraic group, let $\BGp$ be a Borel subgroup, and let $\TGp \subset \BGp$ be a maximal torus. We denote by $\Wf$ the Weyl group of $(\Gp, \TGp)$, and by $\mathfrak{R}$ its root system. Let also $\mathfrak{R}^+ \subset \mathfrak{R}$ be the system of positive roots consisting of the $\TGp$-weights in $\mathrm{Lie}(\Gp)/\mathrm{Lie}(\BGp)$, and let $\Sf \subset \Wf$ be the corresponding subset of simple reflections.

We set $\mathscr{K}:=\mathbb{C} ( \hspace{-1pt} ( z ) \hspace{-1pt} )$, $\mathscr{O}:=\mathbb{C} [ \hspace{-1pt} [ z ] \hspace{-1pt} ]$, and consider the group ind-scheme $\Gp(\mathscr{K})$. We denote by $\Iw$ the Iwahori subgroup of $\Gp(\mathscr{K})$ determined by $\BGp$, i.e.~the inverse image of $\BGp$ under the morphism $\Gp(\mathscr{O}) \to \Gp$ induced by the ring map $\mathscr{O} \to \mathbb{C}$ sending $z$ to $0$. Let also $\Iw^u$ be the pro-unipotent radical of $\Iw$, i.e.~the inverse image of the unipotent radical of $\BGp$ under the map $\Gp(\mathscr{O}) \to \Gp$ considered above. We define the affine flag variety $\Fl$ and its ``left variant'' $\Fl'$ as the quotients
\[
\Fl := \Gp(\mathscr{K}) / \Iw, \quad \Fl' := \Iw \backslash \Gp(\mathscr{K}).
\]

The ind-varieties $\Fl$ and $\Fl'$ have Bruhat decompositions (with respect to the natural action of $\Iw$) parametrized by the affine Weyl group
\[
W:=\Wf \ltimes X_*(\TGp), 
\]
and for any integral complete local 
ring $\bk$, 
we can consider the Bruhat-construc\-tible (or equivalently $\Iw^u$-equivariant) mixed derived categories $\Dmix_{(\Iw)}(\Fl, \bk)$ and $\Dmix_{(\Iw)}(\Fl', \bk)$, cf.~\cite[\S 10.7]{rw}.

For $w \in W$, we have standard objects $\Delta_w$ in $\Dmix_{(\Iw)}(\Fl, \bk)$ and $\Delta'_w$ in $\Dmix_{(\Iw)}(\Fl', \bk)$, costandard objects $\nabla_w$ in $\Dmix_{(\Iw)}(\Fl, \bk)$ and $\nabla'_w$ in $\Dmix_{(\Iw)}(\Fl', \bk)$, indecomposable parity complexes $\cE_w$ in $\Dmix_{(\Iw)}(\Fl, \bk)$ and $\cE_w'$ in $\Dmix_{(\Iw)}(\Fl', \bk)$, and indecomposable mixed tilting perverse sheaves $\mathcal{S}_w$ in $\Dmix_{(\Iw)}(\Fl, \bk)$ and $\mathcal{S}_w'$ in $\Dmix_{(\Iw)}(\Fl', \bk)$.

Let $S \subset W$ be the set of simple reflections (chosen as the reflections along the walls of the fundamental alcove $\{\lambda \in \mathbb{R} \otimes_{\Z} X_*(\TGp) \mid \forall \alpha \in \mathfrak{R}, \, 0 \leq \langle \lambda, \alpha \rangle \leq 1 \}$). We consider the realization $\fh = (V, \{\alpha_s^\vee\}, \{\alpha_s\})$ of $(W,S)$ over $\bk$ defined as follows:
\begin{enumerate}
 \item
 the underlying free $\bk$-module is $V = \bk \otimes_{\Z} X_*(\TGp)$;
 \item
 if $s \in \Sf$, $\alpha_s$ is the image in $V^* \cong \bk \otimes_{\Z} X^*(\TGp)$ of the simple root associated with $s$, and $\alpha_s^\vee$ is the image in $V$ of the simple coroot associated with $s$;
 \item
 if $s \in S \smallsetminus \Sf$, let $\gamma$ be the unique positive root such that the image of $s$ in $\Wf \cong W / X_*(\TGp)$ is $s_\gamma$; then $\alpha_s$ is the image of $-\gamma$ in $V^*$ and $\alpha_s^\vee$ is the image of $-\gamma^\vee$ in $V$.
\end{enumerate}

\begin{lem}
\label{lem:cartan-dual}
Assume that $2$ and all the prime numbers which are not very good for $\Gp$ are invertible in $\bk$. Then there is a $\Wf$-equivariant isomorphism $\varphi: V \to V^*$ such that for each $s \in S$, there is a scalar $b_s \in \bk^\times$ such that $\varphi(\alpha_s^\vee) = b_s \alpha_s$.
\end{lem}

\begin{proof}
Write $\alpha_1, \ldots, \alpha_r$ for the simple roots of $\Gp$, and $\alpha_1^\vee, \ldots, \alpha_r^\vee$ for its simple coroots.
Let $A_{\mathrm{f}}=( \alpha_j(\alpha_i^\vee) \rangle)_{i,j=1, \ldots, r}$ be the Cartan matrix for $\Gp$.  Our assumptions imply that $A_{\mathrm{f}}$ is invertible over $\bk$.  Let $D=\mathrm{diag}(\epsilon_1, \ldots, \epsilon_r)$ be the minimal matrix such that $D^{-1} A$ is symmetric, in the sense of~\cite[Definition~1.5.1]{kumar}. Then $\epsilon_1, \ldots, \epsilon_r$ are invertible in $\bk$. The images of the simple coroots span $V$, so that we can define a symmetric perfect pairing on $V$ by setting
\[
\langle 1 \otimes \alpha_i^\vee, 1 \otimes \alpha_j^\vee \rangle = \alpha_i(\alpha_j^\vee) \epsilon_i = \alpha_j(\alpha_i^\vee) \epsilon_j.
\]
The proof of~\cite[Proposition~1.5.2]{kumar} shows that this pairing is $\Wf$-equivariant, so that the induced isomorphism $V \simto V^*$ is $\Wf$-equivariant as well. The fact that $\varphi(\alpha_s^\vee) \in \bk^\times \cdot \alpha_s$ follows from the $\Wf$-equivariance.
\end{proof}

\begin{thm}
\label{thm:duality-affine}
Assume that $2$ and all the prime numbers which are not very good for $\Gp$ are invertible in $\bk$. Then there exists an equivalence of triangulated categories
\[
\kappa : \Dmix_{(\Iw)}(\Fl, \bk) \simto \Dmix_{(\Iw)}(\Fl', \bk)
\]
which satisfies $\kappa \circ \langle 1 \rangle \cong \langle -1 \rangle [1]$ and, for any $w \in W$,
\begin{align*}
\kappa(\Delta_w) \cong \Delta'_w, \quad & \quad \kappa(\nabla_w) \cong \nabla'_w, \\
\kappa(\cE_w) \cong \mathcal{S}'_w, \quad & \quad \kappa(\mathcal{S}_w) \cong \cE'_w.
\end{align*}
\end{thm}

\begin{proof}
For brevity, we write $\DiagBSp$ instead of $\DiagBSp(\fh,W)$ for the additive envelope of the  Elias--Williamson diagrammatic category associated to the realization $\fh$ of $W$.
By~\cite[Theorem~10.16]{rw},
there exists a canonical equivalence of additive monoi\-dal categories
\begin{equation}
\label{eqn:equiv-Diag-Par-aff}
 \DiagBSp \simto \ParityBSp(\Iw \backslash \Gp(\mathscr{K}) / \Iw, \bk),
\end{equation}
where the right-hand side denotes the category of direct sums of Bott--Samelson type $\Iw$-equivariant parity complexes on $\Fl$.
Using this as a starting point, one can run the same constructions as in~\cite{amrw} to construct a category
\[
 \TiltBSp(\Iw^u \fatbslash \Gp(\mathscr{K}) \fatslash \Iw^u, \bk)
\]
of Bott--Samelson type free-monodromic tilting perverse sheaves, and then the same constructions as in Section~\ref{sec:Koszul-duality} provide
an equivalence of categories
\begin{equation}
\label{eqn:equiv-D-Tilt-aff}
 {}' \hspace{-1pt} \DiagBSp \simto \TiltBSp(\Iw^u \fatbslash \Gp(\mathscr{K}) \fatslash \Iw^u, \bk),
\end{equation}
where ${}' \hspace{-1pt} \DiagBSp$ is the additive envelope of the Elias--Williamson diagrammatic category associated to the realization of $(W,S)$ with underlying free $\bk$-module $V^*$, roots $\{\alpha_s^\vee : s \in S\}$ and coroots $\{\alpha_s^\vee : s \in S\}$.

Let $\uDiagBSp$, resp.~${}' \hspace{-1pt} \uDiagBSp$, denote the category obtained from $\DiagBSp$, resp.~${}' \hspace{-1pt} \DiagBSp$, by taking quotients of morphism spaces by the morphisms of the form $f \cdot \lambda$ for $\lambda \in V^*$, resp.~of the form $f \cdot h$ for $h \in V$. Then the equivalence~\eqref{eqn:equiv-Diag-Par-aff} induces an equivalence of categories
\begin{equation}
\label{eqn:equiv-Diag-Par-aff-2}
 \uDiagBSp \simto \ParityBSp(\Iw \backslash \Gp(\mathscr{K}) / \Iw^u, \bk),
\end{equation}
and the equivalence~\eqref{eqn:equiv-D-Tilt-aff} induces an equivalence
\begin{equation}
\label{eqn:equiv-D-Tilt-aff-2}
 {}' \hspace{-1pt} \uDiagBSp \simto \TiltBSp(\Iw^u \fatbslash \Gp(\mathscr{K}) / \Iw, \bk).
\end{equation}


As usual, let $R = \Sym(V^*)$ and $R^\vee = \Sym(V)$, and then let $\imath : R \simto R^\vee$ be the isomorphism induced by the isomorphism $\varphi$ of Lemma~\ref{lem:cartan-dual}.  We can define an equivalence of categories $\DiagBSp \simto {}' \hspace{-1pt} \DiagBSp$ which is the identity on objects, and which is induced on morphisms by the assignment
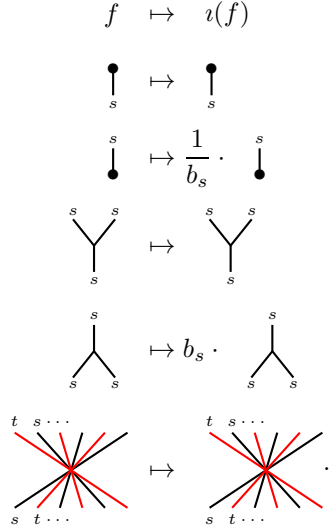
\begin{align*}
  \begin{array}{c}
    \begin{tikzpicture}[thick,scale=0.07,baseline]
      \node at (0,0) {$f$};
    \end{tikzpicture}
  \end{array} &\mapsto   \begin{array}{c}
    \begin{tikzpicture}[thick,scale=0.07,baseline]
      \node at (0,0) {$\imath(f)$};
    \end{tikzpicture}
  \end{array} \\
    \begin{array}{c}
    \begin{tikzpicture}[thick,scale=0.07,baseline]
      \draw (0,-5) to (0,0);
      \node at (0,0) {$\bullet$};
      \node at (0,-6.7) {\tiny $s$};
    \end{tikzpicture}
  \end{array} &\mapsto
    \begin{array}{c}
    \begin{tikzpicture}[thick,scale=0.07,baseline]
      \draw (0,-5) to (0,0);
      \node at (0,0) {$\bullet$};
      \node at (0,-6.7) {\tiny $s$};
    \end{tikzpicture}
  \end{array} \\
    \begin{array}{c}
    \begin{tikzpicture}[thick,baseline,xscale=0.07,yscale=-0.07]
      \draw (0,-5) to (0,0);
      \node at (0,0) {$\bullet$};
      \node at (0,-6.7) {\tiny $s$};
    \end{tikzpicture}
  \end{array} &\mapsto
    \frac{1}{b_s} \cdot \begin{array}{c}
    \begin{tikzpicture}[thick,baseline,xscale=0.07,yscale=-0.07]
      \draw (0,-5) to (0,0);
      \node at (0,0) {$\bullet$};
      \node at (0,-6.7) {\tiny $s$};
    \end{tikzpicture}
  \end{array} \\
    \begin{array}{c}
    \begin{tikzpicture}[thick,baseline,scale=0.07]
      \draw (-4,5) to (0,0) to (4,5);
      \draw (0,-5) to (0,0);
      \node at (0,-6.7) {\tiny $s$};
      \node at (-4,6.4) {\tiny $s$};
            \node at (4,6.4) {\tiny $s$};
    \end{tikzpicture}
  \end{array} &\mapsto
      \begin{array}{c}
    \begin{tikzpicture}[thick,baseline,scale=0.07]
      \draw (-4,5) to (0,0) to (4,5);
      \draw (0,-5) to (0,0);
      \node at (0,-6.7) {\tiny $s$};
      \node at (-4,6.4) {\tiny $s$};
            \node at (4,6.4) {\tiny $s$};
    \end{tikzpicture}
  \end{array} \\
    \begin{array}{c}
    \begin{tikzpicture}[thick,baseline,scale=-0.07]
      \draw (-4,5) to (0,0) to (4,5);
      \draw (0,-5) to (0,0);
      \node at (0,-6.7) {\tiny $s$};
            \node at (-4,6.4) {\tiny $s$};
            \node at (4,6.4) {\tiny $s$};
    \end{tikzpicture}
  \end{array} &\mapsto
   b_s \cdot \begin{array}{c}
    \begin{tikzpicture}[thick,baseline,scale=-0.07]
      \draw (-4,5) to (0,0) to (4,5);
      \draw (0,-5) to (0,0);
      \node at (0,-6.7) {\tiny $s$};
            \node at (-4,6.4) {\tiny $s$};
            \node at (4,6.4) {\tiny $s$};
    \end{tikzpicture}
  \end{array} \\
  \begin{array}{c}
  \begin{tikzpicture}[yscale=0.5,xscale=0.3,baseline,thick]
\draw (-2.5,-1) to (0,0) to (-1.5,1);
\draw (-0.5,-1) to (0,0) to (0.5,1);
\draw (1.5,-1) to (0,0) to (2.5,1);
\draw[red] (-1.5,-1) to (0,0) to (-2.5,1);
\draw[red] (0.5,-1) to (0,0) to (-0.5,1);
\draw[red] (2.5,-1) to (0,0) to (1.5,1);
\node at (-2.5,-1.3) {\tiny $s$\vphantom{$t$}};
\node at (-1.5,1.3) {\tiny $s$\vphantom{$t$}};
\node at (-0.5,-1.3) {\tiny $\cdots$\vphantom{$t$}};
\node at (-1.5,-1.3) {\tiny $t$};
\node at (-2.5,1.3) {\tiny $t$};
\node at (-0.5,1.3) {\tiny $\cdots$\vphantom{$t$}};
    \end{tikzpicture}
\end{array} &\mapsto
  \begin{array}{c}
  \begin{tikzpicture}[yscale=0.5,xscale=0.3,baseline,thick]
\draw (-2.5,-1) to (0,0) to (-1.5,1);
\draw (-0.5,-1) to (0,0) to (0.5,1);
\draw (1.5,-1) to (0,0) to (2.5,1);
\draw[red] (-1.5,-1) to (0,0) to (-2.5,1);
\draw[red] (0.5,-1) to (0,0) to (-0.5,1);
\draw[red] (2.5,-1) to (0,0) to (1.5,1);
\node at (-2.5,-1.3) {\tiny $s$\vphantom{$t$}};
\node at (-1.5,1.3) {\tiny $s$\vphantom{$t$}};
\node at (-0.5,-1.3) {\tiny $\cdots$\vphantom{$t$}};
\node at (-1.5,-1.3) {\tiny $t$};
\node at (-2.5,1.3) {\tiny $t$};
\node at (-0.5,1.3) {\tiny $\cdots$\vphantom{$t$}};
    \end{tikzpicture} .
\end{array}
\end{align*}
(In fact, the only thing one has to check is that this assignment defines a functor, which can be checked by hand using the defining relations.)

Composing the induced equivalence $\uDiagBSp \simto {}' \hspace{-1pt} \uDiagBSp$ with~\eqref{eqn:equiv-D-Tilt-aff-2} we obtain an equivalence of categories 
\[
 \uDiagBSp \simto \TiltBSp(\Iw^u \fatbslash \Gp(\mathscr{K}) / \Iw, \bk).
\]
Comparing with~\eqref{eqn:equiv-Diag-Par-aff-2},
passing to bounded homotopy categories to then composing with the appropriate forgetful functor we deduce the desired equivalence $\kappa$. The fact that $\kappa$ has the stated properties follows from the same arguments as for Theorem~\ref{thm:Koszul-duality-objects}.
\end{proof}

\begin{rmk}
\begin{enumerate}
\item
 It should be clear from the proof of Theorem~\ref{thm:duality-affine} that a similar claim holds
  in the equivariant/free-monodromic setting. We leave this variant to the reader.
\item
The same arguments as in the proof of Theorem~\ref{thm:duality-affine} show that, in the setting of Section~\ref{sec:Koszul-duality}, if $\GKM$ is symmetrizable then the equivalence of Theorem~\ref{thm:Koszul-duality-objects} can be seen as an equivalence
\[
\Dmix(\BKM \backslash \GKM / \UKM, \bk) \simto \Dmix(\UKM \backslash \GKM / \BKM, \bk)
\]
provided a certain finite set of prime numbers depending on $\GKM$ is invertible in $\bk$. (We leave it to the interested reader to make this statement precise.)
\end{enumerate}
\end{rmk}

\subsection{Koszul duality for affine Grassmannians}
\label{ss:duality-Gr}

The parabolic--Whittaker duality of~\S\ref{ss:parabolic-Whittaker} can
also be stated in the present ``affine'' setting. For simplicity we
restrict to the case of the (left variant of the) affine Grassmannian
\[
 \Gr' := \Gp(\mathscr{O}) \backslash \Gp(\mathscr{K}).
\]
The $\Iw$-orbits on this ind-variety are parametrized in a natural way by the subset $\fW \subset W$ consisting of elements $w$ which are minimal in $\Wf \cdot w$.
If $\bk$ is an integral complete local ring,
we denote by $\Dmix_{(\Iw)}(\Gr', \bk)$ the corresponding mixed derived category. For $w \in \fW$, we have a corresponding standard object $\Delta_w^{\Gr'}$, costandard object $\nabla_w^{\Gr'}$, indecomposable parity complex $\cE_w^{\Gr'}$, and indecomposable tilting perverse sheaf $\mathcal{S}_w^{\Gr'}$ in $\Dmix_{(\Iw)}(\Gr', \bk)$.

Now we assume that $\F$ and $\bk$ are as in~\S\ref{ss:Whit-derived-cat}.
We denote by $\TGp_\F$ the $\F$-torus whose lattice of characters is $X^*(\TGp)$, and let $\Gp_\F$ be the semisimple, simply-connected algebraic $\F$-group with maximal torus $\TGp_\F$ and root system $\mathfrak{R}$. Then we can define the Iwahori subgroups $\Iw_\F$ and $\Iw_\F^\circ$ of $\Gp_\F(\F ( \hspace{-1pt} ( z ) \hspace{-1pt} ))$ associated with the Borel subgroups of $\Gp_\F$ containing $\TGp_\F$ with roots $-\mathfrak{R}^+$ and $\mathfrak{R}^+$ respectively.

We redefine the affine flag variety $\Fl$ as the quotient $\Gp_\F(\F ( \hspace{-1pt} ( z ) \hspace{-1pt} )) / \Iw_\F$, an ind-variety over $\F$. Choosing identifications between $\F$ and each root subgroup of $G_\F$ associated with a simple root, as in~\S\ref{ss:Whit-derived-cat} we obtain an algebraic group morphism
\[
 \chi : \Iw_\F^\circ \to \Ga.
\]
Choosing also a nontrivial additive character
$\psi : \Z/p \Z \to \bk^\times$ (assumed to exist), with corresponding Artin--Schreier local system $\cL_\psi$, we can consider the category
\[
 \Dmix_{\mathcal{IW}}(\Fl, \bk)
\]
of $(\Iw_\F^\circ, \chi^*(\cL_\psi))$-equivariant mixed complexes. (Here ``$\mathcal{IW}$'' stands for ``Iwahori--Whittaker''; this terminology is taken from~\cite{ab}.) The $\Iw_\F^\circ$-orbits supporting an equivariant local system are labelled in a natural way by $\fW$. For $w \in \fW$, we have a corresponding standard object $\Delta_w^\IW$, costandard object $\nabla_w^\IW$, and indecomposable parity complex $\cE_w^\IW$
 in $\Dmix_{\mathcal{IW}}(\Fl, \bk)$.

The proof of the following theorem is identical to that of Theorem~\ref{thm:duality-par-Whit}.

\begin{thm}
Assume that $2$ and the prime numbers which are not very good for $\Gp$ are invertible in~$\bk$. Then
there exists an equivalence of triangulated categories
 \[
  \kappa_{\Gr'} : \Dmix_{\mathcal{IW}}(\Fl, \bk) \simto \Dmix_{(\Iw)}(\Gr', \bk)
 \]
which satisfies $\kappa \circ \langle 1 \rangle \cong \langle -1 \rangle [1]$ and, for any $w \in \fW$,
\[
\kappa(\Delta_w^\IW) \cong \Delta^{\Gr'}_w, \qquad \kappa(\nabla_w^\IW) \cong \nabla^{\Gr'}_w, \qquad
\kappa(\cE_w^\IW) \cong \mathcal{S}^{\Gr'}_w.
\]
\end{thm}

\subsection{Character formula for tilting mixed perverse sheaves on \texorpdfstring{$\Gr'$}{Gr'}}

We now assume that $\bk$ is a field (which does not necessarily satisfy the conditions of~\S\ref{ss:Whit-derived-cat}). We denote its characteristic by $\ell$, and assume that $\ell$ is odd and very good for $G$. We let
\[
 \{ \ppn_{y,w} : y,w \in \fW \}
\]
be the antispherical $\ell$-Kazhdan--Lusztig polynomials as considered in~\cite[\S 1.4]{rw}.
The following corollary was our main motivation to develop the ``parabolic--Whitta\-ker'' formalism of Section~\ref{sec:parabolic-whittaker}.

\begin{cor}
\label{cor:characters-tilting-Gr}
 For any $w,y \in \fW$ we have
 \[
  \ppn_{y,w}(v) = \sum_{i \in \Z} \bigl( \mathcal{S}^{\Gr'}_w : \nabla^{\Gr'}_y \langle -i \rangle \bigr) \cdot v^i.
 \]
\end{cor}

\begin{proof}
The same arguments as in~\cite[Lemma~3.8]{williamson} show that if $\bk \to \bk'$ is a field extension, then the extension-of-scalars functor $\Dmix_{(\Iw)}(\Gr', \bk) \to \Dmix_{(\Iw)}(\Gr', \bk')$ sends the indecomposable tilting perverse sheaf labelled by $w$ with coefficients $\bk$ to its counterpart for coefficients $\bk'$. Therefore, we can assume that $\bk$ satisfies the conditions of~\S\ref{ss:Whit-derived-cat}.
Then by definition and~\cite[Theorem~11.11]{rw}, if $\F$ is as above, we have
\[
 \ppn_{y,w}(v) = \sum_{i \in \Z} \dim \Hom_{\Dmix_{\mathcal{IW}}(\Fl, \bk)}(\Delta^\IW_y, \cE^\IW_w \{i\}) \cdot v^i.
\]
Using the equivalence $\kappa_{\Gr'}$ we deduce that
\[
 \ppn_{y,w}(v) = \sum_{i \in \Z} \dim \Hom_{\Dmix_{(\Iw)}(\Gr', \bk)}(\Delta^{\Gr'}_y, \mathcal{S}^{\Gr'}_w \langle i\rangle) \cdot v^i.
\]
The claim follows.
\end{proof}

\subsection{Tilting character formula}

From now on we assume that $\bk$ is an algebraically closed field of characteristic $\ell>0$, and let
$\bG$ be a connected reductive group over $\bk$ with simply-connected derived subgroup. We let $h$ be the Coxeter number of $\bG$, and assume that $\ell>h$.

We choose a Borel subgroup $\bB \subset \bG$ and a maximal torus $\bT \subset \bB$. We let $\mathfrak{S}$ be the root system of $(\bG, \bT)$, and $\mathfrak{S}^+ \subset \mathfrak{S}$ be the system of positive roots consisting of the $\bT$-weights in $\mathrm{Lie}(\bG) / \mathrm{Lie}(\bB)$. We also set $\bbX := X^*(\bT)$, and denote by $\bbX^+ \subset \bbX$ the subset of dominant weights.

For any $\lambda \in \bbX^+$, we denote by $\nabla(\lambda)$, resp.~$\Delta(\lambda)$, resp.~$\mathsf{T}(\lambda)$, the induced, resp.~Weyl, resp.~indecomposable tilting, $\bG$-module of highest weight $\lambda$.

We also denote by $\TGp$ the complex torus with weights $\Hom_{\Z}(\Z \mathfrak{S}, \Z)$, and let $\Gp$ be the semisimple, simply-connected complex algebraic group with maximal torus $\TGp$ and coroot system $\mathfrak{S}$. We have an associated affine Weyl group $W$ as in~\S\ref{ss:duality-affine} (which identifies with the semi-direct product $\Wf \ltimes \Z\mathfrak{S}$), and antispherical $\ell$-Kazhdan--Lusztig polynomials $\ppn_{y,w}$ as in~\S\ref{ss:duality-Gr}.

Let $\rho=\frac{1}{2} \sum_{\alpha \in \mathfrak{S}^+} \alpha$; then we can consider the ``dot-action'' of $W$ on $\bbX$ defined by
\[
 (w t_\lambda) \cdot_p \mu = w(\mu + p\lambda + \rho)-\rho
\]
for $w \in \Wf$ and $\lambda \in \Z\mathfrak{S}$.
The following result proves the ``combinatorial'' part of the main conjecture from~\cite{rw}.

\begin{thm}
\label{thm:char-formula-tiltings}
 For any $w,y \in \fW$ we have
 \[
  \bigl( \mathsf{T}(w \cdot_p 0) : \nabla(y \cdot_p 0) \bigr) = \ppn_{y,w}(1).
 \]
\end{thm}

\begin{proof}
 It follows from~\cite[Theorem~11.7]{prinblock} that we have
 \[
 \bigl( \mathsf{T}(w \cdot_p 0) : \nabla(y \cdot_p 0) \bigr) = \sum_{i \in \Z} \bigl( \mathcal{S}^{\Gr'}_w : \nabla^{\Gr'}_y \langle i \rangle \bigr).
 \]
(See~\cite[Remark~11.3(2)]{prinblock} for the comparison between our present conventions and those of~\cite{prinblock}.) Then the desired formula follows from Corollary~\ref{cor:characters-tilting-Gr}.
\end{proof}


\end{document}